\documentclass[DIV=10,10pt]{scrartcl}

\usepackage{lmodern}
\usepackage[T1]{fontenc}
\usepackage[english]{babel,csquotes}

\usepackage[backend=biber,style=alphabetic,eprint=false,url=false,giveninits=true,sorting=anyt,date=year,maxbibnames=6]{biblatex}
\usepackage{amsmath,amssymb,amsthm}
\usepackage{mathtools,thmtools,esint, stackrel}
\usepackage{mathrsfs}

\usepackage[pdfusetitle,colorlinks]{hyperref}
\usepackage[nameinlink]{cleveref}
\usepackage{xurl}

\usepackage[svgnames]{xcolor}
\usepackage{graphicx}
\usepackage[disable]{todonotes}
\usepackage{longtable,tabu}
\usepackage[shortlabels]{enumitem}
\usepackage{bm} 
\renewcommand{\textbf}[1]{\begingroup\bfseries\mathversion{bold}#1\endgroup}

\newcommand{\todopap}[2][]{\todo[author=Paul, color=LightBlue, #1]{#2}}

\newcommand{\rev}[1]{{#1}}

\newcommand{\N}{\mathbb{N}}
\newcommand{\Z}{\mathbb{Z}}

\newcommand{\R}{\mathbb{R}}

\newcommand{\measspace}{\mathscr{M}}
\newcommand{\probspace}{\mathscr{P}}
\newcommand{\contspace}{\mathscr{C}}

\newcommand{\mres}{\mathbin{\vrule height 1.6ex depth 0pt width
0.13ex\vrule height 0.13ex depth 0pt width 1.3ex}}
\newcommand{\eps}{\varepsilon}
\newcommand{\hdm}{{\mathscr H}}
\newcommand{\lbm}{{\mathscr L}}
\newcommand{\mass}{\mathbb{M}}
\newcommand{\ener}[1]{{\mathcal{#1}}}
\newcommand{\one}{\mathbf{1}}
\newcommand{\dd}{\mathop{}\mathopen{}\mathrm{d}}
\newcommand{\ac}{\mathrm{ac}}
\newcommand{\sing}{\mathrm{s}}
\newcommand{\BOT}{\mathsf{BOT}}

\DeclarePairedDelimiter\abs{\lvert}{\rvert}
\DeclarePairedDelimiter\norm{\lVert}{\rVert}
\DeclarePairedDelimiterX\cat[2]{(}{)}{#1:#2}

\DeclarePairedDelimiterX\set[1]\lbrace\rbrace{\def\given{\;\delimsize\vert\;}#1}

\DeclareMathOperator{\spt}{spt}

\DeclareMathOperator{\diam}{diam}

\DeclareMathOperator{\essinf}{ess\,inf}
\DeclareMathOperator{\dist}{dist}
\DeclareMathOperator{\argmax}{arg\,max}

\newcommand{\narrowto}{\xrightharpoonup{}}
\newcommand{\xnarrowto}[1]{\xrightharpoonup[#1]{}}
\newcommand{\weakstarto}{\xrightharpoonup{\star}}
\newcommand{\xto}[2][]{\xrightarrow[#1]{#2}}

\newcommand{\tplan}[1]{\mathbf{#1}}

\newcommand{\tplanset}{\mathbf{TP}}

\newcommand{\curvspace}{\Gamma}
\newcommand{\bas}{\mathrm{Bas}}
\newcommand{\conc}{\mathrm{conc}}

\numberwithin{equation}{section}
\declaretheorem[name=Theorem,within=section]{thm}
\declaretheorem[name=Lemma,numberlike=thm]{lemma}
\declaretheorem[name=Proposition,numberlike=thm]{proposition}
\declaretheorem[name=Corollary,numberlike=thm]{corollary}

\declaretheorem[name=Definition,numberlike=thm,style=definition]{dfn}

\declaretheorem[name=Remark,numberlike=thm,style=remark]{remark}

\allowdisplaybreaks
\date{date}
\begin{document}
\title{Optimal quantization with branched optimal transport distances}

\author{Paul Pegon\thanks{CEREMADE, Université Paris--Dauphine, Université PSL, CNRS \& MOKAPLAN, Inria Paris, 75016 Paris, France. Email: \texttt{pegon@ceremade.dauphine.fr}} \ and Mircea Petrache\thanks{Facultad de Matemáticas and Instituto de Ingenería Matemática y Computacional, Avda. Vicuña Mackenna 4860, Macul, Santiago, 6904441, Chile. Email: \texttt{mpetrache@mat.uc.cl}. ORCID id: 0000-0003-2181-169X}}
\date\today

\maketitle

\begin{abstract}
We consider the problem of optimal approximation of a target measure by an atomic measure with $N$ atoms, in branched optimal transport distance. This is a new branched transport version of optimal quantization problems. New difficulties arise, since in classical semi-discrete optimal transport with Wasserstein distance, the interfaces between cells associated with neighboring atoms have Voronoi structure and satisfy an explicit description. This description is missing for our problem, in which the cell interfaces are thought to have fractal boundary. We study the asymptotic behaviour of optimal quantizers for absolutely continuous measures as the number $N$ of atoms grows to infinity. We compute the limit distribution of the corresponding point clouds and show in particular a branched transport version of Zador's theorem. Moreover, we establish uniformity bounds of optimal quantizers in terms of separation distance and covering radius of the atoms, when the measure is $d$-Ahlfors regular. A crucial technical tool is the uniform in $N$ Hölder regularity of the landscape function, a branched transport analog to Kantorovich potentials in classical optimal transport.
\end{abstract}

\vskip\baselineskip\noindent


\tableofcontents

\section*{Notation}

This section is not necessary to read the introduction, as the notation used there should be transparent enough or informally introduced.

\subsection*{General notations}

\begin{itemize}\leftskip 0.7 cm\labelsep=.6 cm\setlength{\itemsep}{0pt}
\item[$\rev{\N,\N^*}$] \rev{Set of nonnegative integers $\{0,1, \ldots\}$ and of positive integers $\{1, 2, \ldots\}$, respectively}
\item[$\rev{\R,\R^+}$] \rev{Set of real numbers $(-\infty,+\infty)$ and nonnegative real numbers $[0,+\infty)$, respectively}
\item[$\# A$] Cardinal of a set $A$
\item[$B_r^d(x)$] Open ball with center $x$ and radius $r >0$ in $\R^d$ (superscript $d$ will often be dropped)
\item[$Q_\delta^d(x)$] $\coloneqq x + \delta[-1/2,1/2]^d$, closed cube (superscript $d$ will often be dropped)
\item[$Q_1^d$] $\coloneqq Q_1^d(0)$
\item[$\lbm^d$] Lebesgue measure on $\R^d$
\item[$\omega_d$] $\coloneqq \lbm^d(B_1(0))$
\item[$\measspace(X)$] Space of finite signed measures over a Polish space $X$
\item[$\measspace^+(X)$] Subset of measures in $\measspace(X)$ which are positive
\item[$\measspace_c^+(X)$] Subset of measures in $\measspace^+(X)$ with compact support
\item[$\probspace(X)$] Space of probability measures over $X$
\item[$\norm{\mu}$] $\coloneqq \sup \{\int \phi \dd\mu : \phi \in \contspace_b(X), \norm{\phi}_\infty \leq 1\}$, norm (total mass) of a measure $\mu$
\item[$\mu_n \narrowto\mu$] The sequence of measures $(\mu_n)$ converges narrowly to $\mu$, i.e.\ in the duality with continuous and bounded functions
\item[$\mu \wedge\nu$] Greatest submeasure of both $\mu$ and $\nu$
\item[$\mu \perp \nu$] $\mu$ and $\nu$ are mutually singular
\item[$\rho_\ac, \rho_\sing$] Density and singular part of the Radon--Nikodym decomposition of $\rho = \rho_\ac \lbm^d + \rho_\sing$
\end{itemize}
We use the conventions $0^b = +\infty$ if $b <0$, $0^0 = 1$, and $+\infty \times 0 = 0$. With a slight abuse, when $\rho \ll \lbm^d$, we will denote by $\rho$ as well its density with respect to $\lbm^d$.

\subsection*{Notations on branched optimal transport}
These notations are introduced in \Cref{sec:bot}.

\begin{itemize}\leftskip 1.8 cm\labelsep=.6 cm\setlength{\itemsep}{0pt}
\item[$\curvspace^d$] $\coloneqq \mathrm{Lip}_1(\R^+,\R^d)$, space of $1$-Lipschitz curves endowed with the (metrizable) topology of uniform convergence on compact sets
\item[$L(\gamma)$] \rev{$\coloneqq \int_\R \abs{\gamma'(t)}\dd t$, length of a Lipschitz curve $\gamma\in\Gamma^d$}
\item[$\tilde \gamma$] Arc-length reparameterization of the curve $\gamma \in \Gamma^d$
\item[$\tplanset^d$]  space of traffic plans on $\R^d$
\item[$\tplanset(\mu^-,\mu^+)$]  set of traffic plans $\tplan P$ on $\R^d$ with initial and final marginals $\mu^-$ and $\mu^+$, respectively
\item[$\tplan P_n \weakstarto \tplan P$]  weak-$\star$ convergence in $\measspace^+(\curvspace^d)$ in the duality with $\contspace(\curvspace^d)$
\item[$\Theta_{\tplan P}(x)$]  $\coloneqq \int_{\curvspace^d} \#\gamma^{-1}(\{x\})\dd\tplan P(\gamma)$, multiplicity at $x$ w.r.t. $\tplan P$
\item[$\Sigma_{\tplan P}$]  $\coloneqq \{x : \Theta_{\tplan P}(x) > 0\}$, network associated with $\tplan P$ 
\item[$\mass^\alpha(\tplan P)$] $\coloneqq \int_{\Sigma_{\tplan P}} \Theta_{\tplan P}(x)^\alpha \dd\hdm^1(x)$, $\alpha$-mass of a rectifiable traffic plan $\tplan P$
\end{itemize}
In all the paper, $\alpha$ will denote a fixed number in $(0,1)$. When needed, we will explicitly assume that $\alpha \in (1-1/d,1)$ where $d$ is the ambient dimension, in which case we will denote by $\beta = \beta(\alpha,d)$ the exponent
\begin{equation*}\label{defbeta}
\beta \coloneqq 1+d\alpha-d = d\left(\alpha - (1-1/d)\right) \in (0,1).
\end{equation*}

\section{Introduction}
In this work we study for the first time the asymptotics and uniformity properties of optimal quantization with interactions given via branched optimal transport distances, which we will also call, for brevity, \emph{branched quantization}. The fields of branched optimal transport and optimal quantization both have a large variety of applications but have not been connected before. We give a very short review and motivations of both, after which we point out why building a connection is interesting to explore.

\subsection{Branched optimal transport and optimal quantization motivations}

Branched optimal transport (or branched transport for short) is an umbrella term for a class of optimization problems, related to classical optimal transport, in which mass particles are assumed to interact (as opposed to traveling independently) while moving from a source to a target distribution. The interaction favours the transportation of particles in a grouped way by lowering the transportation cost, which is justified in many practical situations by an \emph{economy of scale}. 
A consequence of this assumption is that the particles’ paths form a one-dimensional network with a branched structure. 
The most common model assumes a cost of the form $\ell \times m^\alpha$ to move a group of particles of total mass $m$ over a distance $\ell$, where $\alpha \in (0,1)$, so that the cost is a concave power of the mass. 
This problem was first introduced in \cite{gilbertMinimumCostCommunication1967}, in a discrete setting, to optimize communications network, and was extended to two different continuous settings in \cite{xiaOptimalPathsRelated2003,maddalenaVariationalModelIrrigation2003} (both are actually equivalent \cite{paoliniOptimalTransportationNetworks2006,pegonLagrangianBranchedTransport2017}). For an introduction to the theory of branched transport we refer to the book \cite{bernotOptimalTransportationNetworks2009}. 
In pure mathematics, $H$-mass minimization over $1$-dimensional flat chains (or more generally mass minimization amongst $1$-dimensional flat $G$-chains) with fixed boundary provides versions of branched transport, see e.g. \cite{xiaInteriorRegularityOptimal2004,paoliniOptimalTransportationNetworks2006,marcheseSteinerTreeProblem2016} and the fundamental results in \cite{flemingFlatChainsFinite1966,whiteRectifiabilityFlatChains1999,whiteDeformationTheoremFlat1999,depauwSizeMinimizationApproximating2003}. 
We also mention \cite{petracheCoefficientGroupsInducing2018} for a first result on the classification of groups $G$ that produce branching. This last work, together with the classification of homotopy groups of spheres and the classification \cite{hardtConnectingRationalHomotopy2008} indicates that branched transport costs must commonly appear in connections between vortices of nonlinear Sobolev maps. 
This has recently led to important insights into weak density results such as \cite{bethuelCounterexampleWeakDensity2020}. Branched transport is also connected to size-minimization \cite{depauwSizeMinimizationApproximating2003} and the Steiner problem \cite{marcheseSteinerTreeProblem2016}, network transport systems and urban planning \cite{durandArchitectureOptimalTransport2006,buttazzoOptimalUrbanNetworks2009,brancoliniEquivalentFormulationsBranched2016}, superconductivity \cite{contiBranchedTransportLimit2018}, traffic flow optimization \cite{ibrahimOptimalTransportMultilayer2021}, models of tree roots and branches \cite{bressanOptimalShapeTree2018, bressanVariationalProblemsTree2019,bressanOptimalShapesTree2022}, models of river systems \cite{rodriguez-iturbeFractalRiverBasins2001}, amongst others. Finding the optimal branched transport map is in general NP-hard (while for classical optimal transport the complexity is $O(n^3\log n)$ for $n-$point masses), therefore computational approximations are an interesting direction of research, see for instance \cite{oudetModicaMortolaApproximationBranched2011,monteilUniformEstimatesModicaMortola2017,bonafiniConvexApproachGilbertSteiner2020,bonafiniVariationalApproximationFunctionals2018,bonafiniVariationalApproximationFunctionals2021}. Finally, we refer to \cite{devillanovaRemarksFractalStructure2019,lohmannFormulationBranchedTransport2022,lohmannDualityBranchedTransport2022,colomboStabilityMailingProblem2019,colomboWellPosednessBranchedTransportation2021,colomboStabilityOptimalTraffic2022,caldiniGenericUniquenessOptimal2023} for the most recent developments of branched transport theory.

Classical optimal quantization deals with the question of how to discretize a given positive measure $\nu\in \measspace^+_c(X)$ (in which $X=\mathbb R^d$ or $X$ is a more general metric space), in such a way that the discrete $N$-point approximant $\nu_N\in \measspace^+(\mathbb R^d)$ is at minimum distance according to a distance or cost $c:X\times X\to (0,\infty)$. Usually, it is a power of the distance over $X=\mathbb R^d$, i.e.\ $c(x,y)=|x-y|^p, p\geq 1$, inducing the classical $p$-Wasserstein distance over probability measures that is widely used in optimal transport. The quantization problem can then be reformulated as a semi-discrete optimal transport problem, see e.g. \cite{merigotMultiscaleApproachOptimal2011}, which in the case of a uniform target density reduces to the study of Voronoidal tessellations and power diagrams \cite{aurenhammerMinkowskiTypeTheoremsLeastSquares1998,duCentroidalVoronoiTessellations1999}.  We refer to general reference books \cite{grafFoundationsQuantizationProbability2000}, \cite{gershoVectorQuantizationSignal1992} for an overview of the quantization problem, and to \cite{fejestothRepresentationPopulationInfinie1959,zadorDevelopmentEvaluationProcedures1963,zadorAsymptoticQuantizationError1982,gershoAsymptoticallyOptimalBlock1979} for the first historical references. Applications of optimal quantization range from clustering \cite{saxenaReviewClusteringTechniques2017}, \cite{okabeSpatialTessellationsConcepts2000}, to signal processing \cite{gershoVectorQuantizationSignal1992}, to numerical integration and quadrature \cite{pagesSpaceQuantizationMethod1998}, to material science \cite{bourneOptimalityTriangularLattice2014, bourneAsymptoticOptimalityTriangular2021, bourneGeometricModellingPolycrystalline2023}, to spatial economy \cite{bollobasOptimalStructureMarket1972,morganHexagonalEconomicRegions2002}, where optimal quantization is often referred to as optimal location \cite{bouchitteAsymptotiqueProblemePositionnement2002,bouchitteAsymptoticAnalysisClass2011,brancoliniLongtermPlanningShortterm2009,buttazzoAsymptoticOptimalLocation2013}. Asymptotics and continuum limits of the problem as the number of discretization points tends to infinity have been studied for the classical optimal quantization problems by Zador \cite{zadorAsymptoticQuantizationError1982}, by Bouchitt\'e, Jimenez and Mahadevan \cite{bouchitteAsymptotiqueProblemePositionnement2002,bouchitteAsymptoticAnalysisClass2011}, who introduced a $\Gamma$-convergence approach, and by Gruber \cite{gruberOptimumQuantizationIts2004} who also provides geometric information on optimal configurations.

Further problems that are not directly formulated as a quantization problem but can also be seen as generalizations of the problem in other directions, appear in minimization of energies of a large number of \enquote{charges} under Riesz--Coulomb interactions: see the book \cite{borodachovDiscreteEnergyRectifiable2019}, and the crystallization survey \cite{lewinCrystallizationConjectureReview2015}. We mention the related problems of optimal unconstrained polarization \cite{hardinUnconstrainedPolarizationChebyshev2022}, jellium equidistribution \cite{petracheEquidistributionJelliumEnergy2018}, amongst others.

In this work, we focus on the case where the cost underlying the optimal quantization problem is given by a branched transport cost. The motivations for formulating this new problem come both from mathematically interesting new difficulties, and from its relevance for mathematical modelling and its potential applications.

Mathematically, the most important difficulty with the optimal quantization problem via branched transport is that the regularity of interfaces is not known (rather thought fractal), and the interfaces do not satisfy an explicit condition. This makes branched quantization much more challenging than classical optimal quantization and required us to give replacements for the main steps in the proofs known in the classical Wasserstein setting. We expect that our approach may allow to study some classes of problems involving random interfaces as well, since we do not make direct use of properties of the shapes of the interfaces in our estimates. 

In terms of modelling and applicability, in many clustering tasks the choice of classical distances is only due to their being \emph{a simple first choice} and computationally easy to handle. However, complex clustering tasks are better approximated via hierarchical tree-like clustering structures, such as those formed by branched transport networks. Many biological models such as the study of plant root competition (a natural extension to models such as \cite{bressanOptimalShapeTree2018}) would directly lead to branched quantization formulations. The same goes for supply chains modelling, in which several sources have to be optimized in order to supply a target density of users: in an urban area relying on a transportation network, branched transport distances (or other $H$-mass generalizations) are much more realistic than classical Wasserstein distances.

\subsection{Main results}

In this section we give simplified statements of our asymptotics and uniformity results for optimal branched quantizers, using a minimal amount of definitions. For full definitions and background results, see \Cref{sec:background}. 

Loosely speaking, a traffic plan $\tplan P$ between probability measures $\mu,\nu$, is a suitable measure over $1$-Lipschitz curves transporting $\mu$ to $\nu$. For $\alpha\in(0,1)$ we define the $\alpha$-mass $\mass^\alpha(\tplan P)$ as the integral of the $\alpha$-th power of the transported mass flux $\Theta_{\tplan P}$ (called \emph{multiplicity}), over the network $\Sigma_{\tplan P}$ induced by $\tplan P$, when the latter is $1$-rectifiable (see full definition in \Cref{sec:bot}). It is indeed proportional to $m^\alpha \times \ell$ when moving a total mass $m$ over a distance $\ell$. Then for a given $N\geq 1$ we consider the \emph{branched optimal quantization problem} defined as
\begin{equation*}
    \ener{E}^\alpha(\nu,N) \coloneqq \inf \left \{ \mathbf d^\alpha\left(\mu, \nu\right) : \#\spt \mu \leq N\right\},
\end{equation*}
where $\mathbf d^\alpha$ is the branched transport distance, given as the infimum of $\alpha$-mass $\mass^\alpha(\tplan P)$ amongst traffic plans $\tplan P$ transporting $\mu$ to $\nu$. When $\ener{E}^\alpha(\nu,N)$ is finite, an optimizer $\mu_N$ for this problem is called an \emph{optimal $N$-point quantizer of $\nu$.}

Our first main result is a branched transport version of a result by Zador \cite{zadorAsymptoticQuantizationError1982} valid for classical quantization.

\begin{thm}\label{thm:zador}
Let $\nu \in \measspace^+_c(\R^d)$ be a measure which is absolutely continuous with respect to the Lebesgue measure $\lbm^d$, and let $\alpha\in(1-1/d, 1)$.
\begin{enumerate}[(i)]
    \item If $(\mu_N)_{N\in\N^*}$ is a sequence of optimal $N$-point quantizers of $\nu$,
    \[
        \mu_N^\diamond \coloneqq \frac 1N \sum_{\{x : \mu_N(\{x\}) > 0\}} \delta_x \xnarrowto{N\to+\infty}  M_{\alpha,d}(\nu)^{-1} \nu^{\frac{\alpha}{\alpha+\frac 1d}},
    \]
    where $M_{\alpha,d}(\nu) \coloneqq \int_{\R^d} \nu(x)^{\frac{\alpha}{\alpha + \frac 1d}} \dd x$.
    \item The leading-order asymptotics of the optimal quantization error is given by
    \begin{equation*}
        \lim_{N \to \infty} N^{\beta/d} \ener{E}^\alpha(\nu,N) = c_{\alpha,d}M_{\alpha,d}(\nu)^{\alpha + \frac 1d}
    \end{equation*}
    where $c_{\alpha,d} \in (0,+\infty)$ is the constant defined in \labelcref{def_c_asymptotic} of \Cref{p:limitcube}.
\end{enumerate}
\end{thm}
The above theorem is a consequence of a more precise $\Gamma$-convergence result of $\mathbf d^\alpha$-distance along sequences with fixed support density, given in \Cref{thm_GCV}. The latter is a branched transport analogue of the asymptotic result of \cite{bouchitteAsymptotiqueProblemePositionnement2002} for classical optimal quantization, from which it is inspired.

Our second main result pertains to uniformity estimates on the optimal quantizer's support, adapting the general strategy of \cite{gruberOptimumQuantizationIts2004}, which is valid for classical quantization, to the branched quantization case. We provide bounds on the \emph{covering radius} and on the \emph{separation distance} for the support of optimal quantizers, both at the natural scale of $N^{-1/d}$, which is coherent with the principle that for an optimal quantizer, roughly speaking, a ball-like set of volume $\approx 1/N$ is assigned to each of the $N$ points in the quantizer's support. The covering radius bound quantifies the property that the atoms of an optimal quantizer are never farther than $c_1 N^{-1/d}$ from the support of the quantized measure $\nu$, and the separation bound indicates that the atoms of the quantizer are never closer than $c_2 N^{-1/d}$ to each other.

\begin{thm}\label{thm:delone}
    Assume $\alpha \in(1-1/d,1)$. Let $\nu \in \measspace_c^+(\R^d)$ be a compactly supported $d$-Ahlfors regular measure\footnote{Meaning that $c_A r^d \leq \nu(B_r(x)) \leq C_A r^d$ for every $x\in \spt \nu, r\in (0,\diam(\spt\nu)])$ and some constants $c_A,C_A > 0$.}  with constants $c_A,C_A > 0$ on $\R^d$ and $\mu_N = \sum_{i\leq N} m_i\delta_{x_i}$ be an $N$-point optimal quantizer with atoms $\mathcal X = \{x_i\}_{1\leq i\leq N}$. Then the covering radius and separation distance respectively enjoy the following bounds:
\begin{gather}
    \omega(\spt \nu,\mathcal X) \coloneqq \rev{\sup_{x\in\spt \nu} \min_{x'\in \mathcal X}} |x-x'|\leq c_1 N^{-1/d},\label{covering}\\
    \delta(\mathcal X) \coloneqq \min_{\rev{x,x'\in \mathcal X} : x\neq x'} |x-x'| \geq c_2 N^{-1/d}.\label{separation}
\end{gather}
for some constants $c_1,c_2>0$ that depend on $(\alpha,d,c_A,C_A,\diam(\spt\nu))$ but not on $N$. More precisely, $c_1, c_2$ are of the form $\diam(\spt\nu)$ multiplied by constants depending only on $(\alpha,d,c_A,C_A)$.
\end{thm}

We follow the general strategy of \cite{gruberOptimumQuantizationIts2004}, which entails major extra difficulties. A crucial new technical result is established in \Cref{thm:holderz}, where we give a uniform Hölder control of the so-called \emph{landscape function}, a substitute for classical Kantorovich potentials in  branched transport theory. Recall that in classical optimal transport theory, so-called Kantorovich duality allows to transform the problem into a dual version based on Kantorovich potentials (see e.g. \cite[Chapter~1]{santambrogioOptimalTransportApplied2015}). In turn, Kantorovich potentials can be used to show that for optimal quantization with cost $\abs{x-y}^p$ interfaces of the quantization cells are straight (see e.g. in \cite{merigotOptimalTransportDiscretization2021}).

In branched transport there is no useful analogue of Kantorovich duality, but optimal Kantorovich potentials have a partial analogue in the landscape function, which is in particular a (upper) first variation of the branched transport distance $\mathbf d^\alpha$. As a reference for the single-source landscape function $z_{\tplan P}$ (corresponding to case $N=1$ in our notation) see e.g. \cite{santambrogioOptimalChannelNetworks2007,brancoliniHolderRegularityLandscape2011}. Its basic properties are recalled in \Cref{properties_landscape}. For general $N\geq 1$, in \Cref{thm:holderz} we prove the following result, which holds for quantizers that are \emph{mass-optimal}, a condition that is weaker than optimality and only requires the quantizer to be optimal among measures with a fixed support but with varying masses (see \Cref{def:mass-optimal}).

\begin{thm}[Simplified statement of \Cref{thm:holderz}]\label{thm:holderz-intro}
    Let $\nu \in \measspace_c^+(\R^d)$ be a compactly supported $d$-Ahlfors regular measure, and let $\tplan P \in \tplanset(\mu,\nu)$ be an optimal traffic plan where $\mu = \sum_{i=1}^N m_i \delta_{x_i}$ is a $N$-point mass-optimal quantizer of $\nu$ with respect to $(x_i)_{1\leq i \leq N}$. There exists a unique function $z_{\tplan P} : \spt \nu \to \R_+$ that we call \emph{landscape function associated with $\tplan P$} which locally coincides with the single-source landscape functions $z_{\tplan P^{x_i}}$ for each source $x_i$ and is $\beta$-Hölder continuous for $\beta = 1+d\alpha-d \in (0,1)$, with a Hölder constant independent from $N$.
\end{thm}

This theorem provides at the same time a definition of landscape and its Hölder regularity in the case of a source measure that is a mass-optimal quantizer. By definition, this landscape function will inherit the same key properties as in the single-source case, as stated in \Cref{properties_landscape_extended}. We emphasize that in the proof of \Cref{thm:delone} we make crucial use of the \emph{uniform in $N$} Hölder control of $z_{\tplan P}$, without which we do not expect the same results to hold.

Finally, we remark that the notion of mass-optimal quantizers allows us to define analogues of Voronoi cells in the context of branched transport, that we call \emph{branched Voronoi basins}, which exhibit striking differences to Voronoi cells. In particular, we believe the boundaries of these cells to be fractal, in line with the conjecture on the so-called \emph{irrigation balls} of branched transport studied in \cite{pegonFractalShapeOptimization2019}. More insight on these Voronoi basins is given in \Cref{voronoi_basins}.

\subsection{Structure of the paper}
\begin{itemize}
\item In \Cref{sec:background} we complete the definitions underlying our main theorems, recall important foundational results in branched optimal transport and establish preliminary results on the optimal quantization and partition problems.
\item In \Cref{sec:gamma} we prove our main $\Gamma$-convergence and asymptotic results, \Cref{thm_GCV} and \Cref{thm:zador}. 
\item In \Cref{sec:landscape} we prove the above \Cref{thm:holderz-intro} on the regularity of the landscape function.
\item In \Cref{sec:uniform} we prove the uniformity results of \Cref{thm:delone}. 
\end{itemize}

\section{Background and preliminaries}\label{sec:background}
\subsection{Background in branched optimal transport}\label{sec:bot}

In this section we set up the static \enquote{Lagrangian} model of branched optimal transport based on \emph{traffic plans} developed by \cite{bernotTrafficPlans2005} and \cite{maddalenaVariationalModelIrrigation2003}. The main reference on branched optimal transport is the book \cite{bernotOptimalTransportationNetworks2009}. The presentation, notation and definitions that we adopt in this paper have been slightly simplified following more recent works such as \cite{pegonLagrangianBranchedTransport2017,colomboStabilityMailingProblem2019,colomboWellPosednessBranchedTransportation2021,colomboStabilityOptimalTraffic2022}.

\subsubsection*{Traffic plans} A \emph{traffic plan} on $\R^d$ is a finite positive measure $\tplan P\in \measspace^+(\curvspace^d)$ on the set of $1$-Lipschitz curves $\curvspace^d \coloneqq \mathrm{Lip}_1(\R^+,\R^d)$,
endowed with the metrizable topology of uniform convergence on compact sets, which is concentrated on the set of curves with finite stopping time:
\begin{gather}
    \tplan P(\{\gamma \in \curvspace^d : T(\gamma) = +\infty\}) = 0,\\
    \shortintertext{where for every $\gamma \in \curvspace^d$,}
    T(\gamma) \coloneqq \inf \{\tau\geq 0 : \gamma \text{ constant on } [\tau,+\infty)\} \in [0,+\infty].\nonumber
\end{gather}
We denote by $\tilde \gamma$ the arc-length reparameterization of a curve $\gamma\in \Gamma^d$ which stops at the length $L(\gamma)$ of $\gamma$, so that $\abs{{\tilde \gamma}'(t)} = 1$ for a.e. $t\in [0,L(\gamma)]$ and $T(\tilde \gamma) = L(\gamma) = L(\tilde\gamma)$.

We denote by $\tplanset^d$ the space of traffic plans over $\R^d$, and whenever $\mu^\pm \in \measspace^+(\R^d)$ have equal total mass we denote by $\tplanset(\mu^-,\mu^+)$ the set of traffic plans transporting $\mu^-$ to $\mu^+$ i.e.\ such that  $(e_0)_\sharp \tplan P = \mu^-$ and $(e_\infty)_\sharp \tplan P = \mu^+$ where $e_0(\gamma) \coloneqq \gamma(0)$ and $e_\infty(\gamma) \coloneqq \gamma(+\infty)\coloneqq \gamma(T(\gamma))$ for every $\gamma \in \curvspace^d$. The measures $\mu^-$ and $\mu^+$ are respectively called the \emph{source and sink measures} of $\tplan P$.

For every $x\in \R^d$, the \emph{multiplicity}
\begin{equation*}
    \Theta_{\tplan P}(x) \coloneqq \int_{\curvspace^d} \# \gamma^{-1}(\{x\}) \dd \tplan P(\gamma)
\end{equation*}
represents the amount of curves, measured by $\tplan P$, which visit $x$ (each curve being counted as many times as it visits $x$). The \emph{network} of $\tplan P$ is the (possibly empty) countably $1$-rectifiable set\footnote{It is countably $1$-rectifiable by \cite[Section~2.1]{pegonLagrangianBranchedTransport2017} or \cite[Lemma~6.3]{bernotTrafficPlans2005}, meaning that it is included, up to a $\hdm^1$-null set, in a countable union of Lipschitz curves.}
\begin{equation*}
    \Sigma_{\tplan P} \coloneqq \{x \in \R^d : \Theta_{\tplan P}(x) > 0\}.
\end{equation*}
The traffic plan $\tplan P$ is said \emph{rectifiable} if there exists a $1$-rectifiable set $\Sigma$ such that
\begin{equation}\label{eq:defrectif}
\hdm^1(\gamma(\R^+) \setminus \Sigma) = 0\text{ for $\tplan P$-almost every }\gamma \in \curvspace^d,
\end{equation}
in which case \labelcref{eq:defrectif} holds with $\Sigma = \Sigma_{\tplan P}$. It is said \emph{simple} if it is concentrated on simple curves, i.e.\ curves $\gamma \in \curvspace^d$ such that $\gamma$ is constant on $[s,t]$ whenever $\gamma(s) = \gamma(t)$ and $s < t$.

Finally, two traffic plans $\tplan P_1, \tplan P_2$ are said \emph{disjoint} if there exist two disjoint sets $A_1,A_2 \subseteq\R^d$ such that for $i \in \{1,2\}$,
\begin{equation}\label{eq:defdisjoint}
\hdm^1(\gamma(\R^+) \setminus A_i) = 0\text{ for $\tplan P_i$-almost every }\gamma \in \curvspace^d.
\end{equation}
For rectifiable traffic plans $\tplan P_1, \tplan P_2$, it is equivalent to
\begin{equation}\label{disjointness_rectifiable}
\hdm^1(\Sigma_{\tplan P_1} \cap \Sigma_{\tplan P_2}) = 0 \quad\text{or equivalently} \quad \Theta_{\tplan P_1} \hdm^1 \perp \Theta_{\tplan P_2} \hdm^1.
\end{equation}

\subsubsection*{Concatenation of traffic plans} 

We follow the presentation of concatenations provided in \cite[\S3.3]{colomboStabilityMailingProblem2019}. If $(\gamma_1, \gamma_2)$ belong to the set $\Lambda^d \subseteq \curvspace^d \times \curvspace^d$ of pairs of curves which have finite stopping time and satisfy $\gamma_1(+\infty) = \gamma_2(0)$, we set for every $t\in \R^+$
\[\cat{\gamma_1}{\gamma_2}(t) \coloneqq \begin{dcases*}
    \gamma_1(t)& if $t\in [0,T(\gamma_1))$,\\
    \gamma_2(t-T(\gamma_1))& if $t \in [T(\gamma_1),+\infty)$.
    \end{dcases*}\]
We denote this map by $\conc : \Lambda^d \to \curvspace^d$.

If $\tplan P_1, \tplan P_2 \in \tplanset^d$ are such that $(e_\infty)_\sharp \tplan P_1 = (e_0)_\sharp \tplan P_2$, we say that $\tplan P$ is a concatenation of $\tplan P_1$ and $\tplan P_2$ if there exists a measure $\tplan{\bar P} \in \measspace^+(\curvspace^d \times \curvspace^d)$, which is concentrated on $\Lambda^d$ and satisfies
\begin{gather*}
    \tplan P = \conc_\#\tplan{\bar P}\\
    (p_i)_\sharp \tplan{\bar P} = \tplan P_i \quad \text{where $p_i : (\gamma_1,\gamma_2) \mapsto \gamma_i$ for $i\in\{1,2\}$}.
\end{gather*}
We denote by $\cat{\tplan P_1}{\tplan P_2}$ the set of concatenations of $\tplan P_1$ and $\tplan P_2$. We will need some properties of concatenations that are summarized in the following proposition, extracted from \cite[\S3.3]{colomboStabilityMailingProblem2019}.

\begin{proposition}[{\cite[Lemma~3.6]{colomboStabilityMailingProblem2019}}]\label{concatenation}
    If $\tplan P_1, \tplan P_2 \in \tplanset^d$ are such that $(e_\infty)_\sharp \tplan P_1 = (e_0)_\sharp \tplan P_2$, then:
    \begin{enumerate}[(i)]
    \item\label{concat_exist} $\cat{\tplan P_1}{\tplan P_2}$ is nonempty,
    \item\label{concat_marginals} $(e_0)_\sharp \tplan P = (e_0)_\sharp \tplan P_1$ and $(e_\infty)_\sharp \tplan P = (e_\infty)_\sharp \tplan P_2$ for every $\tplan P \in (\tplan P_1 : \tplan P_2)$,
    \item\label{concat_multip} for every $\tplan P \in \cat{\tplan P_1}{\tplan P_2}$, $\Theta_{\tplan P} = \Theta_{\tplan P_1} + \Theta_{\tplan P_2}$,
    \item\label{concat_sum} if $\tplan P \in \cat{\tplan P_1}{\tplan P_2}$ and $\tplan {P'} \in \cat{\tplan{P'}_1}{\tplan{P'}_2}$ then $\tplan P + \tplan{P'} \in \cat{\tplan P_1+\tplan{P'}_1}{\tplan P_2+\tplan{P'}_2}$.
    \end{enumerate}
\end{proposition}

\subsubsection*{The $\alpha$-mass functional and the irrigation problem}

For $\alpha \in (0,1)$, the $\alpha$-mass\footnote{It is the equivalent, for traffic plans, of the $\alpha$-mass of currents.} of a traffic plan is defined as
\begin{equation}\label{irrigation_problem}
\mass^\alpha(\tplan P) = \begin{dcases*}
    \int_{\Sigma_{\tplan P}} \Theta_{\tplan P}(x)^\alpha \dd\hdm^1(x)& if $\tplan P$ is rectifiable,\\
    +\infty& otherwise.
\end{dcases*} 
\end{equation}
If $\mu^\pm$ are two positive measures on $\mathbb R^d$ of equal (finite) mass, the irrigation problem then reads as
\begin{equation}\label{lag_irrigation_pb}\tag{$\text{I}^\alpha$}
\inf \quad \set*{\mass^\alpha(\tplan P) \given \tplan P \in \tplanset(\mu^-,\mu^+)},
\end{equation}
and we denote by $\mathbf d^\alpha(\mu^-,\mu^+)$ this infimum value.

\begin{dfn}[Optimal traffic plan]
We say that $\tplan P \in \tplanset(\mu^-,\mu^+)$ is optimal if the infimum in \labelcref{lag_irrigation_pb} is finite and attained by $\tplan P$.
\end{dfn}

Let us state some known results that we shall use throughout the paper.
\begin{enumerate}[(1)]
    \item \textbf{Subadditivity, additivity and disjointness.} The $\alpha$-mass is subadditive in the sense that
    \begin{equation}\label{sub_additivity}
    \begin{gathered}
        \forall \tplan P_1, \tplan P_2 \in \tplanset^d, \quad \mass^\alpha(\tplan P_1 + \tplan P_2) \leq \mass^\alpha(\tplan P_1) + \mass^\alpha(\tplan P_2)\\
        \text{with equality when $\tplan P_1$ and $\tplan P_2$ are disjoint.}
    \end{gathered}
    \end{equation}
     Furthermore, if $\tplan P^1$ and $\tplan P^2$ are disjoint and $\tplan P^1+\tplan P^2$ is optimal, then $\tplan P^1$ and $\tplan P^2$ are also optimal (for their own marginals).

     Finally, it follows from \Cref{concatenation} \labelcref{concat_multip} that for every concatenation $\tplan P \in \cat{\tplan P_1}{\tplan P_2}$
     \begin{equation}\label{concat_alpha_mass}
     \mass^\alpha(\tplan P) \leq \mass^\alpha(\tplan P_1) + \mass^\alpha(\tplan P_2).
     \end{equation}
    \item \textbf{Irrigability and irrigation distance.} Contrary to the classical optimal transport problem, for some pair of compactly supported measures $(\mu^-,\mu^+)$ and some exponent $\alpha$ it is possible that \labelcref{irrigation_problem} admits no competitor of finite $\alpha$-mass, typically when the measures spread on a set of large dimension while the exponent $\alpha$ is too small. However, when $\alpha > 1 - \frac 1d$ any measure $\mu \in \measspace^+_c(\R^d)$ is \emph{$\alpha$-irrigable}, meaning that $\mathbf d^\alpha(\mu, \norm{\mu}\delta_0) < +\infty$, and there exists a competitor of finite $\alpha$-mass for any pair of measures $\mu^-,\mu^+ \in \measspace_c^+(\R^d)$ of equal total mass, as shown\footnote{It is shown in the Eulerian model based on vector measures, but adapting the proof in our Lagrangian setting is straightforward, or one can also invoke the equivalence of the models \cite{pegonLagrangianBranchedTransport2017}.} in \cite{xiaOptimalPathsRelated2003}. In particular
    \begin{equation*}
    \begin{gathered}
       (\alpha \in (1-1/d,1) \quad \text{and} \quad K \subseteq \R^d \text{ compact})\\
        \Downarrow\\
        (\mathbf d^\alpha\text{ is a distance on $\probspace(K)$ which metrizes weak-$\star$ convergence in $\contspace(K)'$}).
    \end{gathered}
    \end{equation*}
    \item \textbf{Existence for the irrigation problem.} Problem \labelcref{lag_irrigation_pb} \rev{admits a minimizer whenever it admits a competitor of finite cost.} It results for example from the existence of $\mathbb{E}^\alpha$ minimizers established in \cite[Section~3.4]{bernotOptimalTransportationNetworks2009}, where $\mathbb{E}^\alpha$ is a more complicated variant\footnote{A reason for using this functional $\mathbb{E}^\alpha$ (denoted by $\ener E^\alpha$ in \cite{bernotOptimalTransportationNetworks2009}) was that one could establish its lower semicontinuity (on suitable subsets) via another expression (so-called energy formula). Nowadays, one may instead prove lower semicontinuity of $\mass^\alpha$ and work with it directly.} of $\mass^\alpha$, knowing that $\mathbb{E}^\alpha \geq \mass^\alpha$ and that a minimizer of $\mathbb{E}^\alpha$ is also a minimizer of $\mass^\alpha$ with the same cost.
    \item \textbf{Upper estimates on the $\alpha$-mass.} If $\alpha \in ( 1-\frac 1d,1)$, there exists a constant $C_\BOT = C_\BOT(\alpha,d) \in (0,+\infty)$ such that for any compactly supported measures $\mu^\pm$ of equal total mass,
    \begin{equation}\label{upper_estimate_alpha_mass}
        \mathbf d^\alpha(\mu^-,\mu^+) \leq C_\BOT \diam(\spt(\mu^+-\mu^-)) \norm{\mu^+-\mu^-}^\alpha.
    \end{equation}
    Indeed, it is proven in \cite{xiaOptimalPathsRelated2003} that $\mathbf d^\alpha(\delta_0,\mu) \leq C_\BOT(\alpha,d)/2$ for every $\mu \in \probspace(Q_1^d)$ and some best constant $C_\BOT(\alpha,d) \in (0,\infty)$, from which we deduce $\mathbf d^\alpha(\mu^-,\mu^+) \leq C r m^\alpha$ when $\diam(\spt(\mu^-+\mu^+)) \leq r$ and $\norm{\mu^-} = \norm{\mu^+} = m$ using the triangle inequality and the $1$-homogeneity in space and $\alpha$-homogeneity in mass of the $\alpha$-mass. Applying this to the measures $\tilde \mu^\pm = \mu^\pm - \mu^-\wedge \mu^+$ yields \labelcref{upper_estimate_alpha_mass} since $\norm{\tilde\mu^\pm} = \norm{\mu^+-\mu^-}/2$, $\spt (\mu^+-\mu^-) = \spt (\tilde \mu^+ + \tilde \mu^-)$, and $\mathbf d^\alpha(\mu^-,\mu^+) = \mathbf d^\alpha(\tilde\mu^-,\tilde\mu^+)$.
    \item \textbf{First variation of the $\alpha$-mass.} If $\tplan{P}, \tplan{\tilde P}$ are traffic plans with $\mass^\alpha(\tplan P) < \infty$, then
    \begin{gather}
        \mass^\alpha(\tplan{\tilde P}) \leq \mass^\alpha(\tplan P) + \alpha \int_{\curvspace^d} Z_{\tplan P}(\gamma) \dd(\tplan{\tilde P} - \tplan P)(\gamma),\label{first_variation_alpha_mass}\\
        \shortintertext{where}
        Z_{\tplan P}(\gamma) \coloneqq \int_\gamma \Theta_{\tplan P}^{\alpha -1} = \int_{\gamma(\R)} \Theta_{\tplan P}(x)^{\alpha -1} \#\gamma^{-1}(x) \dd\hdm^1(x)\label{landscape_precursor}.
    \end{gather}
    The proof of \labelcref{first_variation_alpha_mass} relies on the concavity of $m\mapsto m^\alpha$ on $\R^+$ applied to $\Theta_{\tplan{\tilde P}} = \Theta_{\tplan{P}} + (\Theta_{\tplan{\tilde P}} - \Theta_{\tplan{P}})$ and on Fubini's theorem (we refer to \cite[Chapter~11]{bernotOptimalTransportationNetworks2009}, \cite[Theorem~3.1]{santambrogioOptimalChannelNetworks2007}). Notice that the integral $\int Z_{\tplan P} \dd(\tplan{\tilde P}-\tplan P)$ is well-defined (possibly infinite) since by Fubini's theorem we may show that
    \begin{align*}
        \infty > \mass^\alpha(\tplan P) & \rev{ =\int_{\R^d} \Theta_{\tplan P}(x)^{\alpha-1} \Theta_{\tplan P}(x) \dd\hdm^1(x)}\\
        &\rev{= \int_{\R^d} \int_{\Gamma^d} \Theta_{\tplan P}(x)^{\alpha -1} \#\gamma^{-1}(x) \dd\tplan P(\gamma)\dd\hdm^1(x)} =\int_{\curvspace^d} Z_{\tplan P} \dd{\tplan P}
    \end{align*}
    and $\int_{\curvspace^d} Z_{\tplan P} \dd\tplan{\tilde P} \in [0,\infty]$. 
    \item \textbf{Single-path property.} If $\tplan P \in \tplanset(\mu, \nu)$ is an optimal traffic plan, it is simple and satisfies the single-path property, which can be stated in the single-source case where $\mu = m \delta_s$ as follows: for every $x \in \Sigma_{\tplan P}$, there exists a (unique) injective curve parameterized by arc length $\gamma_{\tplan P,x} : [0,\ell] \to \R^d$ such that $\tplan P$-a.e. curve $\gamma$ passing by $x$ follows the trajectory of $\gamma_{\tplan P,x}$, meaning: if $t_x(\gamma)$ denotes the greatest $t \in [0,T(\gamma)]$ such that $\gamma(t) = x$,
    \begin{equation}\label{single_path_property}
        \text{for $\tplan P$-a.e. $\gamma$ s.t. $x\in \gamma(\R_+)$,} \quad \tilde{\gamma}_{[0,t_x(\tilde\gamma)]} = \gamma_{\tplan P,x},
    \end{equation}
    where we recall $\tilde{\gamma}$ denotes the arc-length reparameterization of $\gamma \in\curvspace^d$. This fact is stated in \cite[Proposition~7.4]{bernotOptimalTransportationNetworks2009}
    .
    \end{enumerate}

\subsubsection*{Landscape function for a single source}

Given an optimal irrigation plan ${\tplan P} \in \tplanset(m\delta_s,\nu)$, following \cite{santambrogioOptimalChannelNetworks2007} we say that a curve $\gamma$ is \emph{${\tplan P}$-good} if, recalling the notation \labelcref{landscape_precursor},
\begin{itemize}
\item $Z_{\tplan P}(\gamma) <+\infty$,
\item \rev{for all $x=\gamma(t)$ with $t < T(\gamma)$,}
\[\rev{\Theta_{\tplan P}(x) = {\tplan P}(\{\kappa\in \Gamma^d : \tilde\gamma = \tilde\kappa \text{ on } [0,t_x(\tilde \gamma)]\}).}\]
\end{itemize}
It is proven in \cite{santambrogioOptimalChannelNetworks2007} that any optimal traffic plan ${\tplan P}$ is concentrated on the set of ${\tplan P}$-good curves, and that for every ${\tplan P}$-good curve $\gamma$ the quantity $Z_{\tplan P}(\gamma)$ depends only on the final point $\gamma(\infty)$ of the curve, thus we may define the landscape function $z_{\tplan P}$ as follows:
\begin{equation}\label{landscdef}
z_{\tplan P}(x)\coloneqq \left\{\begin{array}{ll}
Z_{\tplan P}(\gamma)&
\mbox{if $\gamma$ is an ${\tplan P}$-good curve s.t. $x = \gamma(\infty)$,}\\
+\infty&
\mbox{otherwise.}
\end{array}\right.
\end{equation}

We summarize the properties of the landscape function, extracted from\footnote{One may look at \cite{pegonBranchedTransportFractal2017} to see the slightly more general statements presented here, with similar notation.} \cite{santambrogioOptimalChannelNetworks2007}, that we shall need.
\begin{proposition}\label{properties_landscape}
Let $\nu \in \measspace^+(\R^d)$ of total mass $m$. If $\tplan P \in \tplanset(m\delta_s,\nu)$ is optimal with $\mass^\alpha(\tplan P) <\infty$, $\alpha \in [0,1)$, and $z_\tplan P$ is its landscape function, then $z_{\tplan P}$ is lower semicontinuous and finite on $\Sigma_{\tplan P} \cup \spt \nu$. Moreover:
\begin{enumerate}[(A)]
    \item\label{formula_greater distance} $z_{\tplan P}(x)\geq m^{\alpha-1}|x-s|$ for every $x \in \R^d$;
    \item\label{formula_optimal_cost_landscape} the $\alpha$-mass may be expressed in terms of $z_{\tplan P}$:
    \begin{equation*}
        \mathbf d^\alpha(m\delta_s, \nu) = \mass^\alpha({\tplan P}) = \int_{\curvspace^d} Z_{\tplan P}(\gamma) \dd{\tplan P}(\gamma) = \int_{\mathbb R^d} z_{\tplan P}(x) \dd\nu(x);
    \end{equation*}
    \item\label{first_variation_mass_landscape} if $\tplan{\tilde P} \in \tplanset(m\delta_s, \tilde \nu)$ is a traffic plan concentrated on $\tplan P$-good curves, then
    \begin{equation*}
        \mass^\alpha(\tplan{\tilde P}) \leq \mass^\alpha(\tplan P) + \alpha \int_{\R^d} z_{\tplan P} \dd(\tilde\nu-\nu),
    \end{equation*}
    and the inequality is strict if $\Theta_{\tplan{\tilde P}}-\Theta_{\tplan P}$ does not vanish $\hdm^1$-a.e. on $\Sigma_{\tplan P}$;
    \item\label{first_variation_distance} in particular, $z_{\tplan P}$ is an \emph{upper first variation of the irrigation distance}, in the sense that for every $\tilde \nu \in \measspace^+(\R^d)$,
    \begin{equation*}
        \mathbf d^\alpha(\norm{\tilde\nu}\delta_s,\tilde\nu) \leq \mathbf d^\alpha(\norm{\nu}\delta_s,\nu) + \alpha \int_{\R^d} z_{\tplan P} \dd(\tilde\nu - \nu).
    \end{equation*}
\end{enumerate}
\end{proposition}


\subsection{The optimal quantization problem and mass-optimal quantizers}
For $\alpha \in (0,1)$ and $\nu\in \measspace^+(\R^d)$, the optimal branched quantization problem is the following:
\begin{equation}\label{pb:opt_quantization}
\ener{E}^\alpha(\nu,N) \coloneqq \inf \left \{ \mathbf d^\alpha\left(\mu, \nu\right) : \norm{\mu}=\norm{\nu} \text{ and } \#\spt \mu \leq N\right\}.
\end{equation}
An admissible candidate $\mu_N$ in this problem will be called a $N$-point quantizer of $\nu$.

\begin{dfn}[Optimal quantizer]
When $\ener E^\alpha(\nu,N) < +\infty$ a solution of \labelcref{pb:opt_quantization} will be called an \emph{$N$-point optimal quantizer} of $\nu$.
\end{dfn}

\begin{thm}\label{existence_opt_quantization}
    For any finite positive measure  $\nu\in \measspace^+(\R^d)$ and any $N\in \N^*$, \rev{if $\ener E^\alpha(\nu,N) < +\infty$ then} the optimal quantization problem \labelcref{pb:opt_quantization} admits a solution.
\end{thm}
\begin{proof}
Take an integer $N \geq 1$ and $\nu$ a finite positive measure on $\R^d$, assuming without loss of generality that it has unit mass. Suppose that $\ener{E}^\alpha(\nu,N) < +\infty$ (otherwise there is nothing to prove) and take $(\mu^n)_{n\in \N}$ a minimizing sequence with $\sup_{n \in \N} \mathbf d^\alpha(\mu^n,\nu) \eqqcolon \Lambda < +\infty$. Let us show that it is tight. Take $\eps > 0$ and $R \geq 4\Lambda/\eps$ large enough such that $\nu(\R^d \setminus B_{R/2}) \leq \eps/2$. If $\tplan P^n \in \tplanset(\mu^n,\nu)$ is an optimal traffic plan, then
\begin{align*}
    \Lambda \geq \mass^\alpha(\tplan P^n) &\geq \int_{\curvspace^d} L(\gamma) \dd\tplan P^n(\gamma)\\
    &\geq \int_{\curvspace^d} L(\gamma) \one_{\{\gamma : \gamma(0) \in \R^d\setminus B_{R}, \gamma(\infty) \in B_{R/2}\}} \dd\tplan P^n(\gamma)\\
    &\geq \frac{R}2 (\mu^n(\R^d\setminus B_R) - \nu(\R^d\setminus B_{R/2})),
\end{align*}
which implies that
\[\mu^n(\R^d \setminus B_R) \leq 2\Lambda/R + \nu(\R^d\setminus B_{R/2}) \leq \eps.\]
As a consequence, $(\mu^n)$ admits a subsequence converging narrowly to some $\mu \in \probspace(\R^d)$. Necessarily, $\mu$ has at most $N$ atoms as well and by lower semicontinuity\footnote{In the compact case, lower semicontinuity of $\mathbf d^\alpha$ is obtained from \cite[Chapter~3]{bernotOptimalTransportationNetworks2009}, see for example \cite[Section~1.2.3]{pegonBranchedTransportFractal2017} for the case of $\R^d$.} of $\mathbf d^\alpha$ for the narrow convergence of probability measures $\mu$ is a $N$-point optimal quantizer.
\end{proof}

We will consider a class of quantizers which is broader than optimal quantizers and that we call \emph{mass-optimal quantizers}. They will be used to establish the full $\Gamma$-convergence result in \Cref{sec:gamma}, and may also provide a notion of Voronoi cells, called \emph{Voronoi basins}, in the setting of branched optimal transport (see \Cref{voronoi_basins}).

\begin{dfn}[Mass-optimal quantizer]\label{def:mass-optimal}
    Let $\nu \in \measspace^+(\R^d)$ be a finite positive measure and $\mathcal X = \{x_i\}_{1\leq i\leq N}$ be a set of cardinality $N$. If $\mu$ is a measure supported on $\mathcal X$ such that the masses of its atoms are chosen in the best way to approximate $\nu$ in $\mathbf d^\alpha$ distance, i.e.\ $\mathbf d^\alpha(\mu,\nu) = \mathbf d^\alpha(\mathcal X,\nu)$ where
    \begin{equation*}
        \mathbf d^\alpha(\mathcal X,\nu) \coloneqq \inf \left\{\mathbf d^\alpha(\mu',\nu) : \spt \mu' \subseteq \mathcal X\right\},
    \end{equation*}
    we say that $\mu$ is an $N$-point \emph{mass-optimal quantizer} with respect to $\{x_i\}_{1\leq i\leq N}$.
\end{dfn}

We will also need to decompose any traffic plan $\tplan P\in \tplanset(\mu,\nu)$, where $\mu$ is purely atomic, with respect to the atoms of $\mu$, also called the \emph{sources of $\tplan P$}.

\begin{dfn}[Restrictions and basins from a source]
    If $\tplan P\in \tplanset(\mu,\nu)$ where $\mu$ is purely atomic and $x$ is an atom of $\mu$, the \emph{restriction of $\tplan P$ from the source $x$} is defined by
\begin{equation*}\label{def_restriction_source}
    \tplan P^x \coloneqq \tplan P \mres\{\gamma\in \curvspace^d : \gamma(0)=x\},
\end{equation*}
so that the following \emph{source decomposition} of $\tplan P$ holds:
\begin{equation*}\label{decomposition_sources}
    \tplan P = \sum_{x\in \spt{\mu}} \tplan P^x.
\end{equation*}
The decomposition is said \emph{disjoint} if all these restrictions $(\tplan P^x)_{x\in \spt \mu}$ are pairwise disjoint (as defined in \labelcref{eq:defdisjoint}). 
We also introduce the \emph{basin from $x$ with respect to $\tplan P$} as the support of the sink measure of $\tplan P^x$:
\begin{equation*}\label{notation_basin}
    \bas(\tplan P, x)\coloneqq \spt ((e_\infty)_\sharp \tplan P^x).
\end{equation*}
\end{dfn}

In the next lemma, we show that the source decomposition of an optimal traffic plan between a measure and a mass-optimal quantizer is disjoint in the above sense, and that the corresponding sink measures are mutually singular. This result plays a key role in the proof of equivalence between the optimal quantization and optimal partition, and will also be crucial in \Cref{sec:landscape} to define and study the landscape function, and subsequently to show the disjointness of irrigation basins.

\begin{lemma}[Disjointness properties of mass-optimal quantizers]\label{quantizers_disjointness}
    Let $\mu$ of the form $\sum_{i=1}^N m_i \delta_{x_i}$ be an $N$-point mass-optimal quantizer of $\nu \in \measspace_c^+(\R^d)$ with respect to $\mathcal X \coloneqq \{x_i\}_{1\leq i\leq N}$ and $\tplan P \in \tplanset(\mu,\nu)$ be an optimal traffic plan. Then:
    \begin{enumerate}[(i)]
        \item\label{disjointness_plans} the traffic plans $\tplan P^{x_i}$ are disjoint, thus they are optimal for their own marginals, and
        \[
            \mass^\alpha(\tplan P) = \sum_{i=1}^N \mass^\alpha(\tplan P^{x_i}),
        \]
        \item\label{disjointness_irrigated_measures} the irrigated measures $\nu_i \coloneqq (e_\infty)_\sharp \tplan P^{x_i}$ are mutually singular.
    \end{enumerate}
\end{lemma}
\begin{proof}
    Let us start by proving \labelcref{disjointness_plans}. Since $\tplan P$ is rectifiable, so are the $\tplan P^{x_i}$'s, thus by \labelcref{disjointness_rectifiable} it suffices to show that the measures $\Theta_{\tplan P^{x_i}} \hdm^1$ are mutually singular. By contradiction, assume that for some $i\neq j$, $\Theta_{\tplan P^{x_i}} \hdm^1$ and $\Theta_{\tplan P^{x_j}} \hdm^1$ are not mutually singular. We shall contradict the optimality of $\tplan P$. There exists a Borel set $A \subseteq \Sigma_{\tplan P}$ and a constant $m_0 > 0$ such that $\hdm^1(A) > 0$ and $\Theta_{\tplan P^{x_i}}(x) \wedge \Theta_{\tplan P^{x_j}}(x) \geq m_0$ for every $x\in A$. Pick a point $x \in A$ with $x\neq x_i$ and $x\neq x_j$ and consider for $k \in \{i,j\}$ a traffic plan
    \[\tplan P_k \leq \tplan P \mres \{\gamma \in \curvspace^d : \gamma(0) = x_k, x \in \gamma(\R_+)\} \quad\text{such that}\quad \norm{\tplan P_k} = m_0.\]
    For every $\eps \in [0,1]$, we will build a traffic plan $\tplan P_\eps$, obtained from $\tplan P$ by taking a fraction $\eps$ of $\tplan P_i$, replacing for each curve $\gamma$ of $\tplan P_i$ the curve segment between $x_i$ and $x$ by a segment of a curve $\tilde\gamma$ of $\tplan P_j$ from $x_j$ to $x$. To do this, consider the map $t_x : \gamma \mapsto \min \gamma^{-1}(\{x\})$ and the restriction maps $r_x^- : \gamma \mapsto \gamma_{|[0,t_x(\gamma)]}$, $r_x^+ : \gamma \mapsto \gamma_{|[t_x(\gamma),+\infty)}$. Then
    \[
        (e_\infty)_\sharp (r_x^-)_\sharp \tplan P_j = (e_0)_\sharp (r_x^+)_\sharp \tplan P_i = m_0 \delta_x,
    \]
    and by \Cref{concatenation} \labelcref{concat_exist} there exists a concatenation $\tplan Q \in ((r_x^-)_\sharp \tplan P_j : (r_x^+)_\sharp \tplan P_i)$. For $\eps \in [0,1)$ set
    \[
        \tplan P_\eps \coloneqq \tplan P + \eps (\tplan Q-\tplan P_i).
    \]
    We shall do the converse operation for $\eps \in (-1,0)$, namely
    \[\tplan P_\eps \coloneqq \tplan P - \eps (\tplan {Q'} - \tplan P_j), \quad\text{where}\quad \tplan {Q'} \in ((r_x^-)_\sharp \tplan P_i : (r_x^+)_\sharp \tplan P_j).\]
    Notice that for both possible signs of $\eps$, $\tplan P_\eps$ is rectifiable, $\Sigma_{\tplan P_\eps} \subseteq \Sigma_{\tplan P}$ and
    \begin{gather}\label{eq:thetadeltatheta}
        \Theta_{\tplan P_\eps} = \Theta_{\tplan P} + \eps \Delta\Theta
        \quad\text{where}\quad \Delta\Theta \coloneqq \Theta_{(r_x^-)_\sharp\tplan P_j}-\Theta_{(r_x^-)_\sharp\tplan P_i}.
    \end{gather}
    Indeed if for example $\eps\geq 0$ then by \Cref{concatenation} \labelcref{concat_multip} we have $\Theta_{\tplan Q} = \Theta_{(r_x^-)_\sharp\tplan P_j} + \Theta_{(r_x^+)_\sharp\tplan P_i}$, and \labelcref{eq:thetadeltatheta} follows because $\Theta_{\tplan P_i} = \Theta_{(r_x^-)_\sharp\tplan P_i} + \Theta_{(r_x^+)_\sharp\tplan P_i}$.

    Now, the initial measure $\mu_\eps \coloneqq (e_0)_\sharp \tplan P_\eps$ of $\tplan P_\eps$ is still supported on $\{x_i : 1\leq i\leq N\}$ and its final measure is still $\nu$, thus by mass-optimality of $\mu$,
    \[
        \int_{\R^d} \Theta_{\tplan P} ^\alpha \dd \hdm^1 = \mass^\alpha(\tplan P) \leq \mass^\alpha(\tplan P_\eps) = \int_{\Sigma_{\tplan P}}(\Theta_{\tplan P}+ \eps\Delta\Theta)^\alpha \dd \hdm^1.
    \]
    For $k\in \{i,j\}$, by the single-path property recalled in \labelcref{single_path_property}, $\tplan P$-a.e. curve starting at $x_k$ and visiting $x$ follows a trajectory given by a single (simple, parameterized by arc length) curve $\gamma_k$ such that $\gamma_k(0) = x_k$, $ \gamma_k(\infty) = x$ and $\gamma_k(\R_+)\subseteq \Sigma_{\tplan P}$; in particular, $\Theta_{(r_x^-)_\sharp\tplan P_k} = m_0 \one_{\gamma_k(\R_+)}$. Since $x_i,x_j,x$ are distinct points, we get
    \[
        \hdm^1(\{y \in\Sigma_{\tplan P} : \Delta \Theta(y) \neq 0\}) = \hdm^1(\gamma_i(\R_+) \Delta \gamma_j(\R_+)) > 0,
    \]
    and as $\alpha \in (0,1)$, the function $\eps \mapsto \int_{\Sigma_{\tplan P}} (\Theta_{\tplan P}+\eps\Delta\Theta)^\alpha \dd\hdm^1$ is finite\footnote{Since $\abs{\Delta\Theta}\leq \Theta_{\tplan P}$ and $\mass^\alpha(\tplan P) < +\infty$.} and strictly concave on $(-1,1)$. But it is minimized at $\eps = 0$: a contradiction.  Consequently, $\Theta_{\tplan P^{x_i}} \hdm^1$ and $\Theta_{\tplan P^{x_j}} \hdm^1$ are mutually singular for every $i\neq j$.

    Since $\tplan P = \sum_{i=1}^N \tplan P^{x_i}$ is a disjoint decomposition, by \labelcref{sub_additivity} we get that $\mass^\alpha(\tplan P) = \sum_{i=1}^N \mass^\alpha(\tplan P^{x_i})$ and the optimality of each $\tplan P^{x_i}$ for $i\in \{1,\ldots,N\}$ follows from that of $\tplan P$.

    Let us now prove \labelcref{disjointness_irrigated_measures}. For every $i \in \{1,\ldots, N\}$, the traffic plan $\tplan P^{x_i}$ being optimal with a single source, we may consider its landscape functions $z_{\tplan P^{x_i}}$ as in \labelcref{landscdef}. By contradiction assume that $\nu_i\perp \nu_j$ does not hold for some $i \neq j$. Then we have:
    \[
        m \coloneqq \norm{\tilde\nu} >0 \quad\text{where}\quad \tilde\nu \coloneqq \nu_i \wedge \nu_j.
    \]
    For $k \in \{i,j\}$ consider a plan $\tplan P_k \leq \tplan P^{x_k}$ such that $\tplan P_k \in \tplanset(m\delta_{x_k},\tilde\nu)$, and define for $\eps \in (-1,1)\setminus\{0\}$ the competitor
    \[
        \tplan P_\eps = \tplan P +\eps (\tplan P_j - \tplan P_i).
    \]
    Its initial measure $(e_0)_\sharp \tplan P_\eps$ is still supported on $\{x_i : 1 \leq i \leq N\}$ and its final measure is still $\nu$, thus by mass-optimality of $\mu$ we get
    \begin{align*}
        &\qquad\mass^\alpha(\tplan P) =\mathbf d^\alpha(\mathcal X,\nu) \leq \mass^\alpha(\tplan P_\eps)\\
        &\leq \sum_{k\neq i,j} \mass^\alpha(\tplan P^{x_k}) + \mass^\alpha(\tplan P^{x_i} -\eps \tplan P_i) + \mass^\alpha(\tplan P^{x_j} + \eps \tplan P_j)\\
        &< \sum_{k\neq i,j} \mass^\alpha(\tplan P^{x_k}) + \mass^\alpha(\tplan P^{x_i}) - \alpha\eps \int_{\R^d} z_{\tplan P^{x_i}} \dd(e_\infty)_\sharp \tplan P_i + \mass^\alpha(\tplan P^{x_j}) + \alpha\eps \int_{\R^d} z_{\tplan P^{x_j}} \dd(e_\infty)_\sharp \tplan P_j\\
        &= \mass^\alpha(\tplan P) + \alpha \eps \int_{\R^d} (z_{\tplan P^{x_j}}-z_{\tplan P^{x_i}}) \dd\tilde\nu,
    \end{align*}
    where we have used the first variation inequality \Cref{properties_landscape} \labelcref{first_variation_mass_landscape} twice on the third line. The inequality in the third line is strict for $\epsilon\in(-1,1)\setminus\{0\}$, because for $k\in \{i,j\}$, $\Theta_{\tplan P_k} \leq \Theta_{\tplan P^{x_k}}$, thus for every $y\in\R^d$ such that $\Theta_{\tplan P_k}(y) > 0$ we have $\abs{\eps} \Theta_{\tplan P_k} < \Theta_{\tplan P^{x_k}}$: this strict inequality holds on a $\hdm^1$-positive subset of $\Sigma_{\tplan P^{x_k}}$. We can now choose $\eps$ such that $\eps\int_{\R^d} (z_{\tplan P^{x_j}}-z_{\tplan P^{x_i}}) \dd\tilde\nu\leq 0$, and we get the contradiction $\mass^\alpha(\tplan P) < \mass^\alpha(\tplan P)$.
\end{proof}

\subsection{The optimal partition problem and equivalence with optimal quantization}

For $\alpha \in (0,1)$ and any finite nonnegative measure $\nu$, we define
\begin{equation*}\label{pb:opt_partition_sets1}
\mathbf{C}^\alpha(\nu) \coloneqq \inf_{x \in \R^d} \mathbf{d}^\alpha(\norm{\nu} \delta_x, \nu).
\end{equation*}

Given a nonnegative measure $\nu$ on $\R^d$ and an integer $N > 1$, we define the optimal (branched) $\nu$-partition problem as:
\begin{equation}\label{pb:opt_partition_sets}
\inf \left \{ \sum_{i=1}^N \mathbf{C}^\alpha(\nu \mres \Omega_i) : \ (\Omega_i)_{1\leq i \leq N}\subseteq \spt \nu, \nu(\R^d\setminus \bigcup_i \Omega_i) = 0 \text{ and }  \nu(\Omega_i \cap \Omega_j) = 0\, (\forall i \neq j)\right\}.
\end{equation}
It may be equivalently written in terms of measures as:
\begin{equation}\label{pb:opt_partition_measures}
    \inf \left \{ \sum_{i=1}^N \mathbf{C}^\alpha(\nu_i) : \ (\nu_i)_{1\leq i \leq N}, \nu = \sum_{1\leq i\leq N} \nu_i \text{ and } \nu_i \perp \nu_j\, (\forall i \neq j)\right\}.
\end{equation}

For more classical transport costs, e.g. corresponding to Wasserstein distances $W_p$ ($p\geq 1$), it is straightforward to see that optimal quantization is equivalent to an optimal partition problem, where the optimal partitions are given by Voronoi diagrams associated with a finite set of points. With our branched transportation cost $\mathbf C^\alpha$ (corresponding to the distance $\mathbf d^\alpha$), the situation is a priori much more difficult, since there is no clear decomposition of the target space into regions associated with the atoms of a quantizer: on the contrary, in branched transport it is expected that several atoms are first collected together along a graph, then irrigate some part of the target measure, so that we cannot associate these irrigated points with a single atom. However, we have seen in \Cref{quantizers_disjointness} that such situations do not occur for \emph{(mass-)optimal} quantizers, resulting in the equivalence between the optimal quantization and optimal partition problems.

Note that the existence of minimizers for \labelcref{pb:opt_partition_measures} is not direct from functional analysis results, since the condition $\nu_i\perp\nu_j$ does not pass to weak limits of measures. To prove existence of solutions, we introduce the following relaxed optimal partition problem
\begin{equation}\label{pb:opt_partition_measures_relaxed}
    \inf \left \{ \sum_{i=1}^N \mathbf{C}^\alpha(\nu_i) : \ (\nu_i)_{1\leq i \leq N}, \nu = \sum_{1\leq i\leq N} \nu_i\right\}.
\end{equation}
We shall see below that the optimal quantization problem and the original and relaxed partition problems are equivalent (\Cref{thm:opoleq}), and obtain existence to \labelcref{pb:opt_partition_measures} in \Cref{existence_opt_partition}.

\begin{thm}[Optimal quantization $\simeq$ optimal partition]\label{thm:opoleq}
    Given a measure $\nu \in \measspace^+(\R^d)$, the minimal values of the optimal quantization problem \labelcref{pb:opt_quantization} and the optimal partition problem \labelcref{pb:opt_partition_measures}, as well as its relaxation \labelcref{pb:opt_partition_measures_relaxed}, are equal. Furthermore, when the optimal values are finite the minimizers of these problems are related as follows:
    \begin{enumerate}[(i)]
        \item\label{opt_from_quantization_to_partition} If $\mu_N=\sum_{i=1}^N m_i\delta_{x_i}$ is solution of \labelcref{pb:opt_quantization} with optimal traffic plan $\tplan P \in\tplanset(\mu_N,\nu)$ then the irrigated measures $\nu_i \coloneqq (e_\infty)_\sharp \tplan P^{x_i}$ for $1\leq i\leq N$ form an optimizer of \labelcref{pb:opt_partition_measures}.
        \item\label{opt_from_partition_to_quantization} If $(\nu_i)_{1\leq i\leq N}$ is an optimizer of \labelcref{pb:opt_partition_measures} and if for every $i$, $x_i \in \R^d$ and $\tplan P_i \in \tplanset(\norm{\nu_i}\delta_{x_i},\nu_i)$ are optimal, i.e.\ $\mathbf C^\alpha(\nu_i) = \mathbf d^\alpha(\norm{\nu_i}\delta_{x_i},\nu_i) = \mass^\alpha(\tplan P_i)$, then $\mu_N \coloneqq \sum_{i=1}^N \norm{\nu_i}\delta_{x_i}$ is an optimizer for \labelcref{pb:opt_quantization} and $\tplan P \coloneqq \sum_{i=1}^N \tplan P_i \in \tplanset(\mu_N,\nu)$ is an optimal traffic plan.
        \item\label{equivalence_partition_relaxed} The optimal partition problem \labelcref{pb:opt_partition_measures} and the relaxed problem \labelcref{pb:opt_partition_measures_relaxed} have the same minimizers and minimal value.
    \end{enumerate}
\end{thm}

\begin{proof}[Proof of \Cref{thm:opoleq}]
    Denote by $\ener{E}^\alpha_{\mathrm{p}}(\nu,N)$ and $\ener{E}^\alpha_{\mathrm{pr}}(\nu,N)$ the infima of \labelcref{pb:opt_partition_measures} and \labelcref{pb:opt_partition_measures_relaxed} respectively. Assuming that $\ener E^\alpha(\nu,N) < +\infty$, take $\mu = \sum_{i=1}^N m_i \delta_{x_i}$ a minimizer of \labelcref{pb:opt_quantization}, which exists by \Cref{existence_opt_quantization}, and an optimal traffic plan $\tplan P \in \tplanset(\mu,\nu)$. By \Cref{quantizers_disjointness} \labelcref{disjointness_irrigated_measures}, the irrigated measures $\nu_i = (e_\infty)_\sharp \tplan P^{x_i}$ are mutually singular and $\nu = \sum_{i=1}^N \nu_i$. In particular, $(\nu_i)_{1\leq i\leq N}$ is a competitor for \labelcref{pb:opt_partition_measures}. Besides, by \Cref{quantizers_disjointness} \labelcref{disjointness_plans} the traffic plans $\tplan P^{x_i} \in \tplanset(m_i \delta_{x_i},\nu_i)$ are disjoint and optimal, thus:
    \begin{equation}\label{bound_from_quantization_to_partition}
        \ener{E}^\alpha_{\mathrm{p}}(\nu,N) \leq \sum_{i=1}^N \mathbf C^\alpha(\nu_i) \leq \sum_{i=1}^N \mathbf d^\alpha(m_i \delta_{x_i},\nu_i) = \sum_{i=1}^N \mass^\alpha(\tplan P^{x_i}) = \mass^\alpha(\tplan P) = \ener{E}^\alpha(\nu,N).
    \end{equation}
    
    Vice versa, assuming $\ener E_{\mathrm{pr}}^\alpha(\nu,N) < +\infty$, take a $\eps$-minimizer $(\nu_i)_{1\leq i\leq N}$ of \labelcref{pb:opt_partition_measures_relaxed} for some fixed $\eps > 0$. We can form a competitor for \labelcref{pb:opt_quantization} by simply taking for each $\nu_i$ a point $x_i$ that is optimal, i.e.\ such that $\mathbf d^\alpha(\norm{\nu_i}\delta_{x_i},\nu_i) = \mathbf C^\alpha(\nu_i)$, and setting $\mu \coloneqq \sum_{i=1}^N \norm{\nu_i} \delta_{x_i}$. Moreover, taking for every $i\in \{1,\ldots,N\}$ an optimal traffic plan $\tplan P_i \in \tplanset(\norm{\nu_i} \delta_{x_i},\nu_i)$, the traffic plan $\tplan P \coloneqq \sum_{i=1}^N \tplan P_i $ belongs to $\tplanset(\mu,\nu)$ where $\mu$ has at most $N$ atoms, and therefore by subadditivity of the $\alpha$-mass
    \begin{equation}\label{bound_from_partition_to_quantization}
    \begin{aligned}
        \ener{E}^\alpha(\nu,N) \leq \mathbf d^\alpha(\mu,\nu) \leq \mass^\alpha(\tplan P) &\leq \sum_{i=1}^N \mass^\alpha(\tplan P_i)\\
        &= \sum_{i=1}^N \mathbf d^\alpha(\norm{\nu_i} \delta_{x_i},\nu_i) = \sum_{i=1}^N \mathbf C^\alpha(\nu_i) = \ener{E}^\alpha_{\mathrm{pr}}(\nu,N) + \eps.
    \end{aligned}
    \end{equation}
    Since $\eps$ is arbitrary, this shows $\ener{E}^\alpha(\nu,N) = \ener{E}^\alpha_{\mathrm{pr}}(\nu,N) = \ener{E}^\alpha_{\mathrm{p}}(\nu,N)$ and \labelcref{opt_from_quantization_to_partition} holds because of \labelcref{bound_from_quantization_to_partition} in the first part of the proof; it implies in particular existence for the optimal partition problem and its relaxation. Besides, taking now $(\nu_i)_{1\leq i\leq N}$ a minimizer of the relaxed partition problem \labelcref{pb:opt_partition_measures_relaxed} instead of an $\eps$-minimizer, and plugging it in the inequalities \labelcref{bound_from_partition_to_quantization} (now $\eps = 0$), shows that the quantizer $\mu$ built as above is optimal for $\nu$, and the traffic plan $\tplan P$ (also built as above) is optimal in $\tplanset(\mu,\nu)$. From \labelcref{opt_from_quantization_to_partition} we deduce that the $\nu_i$'s are actually mutually singular and $(\nu_i)_{1\leq i\leq N}$ is a minimizer of \labelcref{pb:opt_partition_measures}, which in turn implies \labelcref{equivalence_partition_relaxed}.
\end{proof}

As direct corollary of \Cref{existence_opt_quantization} and \Cref{thm:opoleq}, we obtain existence for the optimal partition problem.

\begin{corollary}\label{existence_opt_partition}
    For any finite positive measure  $\nu\in \measspace^+(\R^d)$ and any $N\in \N^*$, the optimal partition problem \labelcref{pb:opt_partition_measures} admits a solution.
\end{corollary}
\subsection{Asymptotic energy scaling and asymptotic constant}

From now on and in all the following results of the paper, we shall assume
\begin{equation*}
    \alpha \in (1-1/d,1),
\end{equation*}
where $d \in \N^*$ is the ambient dimension, and extensively use the exponent $\beta = \beta(\alpha,d)$ whose expression we recall:
\begin{equation*}
    \beta = 1+d\alpha-d\in (0,1).
\end{equation*}
We start by proving a general upper bound on the optimal quantization error by $N$ points.
\begin{lemma}\label{lemma:small_cost}
    Let $\nu \in \measspace^+_c(\R^d)$ be a finite measure and $\alpha \in (1-1/d,1)$. If $\nu$ is supported on a cube $Q$ of edge length $r$, it holds:
    \[
        \ener E^\alpha(\nu,N) \leq C_\BOT(\alpha,d) N^{-\beta/d} r \norm{\nu}^\alpha.
    \]
\end{lemma}
\begin{proof}
    Without loss of generality suppose that $Q =Q_1^d$. Take $n\in \N^*$ such that $n^d \leq N < (n+1)^d$ and divide the cube $Q$ into $n^d$ subcubes $\{Q_i\}_{1\leq i \leq n^d}$ of edge length $\frac 1n$. By concavity of $\R_+ \ni m \mapsto m^\alpha$ and \labelcref{upper_estimate_alpha_mass} we have
    \begin{align*}
    \ener E^\alpha(\nu,N) &\leq \ener E^\alpha(\nu,n^d) \leq \sum_{i=1}^{n^d} \ener E^\alpha(\nu \mres Q_i, 1)\leq \sum_{i=1}^{n^d} \frac{C_\BOT(\alpha,d)}2 n^{-1} \nu(Q_i)^\alpha\\
    &\leq \frac{C_\BOT(\alpha,d)}2 n^{-1} n^d \left(\frac{\norm \nu}{n^d}\right)^\alpha= \frac{C_\BOT(\alpha,d)}2 n^{-\beta} \norm{\nu}^\alpha\leq {C_\BOT(\alpha,d)} N^{-\beta/d} \norm{\nu}^\alpha.\qedhere
    \end{align*}
\end{proof}
We show that the optimal $N$-point quantization error of the unit cube behaves as some negative power of $N$ times a nontrivial constant $c_{\alpha,d}$, when the Lebesgue measure is $\alpha$-irrigable.

\begin{proposition}\label{p:limitcube}
If $\alpha \in (1-1/d, 1)$, then there exists a constant $c_{\alpha,d} \in (0,+\infty)$ such that
\begin{equation}\label{def_c_asymptotic}
    \lim_{N\to +\infty} N^{\beta/d} \ener{E}^\alpha(\lbm^d \mres [0,1]^d,N) = c_{\alpha,d}.
\end{equation}
\end{proposition}
The proof is based on a classical result on subadditive processes in ergodic theory (see e.g. \cite{lichtGloballocalSubadditiveErgodic2002}).
\begin{proof}
Define for every Borel set $A \subseteq \R^d$
\begin{equation*}\label{def:energy_set}
\ener S^\alpha(A) = \ener{E}^\alpha(\lbm^d \mres A, \lfloor \lbm^d(A) \rfloor).
\end{equation*}
Notice that for every $N \in \N^*$, by $1$-homogeneity in space and $\alpha$-homogeneity in mass of the $\alpha$-mass, we have
\begin{equation}\label{energy_unit_cube_to_big_cube}
    \begin{aligned}
    N^{\beta/d} \ener{E}^\alpha(\lbm^d \mres [0,1]^d,N) &= \frac 1 N N^{1/d + \alpha} \ener{E}^\alpha(\lbm^d \mres [0,1]^d,N)\\
    &=  \frac 1N {\ener{E}^\alpha(\lbm^d \mres [0,N^{1/d}]^d,N)} = \frac{\ener S^\alpha(Q_N)}N
\end{aligned}
\end{equation}
where $Q_N \coloneqq [0,N^{1/d}]^d$ is a cube of volume $N$. By \cite[Theorem~2.1]{lichtGloballocalSubadditiveErgodic2002}, any nonnegative subadditive translation-invariant function $\ener S$ defined on bounded Borel subsets of $\R^d$ satisfies
\[\lim_{N \to + \infty} \frac{\ener S(Q_N)}{N} = \inf_{n\in \N^*} \frac{\ener S([0,n)^d)}{n^d},\]
hence it suffices to show that $\ener S^\alpha$ is subadditive, the translation invariance being trivial. Subadditivity is a direct consequence of the subadditivity of $\mass^\alpha$ and the superadditivity of the integer part. Indeed, take $A_1, A_2$ two disjoint bounded Borel subsets of $\R^d$, then for any $i\in\{1,2\}$ an optimal quantizer $\mu_i$ of $\lbm^d \mres A_i$ with at most $\lfloor \lbm^d(A_i)\rfloor$ atoms, and an optimal traffic plan $\tplan P_i \in \tplanset(\mu_i, \lbm^d \mres A_i)$. Since $\tplan P_1 + \tplan P_2 \in \tplanset(\mu_1+\mu_2, \lbm^d \mres (A_1 \sqcup A_2))$ and the number $N$ of atoms of $\mu_1+\mu_2$ satisfies
\[N \leq \lfloor \lbm^d(A_1) \rfloor + \lfloor \lbm^d(A_1) \rfloor \leq \lfloor \lbm^d(A_1 \sqcup A_2) \rfloor,\]
we obtain
\begin{align*}
\ener S^\alpha(A_1 \sqcup A_2) &= \ener{E}^\alpha(\lbm^d \mres (A_1 \sqcup A_2), \lfloor \lbm^d(A_1 \sqcup A_2) \rfloor)\\
&\leq \mathbf d^\alpha(\mu_1+\mu_2,\lbm^d \mres (A_1 \sqcup A_2))\\
&\leq \mass^\alpha(\tplan P_1+\tplan P_2)\\
&\leq \mass^\alpha(\tplan P_1) + \mass^\alpha(\tplan P_2) = \ener S^\alpha(A_1) + \ener S^\alpha(A_2).
\end{align*}
We have thus proven the existence of the constant $c_{\alpha,d} \in [0,+\infty]$ of the statement. It is finite because by \Cref{lemma:small_cost},
\[
    c_{\alpha,d} = \inf_{n\in \N^*} \frac{\ener S^\alpha([0,n)^d)}{n^d} = \inf_{n\in \N^*} n^\beta \ener E^\alpha([0,1]^d,n^d) \leq  C_\BOT(\alpha,d)< +\infty.
\]
We next show that $c_{\alpha,d}$ is strictly positive. By \cite[Theorem~2.1]{pegonFractalShapeOptimization2019}, the constant
\begin{equation}\label{optimal_shape_constant}
    e_{\alpha,d} \coloneqq \inf \left\{ \mathbf d^\alpha(\delta_0,\rho) : \rho \in \probspace(\R^d), \rho \leq \lbm^d \right\}
\end{equation}
is a strictly positive real number. Let $\mu_N = \sum_{i=1}^N m_i \delta_{x_i}$ be an $N$-point optimal quantizer of $\lbm^d \mres [0,1]^d$ and $\tplan P \in \tplanset(\mu_N, \lbm^d \mres [0,1]^d)$ an optimal traffic plan. Using \Cref{quantizers_disjointness}, noticing that for every $i\in \{1,\ldots,N\}$, $\nu_i \coloneqq (e_\infty)_\sharp \tplan P^{x_i} \leq \lbm^d$, and using again the homogeneity properties of the $\alpha$-mass, we get
\begin{align*}
\ener{E}^\alpha(\lbm^d \mres [0,1]^d,N) &= \sum_{i=1}^N \mathbf d^\alpha(m_i \delta_{x_i},\nu_i)\\
&\geq \sum_{i=1}^N m_i^{\alpha + \frac 1 d} e_{\alpha,d} \geq N (1/N)^{\alpha +\frac 1d} e_{\alpha,d},
\end{align*}
where the last inequality is due to the convexity of $m \mapsto m^{\alpha+\frac 1d}$ (because $\alpha + \frac 1d > 1$). \rev{Using this in \labelcref{energy_unit_cube_to_big_cube}} implies that $c_{\alpha,d} \geq e_{\alpha,d} > 0$ and concludes the proof.\qedhere
\end{proof}

\section{\texorpdfstring{$\Gamma$-convergence}{Gamma-convergence} and Zador-type Theorem}\label{sec:gamma}

We are now going to provide an equivalent for the optimal quantization error of a compactly supported finite measure $\nu\ll \lbm^d$ as the number of points goes to infinity, analogous to the classical Zador's Theorem (see \cite[Theorem~6.2]{grafFoundationsQuantizationProbability2000}, or the original papers \cite{zadorDevelopmentEvaluationProcedures1963,zadorAsymptoticQuantizationError1982,bucklewMultidimensionalAsymptoticQuantization1982}), which states in particular that
\[
    \ener{E}^{W_2^2}(\nu) N^{-\frac 2d}\xto{N\to +\infty} \ener{E}^{W_2^2}(\lbm^d \mres [0,1]^d) \norm{\nu}_{\frac d{d+2}},
\]
where $\ener{E}^{W_2^2}(\nu) \coloneqq \inf \{ W_2(\mu_N,\nu)^2 : \rev{\#\spt\mu_N \leq N}\}$, $W_2$ is the $2$-Wasserstein distance over (compactly supported) probability measures and $\norm{f}_p$ for $p\geq 1$ is the $L^p$-norm of a Borel function $f$. We shall also be interested in the limit distribution of centers of $N$-point optimal quantizers $\mu_N$, i.e.\ in the weak limit of
\[\mu_N^\diamond \coloneqq \frac 1{\#\spt \mu_N} \sum_{\{x:\mu_N(\{x\}) > 0\}} \delta_x.\]
We tackle the two questions simultaneously by establishing a (stronger) $\Gamma$-convergence result, inspired from of \cite{bouchitteAsymptotiqueProblemePositionnement2002,bouchitteAsymptoticAnalysisClass2011}.

\subsection{A \texorpdfstring{$\Gamma$-convergence}{Gamma-convergence} result}

We establish a $\Gamma$-convergence result in the spirit of \cite{bouchitteAsymptotiqueProblemePositionnement2002}, in a form that is slighly more concise. We do not follow the extended approach of \cite{bouchitteAsymptoticAnalysisClass2011}, where the functionals depend on the quantizers $\mu_N$ and also on an extra variable that encodes the distributions of masses (as measures over $\R_+$), since the $\Gamma$-limit does not have a fully explicit expression in this case, and we are not able to derive useful information from it. Instead, the functionals $\ener F_N$ that we consider will depend solely on sets $\Sigma$ of $N$ points, embedded in the space of probability measures through their empirical measures $\frac 1N \sum_{s\in \Sigma} \delta_s$, leading to the definition
\[\mathscr X_N \coloneqq \left\{ \frac 1N \sum_{s\in \Sigma} \delta_s : \#\Sigma = N\right\} \qquad (\forall N \in \N^*).\]

We fix a compactly supported measure $\nu \in \measspace^+_c(\R^d)$ such that $\nu \ll \lbm^d$. We consider the sequence of functionals $\ener{F}_N : \probspace(\R^d) \to [0,+\infty]$ defined for every $N \in \N^*$ by
\begin{equation*}
\ener{F}_{N}(\rho) =
    \begin{dcases*}
    N^{\beta/d} \inf \{ \mathbf d^\alpha(\mu,\nu) : \spt \mu \subseteq \spt \rho\} & if $\rho \in \mathscr X_N$,\\
    +\infty& otherwise.
    \end{dcases*}
\end{equation*}
Determining the $\Gamma$-limit of the sequence $(\ener F_N)_{N\geq 1}$ amounts to seeking the least (asymptotic) energy to approximate, in the sense of branched optimal transport, the measure $\nu$ by $N$-point quantizers $\mu_N$ while prescribing the limit density of the centers $(\mu_N^\diamond)_{N\geq 1}$, which will correspond to the $\rho$ variable. We shall prove that the $\Gamma$-limit is the functional $\ener{F}_\infty : \probspace(\R^d) \to [0,+\infty]$ defined by
\begin{equation*}
    \ener{F}_\infty(\rho) = c_{\alpha,d} \int_{\R^d} \frac{\nu(x)^{\alpha}}{\rho_\ac(x)^{\frac \beta d}}\dd x
\end{equation*}
where we recall $\beta = 1+d\alpha-d$, $c_{\alpha,d}$ is the constant defined in \labelcref{def_c_asymptotic} and \rev{$\rho_\ac$ is the density of the absolutely continuous part of $\rho$ with respect to $\lbm^d$}.

\begin{thm}\label{thm_GCV}
Let $\nu \in \measspace^+_c(\R^d)$ such that $\nu \ll\lbm^d$ and $\alpha \in(1 - 1/d,1)$. The sequence of functionals $(\ener F_N)_{N\geq 1}$ $\Gamma$-converges to $\ener{F}_\infty$ as $N\to\infty$ with respect to the narrow convergence of probability measures.
\end{thm}
We are going to use the following lemmas.

\begin{lemma}\label{quantization_error_monotone_continuity}
    Let $\nu \in\measspace_c^+(\R^d)$ and $\alpha \in (1-1/d,1)$. 
    It holds:
    %
    \[\lim_{\delta\to 0} \liminf_{N\to+\infty} N^{\beta/d} \inf \{ \ener E^\alpha(\nu',N) : \nu'\leq \nu, \norm{\nu-\nu'}\leq \delta\} = \liminf_{N\to+\infty} N^{\beta/d} \ener E^\alpha(\nu,N).\]
    
\end{lemma}
\begin{proof}
    Suppose that $\nu$ is supported on a closed cube $Q$ of edge length $r > 0$. First of all, it is clear that
    \[\inf \{ \ener E^\alpha(\nu',N) : \nu'\leq \nu, \norm{\nu-\nu'}\leq \delta\} \leq \ener E^\alpha(\nu,N).\]
    for every $\delta >0$. Now, let us take a small $\lambda > 0$. For every $N$ large enough, by \Cref{lemma:small_cost} and subadditivity of the $\alpha$-mass, we have for every $\nu' \leq \nu$:
\begin{align*}
\ener E^\alpha(\nu,N + \lceil\lambda N \rceil) &\leq \ener E^\alpha(\nu',N) + \ener E^\alpha(\nu-\nu',\lceil\lambda N \rceil)\\
&\leq \ener E^\alpha(\nu',N) + C(\alpha,d) N^{-\beta/d} r \norm{\nu-\nu'}^\alpha \lambda^{-\beta/d},
\end{align*}
hence for every $\delta > 0$,
\begin{align*}
    \qquad&\liminf_{N\to+\infty} N^{\beta/d} \inf \{ \ener E^\alpha(\nu',N) : \nu'\leq \nu, \norm{\nu-\nu'}\leq \delta\}\\
    \geq& \liminf_{N\to+\infty} N^{\beta/d}\ener E^\alpha(\nu,N+\lceil\lambda N \rceil)  - C(\alpha,d) r \delta^\alpha \lambda^{-\beta/d}\\
    =& \rev{\liminf_{N\to+\infty} \left(\frac{N}{N+\lceil \lambda N\rceil}\right)^{\beta/d} (N+\lceil \lambda N\rceil)^{\beta/d}\ener E^\alpha(\nu,N+\lceil\lambda N \rceil)  - C(\alpha,d) r \delta^\alpha \lambda^{-\beta/d}}\\
    \geq& (1+\lambda)^{-\beta/d} \liminf_{N\to+\infty} N^{\beta/d}\ener E^\alpha(\nu,N)  - C(\alpha,d) r \delta^\alpha \lambda^{-\beta/d},
\end{align*}
\rev{the last inequality resulting from $N/(N+\lceil \lambda N\rceil) \xto{N\to+\infty} (1+\lambda)^{-1}$ and the fact that the inferior limit along $(N+\lceil \lambda N\rceil)_{N\in \N^*}$ is greater, as a subsequence, than along the whole sequence $(N)_{N\in \N^*}$.} Taking the liminf as $\delta \to 0$ and then $\lambda \to 0$ yields the result.
\end{proof}

\begin{lemma}\label{lem_block_approximation}
    Let $\nu \in \measspace^+(Q \subseteq \R^d)$ be a measure over a closed cube $Q$ of edge length $R$ such that $\nu \ll \lbm^d$, $\alpha \in (1-1/d,1)$ and $(A_i)_{1\leq i\leq I}$ be a $\lbm^d$-essential partition\footnote{Meaning $\lbm^d\left(Q\Delta \bigcup_{1\leq i \leq I} A_i\right) = 0$ and $\lbm^d(A_i\cap A_j) = 0$ for every $i\neq j$.} of $Q$ with $\diam(A_i)\le r$ for $1\leq i\leq I$. We set
    \[
        \nu' \coloneqq \sum_{i=1}^I m_i \lbm^d \mres A_i
    \]
    where $m_i \coloneqq \nu(A_i)/\lbm^d(A_i)$ if $\nu(A_i) > 0$ and $m_i \coloneqq 0$ otherwise. There is a constant $C_\BOT' = C_\BOT'(\alpha,d)$ depending only on $\alpha$ and $d$ such that
    \[
        \mathbf d^\alpha(\nu,\nu') \leq C_\BOT' R^{1-\beta} r^\beta \norm{\nu'-\nu}^\alpha .
    \]
\end{lemma}
\begin{proof}
We know from \cite[Proposition~0.1]{morelComparisonDistancesMeasures2007} (also \cite[Proposition~6.16]{bernotOptimalTransportationNetworks2009}) that
\[\mathbf d^\alpha(\mu^-,\mu^+) \leq C_\BOT' W_1(\mu^-,\mu^+)^\beta\leq C_\BOT'W_\infty(\mu^-,\mu^+)^\beta \]
for every probability measures $\mu^-$ and $\mu^+$ in $\probspace(Q_1^d)$ and some constant $C_\BOT' = C_\BOT'(\alpha,d)$, \rev{where $W_p$ denotes the $p$-Wasserstein distance for $p\in [1,+\infty]$}. Applying it to $\mu^- = \nu - \nu \wedge \nu'$ and $\mu^+ = \nu'- \nu\wedge \nu'$ after appropriate rescalings in mass $m \coloneqq \norm{\mu^-}$ and distance $R$, we obtain:
\begin{align*}
    &\qquad R^{-1}m^{-\alpha} \mathbf d^\alpha(\mu^-,\mu^+) \leq C_\BOT' R^{-\beta} W_\infty(\mu^-,\mu^+)^\beta\\
    &\implies  \mathbf d^\alpha(\nu,\nu') \leq \mathbf d^\alpha(\mu^-,\mu^+) \leq C_\BOT' W_\infty(\mu^-,\mu^+)^\beta m^\alpha R^{1-\beta}.
\end{align*}
By construction, $\nu$ and $\nu'$ have equal mass on each $A_i$, thus the same goes for $\mu^-$ and $\mu^+$, and since $\diam(A_i) \leq r$ it implies that $W_\infty(\mu^-,\mu^+) \leq r$. Since $\norm{\nu-\nu'} = 2 m$, we obtain the desired result.
\end{proof}

We are now in the position to prove \Cref{thm_GCV}.

\begin{proof}[Proof of \Cref{thm_GCV}]
We are going to prove successively the $\Gamma-\liminf$ and $\Gamma-\limsup$ inequality, i.e.
\begin{gather}
    \forall \rho \in\probspace(\R^d), \forall (\rho_N)_{N\in\N^*} \narrowto \rho, \quad \liminf_N \ener{F}_N(\rho_N) \geq \ener{F}_\infty(\rho),\label{Gamma_liminf}\\
    \forall \rho\in\probspace(\R^d), \exists (\rho_N)_{N\in\N^*} \narrowto \rho, \quad \limsup_N \ener{F}_N(\rho_N) \leq \ener{F}_\infty(\rho)\label{Gamma_limsup}.
\end{gather}

\paragraph{Proof of the $\Gamma-\liminf$ inequality \labelcref{Gamma_liminf}.} Let us take a sequence of probability measures $\rho_N \narrowto \rho$, assuming without loss of generality that $\liminf_N \ener{F}_N(\rho_N) < +\infty$. Up to taking a subsequence, we may assume that $\ener F_N(\rho_N)$ converges to $\liminf_N \ener F_N(\rho_N)$ and
\[C \coloneqq \sup_{N\in\N^*} \ener F_N(\rho_N) < +\infty.\]
In particular we know that for every $N\in \N^*$, $\rho_N = \frac 1N \sum_{s\in \Sigma_N} \delta_s$ for some set $\Sigma_N$ of cardinality $N$, and we take a mass-optimal quantizer $\mu_N$ of $\nu$ with respect to $\Sigma_N$, as well as an optimal traffic plan $\tplan P_N \in \tplanset(\mu_N,\nu)$, so that
\begin{align*}
    \ener F_N(\rho_N) = N^{\beta/d} \mathbf d^\alpha(\Sigma_N,\nu) &= N^{\beta/d}\mathbf d^\alpha(\mu_N,\nu) = N^{\beta/d}\mass^\alpha(\tplan P_N) = N^{\beta/d} \int_{\curvspace^d} Z_{\tplan P_N}\dd\tplan P_N(\gamma),
\end{align*}
recalling $Z_{\tplan P_N}$ is defined in \labelcref{landscape_precursor}. A standard strategy to show \labelcref{Gamma_liminf} is to express this energy as the total mass of some measure $e_N$, which converges up to subsequence to some measure $e$, then show a lower bound on $e$ and use the lower semicontinuity of the norm on $\measspace(\R^d)$. In our branched optimal transport setting, in order to follow this strategy we will have to resort to outer measures rather than measures. More precisely, we shall bound from below the energy $\ener F_N(\rho_N)$ by the total mass $E'_N(\R^d)$ of some suitable outer measure $E'_N$, that in some sense becomes a measure asymptotically as $N\to +\infty$.

Notice that
\begin{equation*}
    C N^{-\beta/d} \geq \int_{\curvspace^d} Z_{\tplan P_N}(\gamma) \dd\tplan P_N(\gamma) \geq \norm{\nu}^{\alpha-1} \int_{\curvspace^d} L(\gamma) \dd\tplan P_N(\gamma),
\end{equation*}
so that by Markov's inequality for every $M>0$:
\begin{equation}\label{plan_restriction_small_curves_1}
    \tplan P_N(\{\gamma: L(\gamma)\geq M N^{-\beta/d}\}) \leq \frac{C\norm{\nu}^{1-\alpha}}{M}.
\end{equation}
Consider an increasing sequence $M_N$ tending to $+\infty$ and such that $M_N N^{-\beta/d} \to 0$, and set
\begin{equation}\label{plan_restriction_small_curves_2}
    \tplan P_N' \coloneqq \tplan P_N \mres \Gamma_N \quad\text{where}\quad \Gamma_N \coloneqq \{\gamma:\ L(\gamma) < M_N N^{-\beta/d}\}.
\end{equation}
We define for every Borel set $A \subseteq \R^d$:
\begin{equation}\label{ENEprimeN}
    E_N(A) \coloneqq N^{\beta/d} \mass^\alpha(\tplan P_N \mres e_\infty^{-1}(A)),\qquad E'_N(A) \coloneqq N^{\beta/d} \mass^\alpha(\tplan P'_N \mres e_\infty^{-1}(A)).
\end{equation}
We remark that
\[E_N(\R^d) = \ener{F}_N(\rho_N)\]
and $E'_N$ (and also $E_N$) is an outer measure (being countably subadditive) and a priori is not a measure: it is possible that for two disjoint Borel sets $A_1, A_2$, the plans $\tplan P'_N \mres e_\infty^{-1}(A_1)$ and $\tplan P'_N \mres e_\infty^{-1}(A_2)$ are not disjoint. However, $E_N'$ becomes additive when $\dist(A_1,A_2) > 0$ and $N$ becomes large enough. Indeed, if $M_N N^{-\beta/d} \leq \frac 12 \dist(A_1,A_2)$, which is the case for $N$ large enough, then for every curve $\gamma_i \in \Gamma_N \cap e_\infty^{-1}(A_i)$, $i\in\{1,2\}$,
\[\gamma_1(\R) \cap \gamma_2(\R) = \emptyset,\]
which in turn implies that $\tplan P'_N \mres e_\infty^{-1}(A_1)$ and $\tplan P'_N \mres e_\infty^{-1}(A_2)$ are disjoint, and thus by \labelcref{sub_additivity}
\begin{align}
   \mass^\alpha(\tplan P'_N \mres e_\infty^{-1}(A_1\cup A_2)) &= \mass^\alpha(\tplan P'_N \mres e_\infty^{-1}(A_1)) + \mass^\alpha(\tplan P'_N \mres e_\infty^{-1}(A_2))\notag\\
   \shortintertext{i.e.}
   E'_N(A_1\cup A_2) &= E'_N(A_1) + E'_N(A_2).\label{energy_addivity}
\end{align}
Notice that this additivity property does not hold a priori for $E_N$, which was the point for restricting it and using $E'_N$ instead.

We know that $\nu$-a.e. point $x\in \spt \nu$ satisfies
\begin{equation}\label{lebesgue_point}
    \fint_{Q_\delta(x)} \abs{\nu(y)-\nu(x)} \dd y \xto{\delta \to 0} 0\quad\text{and}\quad \nu(x) \in (0,+\infty),
\end{equation}
where $Q_\delta(x)$ denotes the closed cube $x + \delta [-1/2,1/2]^d$. Fix such a point $x$, take $\delta > 0$ such that $\rho(\partial Q_\delta(x)) = 0$ (this is true for all but countably many $\delta$'s), and consider the slightly smaller $\delta' = \tau \delta$ for $\tau \in (0,1)$ (which we will send to $1$ later). We denote for every $N\in\N^*$
\begin{equation*}
    n_{N,\delta} \coloneqq \#(\Sigma_N \cap Q_\delta(x)),\qquad \nu_N \coloneqq (e_\infty)_\# \tplan P'_N,
\end{equation*}
and we define the $\delta'$-rescalings around $x$
\begin{gather*}
    \nu_{\delta'} \coloneqq \frac 1{\delta'^{d}}\left(y \mapsto \frac{y-x}{\delta'}\right)_\sharp \left(1\wedge \frac{\nu}{\nu(x)} \lbm^d \mres Q_{\delta'(x)}\right),\\
    \nu_{N,\delta'} \coloneqq \frac 1{\delta'^{d}}\left(y \mapsto \frac{y-x}{\delta'}\right)_\sharp \left(1\wedge \frac{\nu_N}{\nu(x)} \lbm^d \mres Q_{\delta'(x)}\right).
\end{gather*}
For $N$ large enough $Q_{\delta'}(x) + B_{M_N N^{-\beta/d}}(0) \subseteq Q_\delta(x)$ because $M_N N^{-\beta/d}$ converges to $0$, hence we have the lower bounds
\begin{align}
    N^{-\beta/d} E'_N(Q_{\delta'}(x)) &= \mass^\alpha(\tplan P_N' \mres e_\infty^{-1}(Q_{\delta'}(x)))\notag\\
    &\geq \mathbf d^\alpha((e_0)_\# (\tplan P'_N \mres e_\infty^{-1}(Q_{\delta'}(x))), \nu_N \mres Q_{\delta'}(x))\label{lbo2}\\
    &\geq \ener{E}^\alpha( \nu_N \mres Q_{\delta'}(x),n_{N,\delta})\label{lbo3}\\
    &\geq \ener E^\alpha (\nu_N\wedge \nu(x) \lbm^d \mres Q_{\delta'}(x), n_{N,\delta})\label{lbo4}\\
    &= \nu(x)^\alpha \delta'^{1+d\alpha} \ener E^\alpha(\nu_{N,\delta'},n_{N,\delta})\label{lbo5},
\end{align}
where \labelcref{lbo2} follows from the definition of $\mathbf d^\alpha$, \labelcref{lbo3} and \labelcref{lbo4} from the facts that the source measure of $\tplan P'_N \mres e_\infty^{-1}(Q_{\delta'}(x))$ is a submeasure of $\mu_N \mres Q_\delta(x)$ (thus has at most $n_{N,\delta}$ atoms) and that $\ener{E}^\alpha(\nu,n)$ is decreasing in $n$ and increasing in $\nu$.

For every $N\in \N^*$, $\nu_{N,\delta'}$ is a submeasure of $\nu_{\delta'}$ because $\nu_N \leq \nu$, and by \labelcref{plan_restriction_small_curves_1} and \labelcref{plan_restriction_small_curves_2} we know that $\norm{\nu_{\delta'}-\nu_{N,\delta'}} \leq C\norm{\nu}^{1-\alpha}/M_N \xto{N\to +\infty} 0$, thus multiplying \labelcref{lbo5} by $N^{\beta/d}$, passing to the liminf in $N$ and using \Cref{quantization_error_monotone_continuity} yields
\begin{align}
    \liminf_{N\to+\infty} E'_N(Q_\delta(x)) &\geq \nu(x)^\alpha (\delta\tau)^{1+d\alpha} \liminf_{N\to +\infty} \left(\frac{N}{n_{N,\delta}}\right)^{\beta/d} \ener{E}^\alpha (\nu_{N,\delta'}, n_{N,\delta}) n_{N,\delta}^{\beta/d}\nonumber \\
    &= \frac{(\delta\tau)^{1+d\alpha} \nu(x)^\alpha}{\rho(Q_\delta(x))^{\beta/d}} \liminf_{n\to+\infty} n^{\beta/d}\ener E^\alpha(\nu_{\delta'},n)
\end{align}
because
\[\frac{n_{N,\delta}}{N}  = \rho_N(Q_\delta(x))\xto{N\to +\infty} \rho(Q_\delta(x)),\]
since $\rho_N \narrowto \rho$ as $N\to+\infty$ and $\rho(\partial Q_\delta(x)) = 0$. Notice that $\nu_{\delta'} \leq \lbm^d \mres Q_1$ and by \labelcref{lebesgue_point} that $\norm{\nu_{\delta'}} \xto{\delta\to 0} 1 = \lbm^d(Q_1)$ hence $\norm{\nu_{\delta'}-\lbm^d\mres Q_1} \xto{\delta \to 0} 0$. We divide by $\delta^d$, pass to the limsup as $\delta \to 0$, then take $\tau \to 1$ recalling that $\delta' = \tau \delta$, and finally use \Cref{quantization_error_monotone_continuity} again, obtaining
\begin{equation}\label{local_energy_lower_bound}
\begin{aligned}
\limsup_{\delta \to 0} \frac{\liminf_N E'_N(Q_\delta(x))}{\lbm^d(Q_\delta(x))} &\geq \nu(x)^\alpha \limsup_{\delta \to 0} \frac{\delta^\beta}{\rho(Q_\delta(x))^{\beta/d}} \liminf_{n\to+\infty} n^{\beta/d} \ener E^\alpha(\lbm^d \mres [0,1]^d,n)\\
&= \frac{\nu(x)^\alpha}{\rho_\ac(x)^{\beta/d}} c_{\alpha,d},
\end{aligned}
\end{equation}
which holds for $\lbm^d$-a.e. (thus $\nu$-a.e.) $x$ by Radon--Nikodym Theorem.

Now, we conclude by applying a covering argument. For fixed $\eps \in (0,1)$ we consider the collection $\mathcal Q_{\eps}$ of cubes $Q_\delta(x)$, $\delta \in (0,1]$, $x\in \R^d$ such that
\begin{enumerate}[(i)]
\item\label{cover2} $\frac{(\eps^{-1}\wedge\nu(x))^\alpha}{(\eps \vee \rho_\ac(x))^{\beta/d}} \geq\fint_{Q_\delta(x)} \frac{(\eps^{-1}\wedge \nu)^\alpha}{(\eps \vee\rho_\ac)^{\beta/d}} - \eps$,
\item\label{cover3}  $\frac{\liminf_N E_N'(Q_\delta(x))}{\lbm^d(Q_\delta(x))} \geq c_{\alpha,d}  \frac{\nu(x)^\alpha}{\rho_\ac(x)^{\beta/d}} - \eps$.
\end{enumerate}
For any fixed $R > 0$, the set of cubes $\mathcal Q_\epsilon$ form a fine cover of a subset of $K_R \coloneqq \{x\in \R^d : \nu(x) > 0\} \cap Q_R(0)$ of full $\lbm^d$-measure because of \labelcref{local_energy_lower_bound} and the fact that for $\lbm^d$-a.e. $x\in K_R$ we have
\begin{equation*}
\lim_{\delta \to 0} \fint_{Q_\delta(x)} \frac{(\eps^{-1} \wedge\nu)^\alpha}{(\eps \vee \rho_\ac)^{\beta/d}} = \frac{(\eps^{-1} \wedge \nu(x))^\alpha}{(\eps\vee\rho_\ac(x))^{\beta/d}}
\end{equation*}
thanks to the Lebesgue--Besicovitch differentiation theorem applied to $\frac{(\eps^{-1}\wedge \nu)^\alpha}{(\eps \vee \nu_\ac)^{\beta/d}} \in L^1(K_R)$. Then, using the Vitali--Besicovitch covering theorem, there exists a countable family of disjoint cubes $(Q_{\delta_i}(x_i))_{i< I} \subseteq \mathcal Q_\eps$, $I\in \N \cup \{+\infty\}$, that cover $K_R$ up to a $\lbm^d$-negligible set.
Using above properties \labelcref{cover2,cover3} of the collection $\mathcal Q_\epsilon$, we get that for every $J< I$,
\begin{align*}
\liminf_{N\to+\infty} E'_N(K_R) &\geq \liminf_{N\to+\infty} E'_N \Bigl(\bigcup_{i\leq J} Q_{\delta_i}(x_i)\Bigr)\\
&\stackrel{\text{{\labelcref{energy_addivity}}}}{=} \liminf_{N\to+\infty} \sum_{i\leq J} E'_N(Q_{\delta_i}(x_i))\\
&\stackrel{\text{\labelcref{cover3}}}{\geq} \sum_{i\leq J} \left(c_{\alpha,d} \frac{\delta_i^d \nu(x_i)^\alpha}{\rho_\ac(x_i)^{\beta/d}} -\eps \delta_i^d\right)\\
&\stackrel{\text{\labelcref{cover2}}}{\geq} c_{\alpha,d} \sum_{i\leq J}\int_{Q_{\delta_i}(x_i)}  \frac{(\eps^{-1}\wedge\nu)^\alpha}{(\eps \vee\rho_\ac)^{\beta/d}} - (1+c_{\alpha,d})\eps\lbm^d(K_R+Q_1).
\end{align*}
Taking $J \to I$, then $\eps \to 0$ and $R\to +\infty$, by the Monotone Convergence Theorem we get
\begin{equation*}
\liminf_{N\to+\infty} \ener{F}_N(\rho_N) \geq \lim_{R\to+\infty}\liminf_{N\to+\infty} E'_N(K_R) \geq \lim_{R\to+\infty} c_{\alpha,d} \int_{K_R} \frac{\nu(x)^\alpha}{\rho_\ac(x)^{\beta/d}} \dd x = \ener{E}_\infty(\rho).
\end{equation*}

\paragraph{Proof of the $\Gamma-\limsup$ inequality  \labelcref{Gamma_limsup}.} Let us remark that the subsets
\[\mathcal A \coloneqq \probspace_c(\R^d) \cap L^1(\R^d), \quad \text{and}\quad \mathcal A' \coloneqq \{\rho \in \mathcal A : \spt\rho \text{ is some cube $Q$ and }\essinf_Q \rho > 0\},\]
where $\probspace_c(\R^d)$ denotes the set of compactly supported probability measures over $\R^d$, are dense in the $\ener F_\infty$ energy for the narrow convergence of measures. First of all, since $\nu$ has compact support it is clear\footnote{Simply consider the family $\frac{\rho \mres Q_{1/\eps}}{\norm{\rho \mres Q_{1/\eps}}}\xnarrowto{\eps \to 0}\rho$ for any given $\rho \in \probspace(\R^d)$.} that $\probspace_c(\R^d)$ is dense in energy. Then let us approximate any $\rho \in \probspace_c(\R^d)$ by measures in $\mathcal A$. Assume that it decomposes as $\rho = \rho_\ac \lbm^d + \rho_\sing$ where $\rho_\sing \perp \lbm^d$, and $\rho_\sing \neq 0$ (otherwise there is nothing to prove). We know that there exists $\rho_{\eps,\sing}$ for $\eps \in (0,1)$ which are absolutely continuous with respect to $\lbm^d \mres \Omega_\eps$, where $\Omega_\eps$ are nondecreasing subsets of $\R^d$ such that $\lbm^d(\Omega_\eps) \leq \eps$, and such that $\rho_{\sing,\eps} \xnarrowto{\eps\to 0} \rho_\sing$, and $\norm{\rho_{\eps,\sing}} = \norm{\rho_\sing}$. We set
\[\rho_\eps \coloneqq \rho_\ac\lbm^d + \rho_{\sing,\eps}.\]
Notice that $(\rho_\eps)_\ac \geq \rho_\ac$ so that
\[\ener F_\infty(\rho) \geq \ener F_\infty(\rho_\eps) \geq c_{\alpha,d}\int_{\R^d\setminus \Omega_\eps} \frac{\nu(x)^\alpha}{\rho_\ac(x)^{\beta/d}} \dd x,\]
and by the Monotone Convergence Theorem we get $\ener F_\infty(\rho_\eps) \xto{\eps \to 0} \ener F_\infty(\rho)$. Now to approximate any $\rho \in \mathcal A$ by measures in $\mathcal A'$, set for every $\eps > 0$
\[\rho_{\eps} \coloneqq \frac{\rho \vee \eps }{\norm{\rho \vee \eps}}.\]
It is clear that $\norm{\rho_{\eps}} = 1$ and $\rho_{\eps} \xnarrowto{\eps \to 0} \rho$, and by the Monotone Convergence Theorem again we get $\ener F_\infty(\rho_\eps)\xto{\eps\to 0} \ener F_\infty(\rho)$.
As a consequence, to prove the $\Gamma-\limsup$ inequality it suffices to find a recovery sequence for any given $\rho \in \mathcal A'$. Several steps are standard and inspired from \cite{bouchitteAsymptotiqueProblemePositionnement2002}, thus some of the constructions will be quickly done.

\begin{enumerate}[Step 1]
\item \textit{(Building approximation sequences.)} Take $\rho \in \mathcal A'$, whose support is by definition a cube $Q$ and assume that $\ener F_\infty(\rho) < +\infty$ (otherwise there is nothing to prove), which implies that $\spt \nu \subseteq Q$. Consider the collection of subcubes of $Q$ of edge length $\lambda N^{-1/d}$ given by
\[\{ Q_{N,i} : i\in I\} \coloneqq \{ \lambda N^{-1/d} (k+ Q_1) \subseteq Q : k \in \Z^d\},\]
 where $\lambda \geq 1$ is taken large (and will be sent to $+\infty$ later) and define piecewise constant approximations of $\nu$:
\begin{equation*}
    \nu_N \coloneqq \sum_{i\in I} \nu_{N,i}\lbm^d\mres Q_{N,i} \quad\text{where}\quad \nu_{N,i} \coloneqq \frac{\nu(Q_{N,i})}{\lbm^d(Q_{N,i})} \quad(\forall i \in I).
\end{equation*}
Notice that $\nu_N \to \nu$ in $L^1(\R^d)$. %
Let us build suitable $N$-point approximations of $\rho$ by putting the appropriate number of points $n_{N,i}$ in each cube $Q_{N,i}$. The number $n_{N,i}$ should be approximately given by
\[N \rho(Q_{N,i})= N (\lambda N^{-1/d})^d \rho_{N,i} = \lambda^d \rho_{N,i}\quad\text{where}\quad \rho_{N,i} \coloneqq \frac{\rho(Q_{N,i})}{\lbm^d(Q_{N,i})}\quad(\forall i\in I).\]
Since
\begin{gather*}
    \sum_{i\in I} \lfloor \lambda^d \rho_{N,i}\rfloor \leq  \sum_{i\in I}\lambda^d \rho_{N,i} = N \rho\Bigl(\bigcup_{i\in I} Q_{N,i}\Bigr) = N + o_{{N\to +\infty}}(N)\\
    \shortintertext{and}
    \# I \sim_{N\to +\infty} N^d/\lambda^d \geq N
\end{gather*}
for $N$ large enough, we may choose for every $i$ an integer $n_{N,i}$ such that
\begin{equation}\label{number_points_per_cube}
    \lfloor \lambda^d \rho_{N,i}\rfloor \leq n_{N,i} \leq \lfloor \lambda^d \rho_{N,i}\rfloor+1 \quad\text{and}\quad
    \sum_{i\in I} n_{N,i} = N.
\end{equation}
Notice that if we took $\lambda$ large enough, $\lambda^d \rho_{N,i} \geq \lambda^d \kappa \geq 1$ where $\kappa \coloneqq \essinf_K \rho$, so that we may assume $n_{N,i} \in \N^*$ for every $i\in I$.

For every $i\in I$, we take $\Sigma_{N,i}$ included in the interior of ${Q}_{N,i}$ as the support of a $n_{N_,i}$-point quantizer of $\lbm^d \mres Q_{N,i}$ which is $\delta_N$-optimal, where $\delta_N \coloneqq \left(\sum_{i\in I} \nu_{N,i}^\alpha\right)^{-1} N^{-\beta/d-1}$, i.e.
\begin{equation}\label{ineq_energy_cube}
\ener E^\alpha(\lbm^d \mres Q_{N,i},n_{N,i}) \leq \mathbf d^\alpha(\Sigma_{N,i},\lbm^d \mres Q_{N,i}) \leq \ener E^\alpha(\lbm^d \mres Q_{N,i},n_{N,i}) + \delta_N,
\end{equation}
and we eventually define
\[\Sigma_N \coloneqq \bigsqcup_{i\in I} \Sigma_{N,i}\quad\text{and}\quad\rho_N \coloneqq \frac 1N \sum_{s\in \Sigma_N} \delta_s.\]
We know that $\rho_N \xnarrowto{N\to+\infty} \rho$ because by \labelcref{number_points_per_cube}
\[\sup_{i\in I} \abs{\rho_N(Q_{N,i}) - \rho({Q}_{N,i})} \leq \frac 1N \sup_{i\in I} \abs{n_{N,i}-\lambda^d \rho_{N,i}} \leq \frac 1N \xto{N\to+\infty} 0.\]
By the triangle inequality and \labelcref{ineq_energy_cube} we find the following:
\begin{align}
    \ener{F}_N(\rho_N) &= N^{\beta/d}\mathbf d^\alpha(\Sigma_N,\nu)\\
    &\leq N^{\beta/d}\mathbf d^\alpha(\Sigma_N,\nu_N) + N^{\beta/d} \mathbf d^\alpha(\nu_N,\nu)\nonumber\\
    &= N^{\beta/d} \mathbf d^\alpha\Bigl(\bigcup_{i\in I} \Sigma_{N,i},\sum_{i\in I} \nu_{N,i} \lbm^d \mres Q_{N,i}\Bigr) + N^{\beta/d} \mathbf d^\alpha(\nu_N,\nu)\nonumber\\
    &\leq N^{\beta/d}\sum_{i\in I} \Bigl(\ener E^\alpha(\nu_{N,i}\lbm^d \mres Q_{N,i},n_{N,i}) + \nu_{N,i}^\alpha \delta_N\Bigr)  + N^{\beta/d} \mathbf d^\alpha(\nu_N,\nu)\nonumber\\
    &\leq N^{\beta/d}\sum_{i\in I}  \ener E^\alpha(\nu_{N,i} \lbm^d \mres Q_{N,i},n_{N,i})\label{eq:enupperbd1}\\
    &\qquad+ N^{\beta/d} \mathbf d^\alpha(\nu_N,\nu) + \frac 1 N.\label{eq:enupperbd2}
    \end{align}

\item \textit{(Bounding \labelcref{eq:enupperbd1}.)} 
We have for every $i\in I$
\begin{equation*}
\ener E^\alpha(\nu_{N,i}\lbm^d \mres Q_{N,i}, n_{N,i}) = \nu_{N,i}^\alpha (N^{-1/d} \lambda)^{1+d\alpha} \ener{E}^\alpha(\lbm^d\mres Q_1,n_{N,i}), 
\end{equation*}
and therefore, if we set $\tilde\rho_N \coloneqq \sum_{i\in I} \rho_{N,i} \one_{Q_{N,i}}$ and $X_N = \bigcup_{i\in I} Q_{N,i}$,
\begin{align*}
    &\qquad N^{\beta/d}\sum_{i\in I} \ener E^\alpha(\nu_{N,i}\lbm^d \mres Q_{N,i}, n_{N,i})\\
    &\leq \sum_{i\in I} N^{-1} \nu_{N,i}^\alpha \lambda^{1+d\alpha}\ener{E}^\alpha(\lbm^d\mres Q_1,\lfloor\lambda^d \rho_{N,i}\rfloor)\\
    &= \sum_{i\in I} \int_{Q_{N,i}} \nu_N(x)^\alpha \ener E^\alpha(\lbm^d\mres Q_1,\lfloor\lambda^d \tilde\rho_N(x)\rfloor) \lambda^\beta \dd x\\
    &= \int_{X_N} \frac{\nu_N(x)^\alpha}{\tilde\rho_N(x)^{\beta/d}} \ener E^\alpha(\lbm^d\mres Q_1,\lfloor\lambda^d \tilde \rho_N(x)\rfloor) (\lambda^d \tilde\rho_N(x))^{\beta/d} \dd x.
\end{align*}
Now note that $\tilde \rho_N \geq \kappa > 0$ for a.e. $x \in X_N$, so that
\[\ener E^\alpha(\lbm^d\mres Q_1,\lfloor\lambda^d \tilde \rho_N(x) \rfloor) (\lambda^d \tilde\rho_N(x))^{\beta/d} \leq \sup_{n\geq \lfloor \lambda^d \kappa\rfloor} \ener E^\alpha(\lbm^d\mres Q_1,n) (n+1)^{\beta/d} \leq c_{\alpha,d}(1+\eps(\lambda)),\]
where $\eps(\lambda) \xto{\lambda \to +\infty} 0$ by \Cref{p:limitcube}. Besides, $(\tilde \rho_N)$ and $(\nu_N)$ converge in $L^1(\R^d)$ to $\rho$ and $\nu$, respectively. All these measures are concentrated on $Q$, thus $\nu_N^\alpha \leq \one_Q + \nu_N$ for every $N$, and since $\nu_N^\alpha \to \nu^\alpha$ in measure and $(\one_Q+\nu_N)$ is equi-integrable, the Vitali convergence theorem yields the convergence of $(\nu_N^\alpha)$ to $\nu^\alpha$ in $L^1(\R^d)$ as well. Therefore by reverse Fatou's Lemma, taking the superior limit as $N\to +\infty$ then the limit $\lambda \to +\infty$ yields
\begin{equation*}\label{eq:enupperbd1sol}
    \limsup_{N\to +\infty} N^{\beta/d}\sum_{i\in I} \ener E^\alpha(\nu_{N,i}\lbm^d \mres Q_{N,i}, n_{N,i}) \leq c_{\alpha,d} \int_Q \frac{\nu(x)^\alpha}{\rho(x)^{\beta/d}} \dd x.
\end{equation*}

\item \textit{(Bounding \labelcref{eq:enupperbd2} and conclusion.)} 
We apply \Cref{lem_block_approximation} to the measures $\nu$ and $\nu' = \nu_N$
\[\mathbf d^\alpha(\nu,\nu_N) \leq C'_\BOT R^{1-\beta} (\lambda N^{-1/d})^\beta \norm{\nu-\nu_N}^\alpha,\]
so that 
\[N^{\beta/d} \mathbf d^\alpha(\nu,\nu_N) \leq C'_\BOT R^{1-\beta} \lambda^\beta \norm{\nu-\nu_N}^\alpha.\]
Taking the limit $N\to+\infty$, since $\nu_N \to \nu$ in $L^1$, we get
\begin{equation}\label{eq:enupperbd2sol}
    \limsup_{N\to+\infty} N^{\beta/d}\mathbf d^\alpha(\nu,\nu_N) = 0.
\end{equation}
By \labelcref{eq:enupperbd1sol} and \labelcref{eq:enupperbd2sol} we thus have
\begin{equation*}
    \limsup_{N\to+\infty} \ener {F}_N(\rho_N) \leq c_{\alpha,d}\int_\Omega \frac{\nu(x)^\alpha}{\rho(x)^{\beta/d}} \dd x = \ener{F}_\infty(\rho),
\end{equation*}
as desired. \qedhere
\end{enumerate}
\end{proof}
\begin{remark}\label{rmk:nocutoff}
    There are alternative approaches for the $\Gamma-\liminf$ part of the proof if we assume that the measure $\nu$ is $d$-Ahlfors regular (see \labelcref{eq:ahlfors}), since we may use the Hölder regularity of the landscape function and its consequences (in particular the bound on the diameter of basins in terms of their masses) that are established in \Cref{sec:landscape}. Indeed, we may use directly the outer measures $E_N$ defined for every Borel set $A$ by
    \[
        E_N(A)\coloneqq N^{\beta/d} \mass^\alpha(\tplan P_N \mres e_\infty^{-1}(A)),
    \]
    rather than the restrictions $E'_N$, or even use the measures defined by
    \[e_N(A) \coloneqq \int_A z_{\tplan P_N} \dd\mu.\]
    The relevance of restricting the plans (and thus of passing from $E_N$ to $E'_N$) is that we can then guarantee that $E'_N$ satisfies additivity for sets at positive distance and $N$ large enough. But one may check that this property holds directly for $E_N$ thanks to \Cref{corr:basindisjoint} and \Cref{lem:volbas}. It is even easier with  $e_N$ which is by definition a measure, although in this case we need to adapt the series of inequalities \labelcref{lbo2}-\labelcref{lbo4} which give the lower bound. 
    
    Also note that similar considerations using H\"older regularity of the landscape function under Ahlfors regularity hypotheses may also apply to \Cref{equidistribution_macro_scale} to replace the outer measure $E'_N$ by $E_N$ or $e_N$ in the statement on the equi-distribution of energy at the macroscopic scale.
\end{remark}
\subsection{Asymptotics of the quantization error and support of optimal quantizers}

From the $\Gamma$-convergence established in the previous subsection, we may obtain the asymptotics of the optimal quantization error (a branched optimal transport variant of Zador's theorem) and the limit density of the centers of optimal quantizers, i.e.\ to establish \Cref{thm:zador}.

\begin{proof}[Proof of \Cref{thm:zador}]
Take for every $N \in \N^*$ a $N$-point optimal quantizer $\mu_N$ of $\nu$. Since $\nu$ is concentrated on some closed cube $Q$, it is straightforward to see that by optimality all the $\mu_N$'s must be concentrated on $Q$ as well. Thus the sequence of probability measures $(\mu_N^\diamond)$ converges narrowly, up to a subsequence, to a measure $\rho$. Since for every $N$,
\[\ener F_N(\mu_N^\diamond) = \inf \ener F_N,\]
 by \Cref{thm_GCV} the measure $\rho$ minimizes the $\Gamma$-limit $\ener{F}_\infty$ and
\[\lim_{N\to+\infty} N^{\beta/d}\ener{E}^\alpha(\nu,N) = \lim_{N\to +\infty} \ener F_N(\mu_N^\diamond) = \ener{F}_\infty(\rho) = c_{\alpha,d} \int_{\R^d} \frac{\nu(x)^\alpha}{\rho_\ac(x)^{\alpha + \frac 1d -1}} \dd x.\]
As a consequence of minimality, $\rho$ is absolutely continuous with respect to $\lbm^d$ and the Euler-Lagrange equation can be written as
\[\nu(x)^\alpha =  (M\rho(x))^{\alpha + \frac 1d}\]
for $\lbm^d$-a.e. $x\in Q$, for a constant $M$ which is given by
\[M = M_{\alpha,d}(\nu) \coloneqq \int_{\R^d} \nu(x)^{\frac{\alpha}{\alpha + \frac 1d}} \dd x.\]
In particular $\rho = M_{\alpha,d}(\nu)^{-1} \nu^{\frac{\alpha}{\alpha+\frac 1d}}$ and
\[\lim_{N\to+\infty} N^{\beta/d} \ener{E}^\alpha(\nu,N) = c_{\alpha,d} M_{\alpha,d}(\nu)^{\alpha +\frac 1d} = c_{\alpha,d}\left(\int_{\R^d} \nu(x)^{\frac \alpha{\alpha + \frac 1d}}\dd x\right)^{\alpha + \frac 1d}.\qedhere\]
\end{proof}

\subsection{Equidistribution results at the macroscopic scale}

To understand uniformizing features at the macroscopic scale, we deduce convergence of measures (or outer measures) of interest from the $\Gamma$-convergence result we have established and its proof. 
\begin{proposition}\label{equidistribution_macro_scale}
    Let $(\mu_N)_{N\in \N^*}$ be a sequence of $N$-point optimal quantizers of $\nu$. 
    \begin{enumerate}[ (A) ]
        \item\label{uniformity_empirical_measures} The empirical measures converge as follows:
        \[ 
            \mu_N^\diamond \narrowto M_{\alpha,d}(\nu)^{-1} \nu^{\frac\alpha{\alpha+\frac 1d}} \quad\text{where}\quad M_{\alpha,d}(\nu) = \int_{\R^d} \nu(x)^{\frac{\alpha}{\alpha + \frac 1d}} \dd x.
        \]
        In particular if $\nu = \lbm^d \mres X$ for some Borel set $X$ satisfying $\lbm^d(X) = 1$, we obtain
        \[
            \frac 1N \# (\spt \mu_N \cap B) \to \lbm^d(B),
        \]
        for every Borel set $B\subseteq X$ such that $\lbm^d(\partial B) = 0$.
        \item\label{uniformity_energy} The energy outer measures $(E'_N)$ defined in \labelcref{ENEprimeN} converge in the following sense:
        \[
            \lim_{N\to+\infty} E_N'(B) = c_{\alpha,d} M_{\alpha,d}(\nu)^{\alpha+\frac 1d-1} \int_{B} \nu(x)^{\frac{\alpha}{\alpha + \frac 1d}} \dd x
        \]
        for every Borel set $B$ such that $\lbm^d(\partial B) = 0$. In particular if $\nu = \lbm^d \mres X$ for some Borel set $X$ satisfying $\lbm^d(X) >0$ then
        \[
            \lim_{N\to +\infty} E_N'(B) = c_{\alpha,d} M_{\alpha,d}(\nu)^{\alpha+\frac 1d-1} \lbm^d(B).
        \]
    \end{enumerate}
\end{proposition}

\begin{proof}
The first item \labelcref{uniformity_empirical_measures} is a direct consequence of \Cref{thm:zador}. For \labelcref{uniformity_energy}, we follow the proof of the $\Gamma-\liminf$ inequality in \Cref{thm_GCV} and apply the covering argument to the subcollection $\mathcal Q_\eps' \subseteq \mathcal Q_\eps$ of cubes $Q_\delta(x)$ which are included in a given open subset $\Omega \subseteq \R^d$. We therefore get
\[
    \liminf_{N\to+\infty} E'_N(\Omega) \geq c_{\alpha,d} M_{\alpha,d}(\nu)^{\alpha+\frac 1d-1} \int_{\Omega} \nu(x)^{\frac{\alpha}{\alpha + \frac 1d}} \dd x.
\]
If $B$ is a Borel set with $\lbm^d(\partial B) = 0$, we may apply the above inequality to $B_{<\eps} \coloneqq \{x : d(x,B^c) > \eps\}$ and $B_{>\eps} = \{ x : d(x,B) > \eps\}$, and using the asymptotic additivity of $E'_N$ for the sets $B$ and $B_{>\eps}$ which are at positive distance, we obtain
\begin{align*}
    c_{\alpha,d} M_{\alpha,d}(\nu)^{\alpha+\frac 1d-1} \int_{B_{<\eps}} \nu(x)^{\frac{\alpha}{\alpha + \frac 1d}}\dd x &\leq \liminf_{N\to+\infty} E_N'(B_{<\eps})\leq \liminf_{N\to+\infty} E_N'(B) \leq \limsup_{N\to+\infty} E_N'(B)\\
    &= \limsup_{N\to+\infty} E'_N (B\cup B_{>\eps}) - E'_N(B_{>\eps})\\
    &\leq \limsup_{N\to+\infty} E_N(\R^d) - \liminf_{N\to+\infty} E'_N(B_{>\eps})\\
    &\leq \begin{multlined}[t]
    c_{\alpha,d} M_{\alpha,d}(\nu)^{\alpha+\frac 1d-1} \int_{\R^d} \nu(x)^{\frac{\alpha}{\alpha + \frac 1d}} \dd x\\
    - c_{\alpha,d} M_{\alpha,d}(\nu)^{\alpha+\frac 1d -1} \int_{B_{>\eps}} \nu(x)^{\frac{\alpha}{\alpha + \frac 1d}} \dd x
    \end{multlined}\\
    &= c_{\alpha,d} M_{\alpha,d}(\nu)^{\alpha+\frac 1d-1} \int_{\R^d\setminus B_{>\eps}} \nu(x)^{\frac{\alpha}{\alpha + \frac 1d}} \dd x.
\end{align*}
Taking the limit $\eps\to 0$, we get the desired result of \labelcref{uniformity_energy}.
\end{proof}

\section{Landscape function for mass-optimal quantizers}\label{sec:landscape}

This section is devoted to the landscape function, its definition and Hölder regularity. We stress that the classical definition of landscape function from \cite{santambrogioOptimalChannelNetworks2007}, recalled in \Cref{sec:bot}, is only given in the case of a single source $\mu = m\delta_x$ and, as already said, an optimal traffic plan with several sources may in general not decompose disjointly according to its sources. This poses a serious issue to define and study the landscape function in such a case. An attempt at defining the landscape function for several sources (even in a more general setting) has been made in \cite[Chapter~4]{pegonBranchedTransportFractal2017}, but the construction is quite technical and the Hölder constant computed there actually explodes when the number of sources tends to infinity. However, in the case of optimal quantizers or even mass-optimal quantizers, the disjointness result established in \Cref{quantizers_disjointness} allows us to give a simple ad hoc definition of landscape function, and, following the approach of \cite{santambrogioOptimalChannelNetworks2007}, we are able to show its Hölder regularity with a Hölder constant that is \emph{uniform in the number of sources}, a crucial information to establish the uniform regularity properties in \Cref{sec:uniform}
.

\subsection{Uniform Hölder regularity}

Our main result is the following:

\begin{thm}[extended version of \Cref{thm:holderz-intro}]\label{thm:holderz}
    Let $\alpha \in (1-1/d,1)$ and $\nu \in \measspace_c^+(\R^d)$ be a measure which is $d$-Ahlfors regular with constants $0 < c_A \leq C_A$, i.e.
    \begin{equation}\label{eq:ahlfors}
        c_A r^d \leq \nu(B_r(x)) \leq C_A r^d \quad (\forall x\in \spt \nu, \,\forall r\leq \diam(\spt \nu)),
    \end{equation}
    and let $\tplan P \in \tplanset(\mu,\nu)$ be an optimal traffic plan where $\mu = \sum_{i=1}^N m_i \delta_{x_i}$ is a $N$-point mass-optimal quantizer of $\nu$ with respect to $\{x_i\}_{1\leq i \leq N}$. There exists a unique function $z_{\tplan P} : \spt \nu \to \R_+$ that we call \emph{landscape function associated with $\tplan P$} satisfying:
    \begin{enumerate}[(i)]
    \item\label{item_landscape_well-defined} for every $i \in \{1, \ldots, N\}$, $z_{\tplan P^{x_i}} = z_{\tplan P}$ everywhere on $\bas(\tplan P,x_i)$;
    \item\label{item_landscape_holder} $z_{\tplan P}$ is $\beta$-Hölder continuous where we recall $\beta = 1+d\alpha-d \in (0,1)$, with a Hölder constant smaller than a constant $C_H = C_H(c_A,C_A,\alpha,d)$.
    \end{enumerate}
\end{thm}
Before giving a proof, it may be of interest to note that the above points \labelcref{item_landscape_well-defined}, \labelcref{item_landscape_holder} have the following direct consequences:
\begin{itemize}
\item Each landscape function $z_{\tplan P^{x_i}}$ is Hölder continuous, a fact which does not follow from the Hölder continuity results of \cite{santambrogioOptimalChannelNetworks2007,brancoliniHolderRegularityLandscape2011} because we do not know that the target measures $\nu_i = (e_\infty)_\sharp \tplan P^{x_i}$ are Ahlfors-regular (even though $\nu$ is);
\item the functions $z_{\tplan P^{x_i}}$ \enquote{glue} continuously, i.e.\ there is no discontinuity at the interfaces;
\item the Hölder constant of the landscape function $z_{\tplan P}$ defined by the above gluing is bounded above by a constant \emph{which does not depend on $N$}.
\end{itemize}

\begin{proof}[Proof of \Cref{thm:holderz}]
    Let us start by setting a candidate landscape function which is uniquely defined $\nu$-almost everywhere on $\spt \nu$. The measures $\nu_i \coloneqq (e_\infty)_\sharp \tplan P^{x_i}$ are mutually singular and sum to $\nu$ thanks to \Cref{quantizers_disjointness}, thus we may define a Borel function $z : \spt \nu \to \R_+$ such that for every $i\in\{1,\ldots,N\}$:
    \[z = z_{\tplan P^{x_i}}\quad \nu_i\text{-almost everywhere}.\]
    
    Let us show that $z$ admits a Hölder continuous representative through Campanato estimates, following the strategy of \cite{santambrogioOptimalChannelNetworks2007}. Take a point $x \in \spt \nu$. For every $r\in (0,2\diam(\spt \nu)]$ we denote by $z_r(x) \coloneqq \fint_{B_r(x)} z\dd\nu$ the mean of $z$ on $B_r(x)$, and by $\bar z_r(x)$ the central median of $z$ on $B_r(x)$ with respect to $\nu$, defined as the midpoint of the interval of values $\ell \in \R_+$ such that $B_r(x)$ may be partitioned into two subsets $A \sqcup B = B_r(x)$ with equal mass, i.e.\ $\nu(A) = \nu(B) = \nu(B_r(x))/2$, and such that $z \geq \ell$ on $A$ and $z\leq \ell$ on $B$. Consider two such sets $A,B$ for the central median $\ell = \bar z_r(x)$ and define the following variation of $\nu$:
    \begin{gather*}
        \tilde\nu \coloneqq \nu - \nu \mres A + \nu \mres B = \sum_{i=1}^N \tilde\nu_i
        \shortintertext{where for every $i\in \{1,\ldots,N\}$,}
        \tilde\nu_i \coloneqq \nu_i - \nu_i \mres A + \nu_i \mres B.
    \end{gather*}
    By \Cref{quantizers_disjointness} again, we know that the $\tplan P^{x_i}$'s are disjoint and thus optimal traffic plans with single source $x_i$, thus we may use the first variation inequality \labelcref{first_variation_distance} of \Cref{properties_landscape} for every $i\in \{1,\ldots,N\}$ to obtain
    \begin{equation}\label{competitor_upper_bound_i}
    \begin{aligned}
        \mathbf d^\alpha(\norm{\tilde\nu_i} \delta_{x_i}, \tilde\nu_i) &\leq \mathbf d^\alpha(m_i\delta_{x_i}, \nu_i) + \alpha  \left(\int_B z_{\tplan P^{x_i}}\dd\nu_i - \int_A z_{\tplan P^{x_i}} \dd\nu_i\right)\\
        &= \mass^\alpha(\tplan P^{x_i}) + \alpha  \left(\int_B z\dd\nu_i - \int_A z \dd\nu_i\right).
        \end{aligned}
    \end{equation}
    We set $\tilde{\tplan P} \coloneqq \sum_{i=1}^N \tilde{\tplan P}_i$ where $\tilde{\tplan P}_i \in \tplanset(\norm{\tilde\nu_i} \delta_{x_i},\tilde\nu_i)$ is an optimal traffic plan for every $i \in \{1,\ldots, N\}$. Summing \labelcref{competitor_upper_bound_i} over $i$, using the subadditivity of the $\alpha$-mass and the disjointness of the $\tplan P^{x_i}$'s together with \labelcref{sub_additivity} yields
    \begin{align}
        \mass^\alpha(\tilde{\tplan P}) \leq \sum_{i=1}^N \mass^\alpha(\tilde{\tplan P}_i) &= \sum_{i=1}^N \mathbf d^\alpha(\norm{\tilde\nu_i} \delta_{x_i}, \tilde\nu_i)\nonumber\\
        &\leq \sum_{i=1}^N \left(\mass^\alpha(\tplan P^{x_i}) + \alpha \left(\int_B z\dd\nu_i - \int_A z \dd\nu_i\right)\right)\nonumber\\
        &=\mass^\alpha(\tplan P) + \alpha  \left(\int_B z\dd\nu - \int_A z \dd\nu\right).\label{competitor_upper_bound}
    \end{align}
    Notice that $\tilde{\tplan P} \in \tplanset(\tilde \mu,\tilde \nu)$ where $\tilde \mu \coloneqq \sum_{i=1}^N \norm{\tilde \nu_i} \delta_{x_i}$. Take an optimal traffic plan $\tplan Q \in \tplanset(\tilde \nu,\nu)$ and consider a concatenation
    \[
        \tplan P' \in \tilde{\tplan P} : \tplan Q \subseteq \tplanset(\tilde\mu,\nu),
    \]
    which is defined thanks to \Cref{concatenation} \labelcref{concat_exist}. Since $\spt(\tilde\nu-\nu) \subseteq \bar B_r(x)$ and $\norm{\tilde \nu-\nu} = \nu(B_r(x)) \leq C_A r^d$, by \labelcref{concat_alpha_mass} and the branched transport upper estimate \labelcref{upper_estimate_alpha_mass} we have
    \begin{equation}\label{competitor_lower_bound}
        \mass^\alpha(\tplan P') \leq \mass^\alpha(\tilde{\tplan P}) + \mass^\alpha(\tplan Q) \leq \mass^\alpha(\tilde{\tplan P}) + C_\BOT \ 2r\ (C_A r^d)^\alpha.
    \end{equation}
    Now we remark that $\tilde\mu$ is still supported on $\{x_i : 1\leq i\leq N\}$ and $\mu$ is a mass-optimal quantizer of $\nu$ with respect to the $x_i$'s so that $\mass^\alpha(\tplan P')$ is greater than $\mass^\alpha(\tplan P)$, thus by combining \labelcref{competitor_lower_bound} and \labelcref{competitor_upper_bound}
    \begin{align*}
        \mass^\alpha(\tplan P) \leq \mass^\alpha(\tplan P')&\leq \mass^\alpha(\tilde{\tplan P}) + 2C_\BOT C_A^\alpha r^{1+d\alpha}\\
        &\leq \mass^\alpha(\tplan P) + \alpha\left(\int_B z\dd\nu- \int_A z\dd\nu\right) + 2 C_\BOT C_A^\alpha r^{1+d\alpha}.
    \end{align*}
    This implies that
    \[
        0 \leq \alpha\left(\int_B z\dd\nu- \int_A z\dd\nu\right) + 2 C_\BOT C_A^\alpha r^{1+d\alpha},
    \]
    hence
    \begin{align}
        \int_{B_r(x)} \abs{z-\bar z_r(x)}\dd\nu &= \int_A z\dd\nu - \int_B z\dd\nu \leq 2\alpha^{-1} C_\BOT C_A^\alpha r^{1+d\alpha}\nonumber\\
        \shortintertext{and finally}
        \int_{B_r(x)} \abs{z-z_r(x)}\dd\nu &\leq \int_{B_r(x)} \abs{z-\bar z_r(x)}\dd\nu + \nu(B_r(x)) \abs{z_r(x)-\bar z_r(x)}\nonumber\\
        &\leq 2 \int_{B_r(x)} \abs{z-\bar z_r(x)}\dd\nu \leq 4\alpha^{-1} C_\BOT C_A^\alpha r^{1+d\alpha}.     \label{campanato_estimate_1}
    \end{align}
    We now use Campanato estimates: for every $x\in \spt \nu$, $r\leq 2\diam(\spt \nu)$ and $r'\in [r/2,r]$,
    \begin{equation}\label{campanato_estimate_2}
    \begin{aligned}
        \abs{z_r(x)-z_{r'}(x)}&\leq \fint_{B_{r'}(x)} \abs{z-z_r(x)}\dd\nu\\
        &\leq \frac 1{\nu(B_{r'}(x))} \int_{B_r(x)} \abs{z-z_r(x)}\dd\nu\leq \frac{4\alpha^{-1} C_\BOT C_A^\alpha r^{1+d\alpha}}{c_A (r/2)^d} \leq C r^\beta,
    \end{aligned}
    \end{equation}
    where we have set $C\coloneqq\frac{2^{d+2} C_\BOT C_A^\alpha}{\alpha c_A} $, and as before $\beta = 1 + d\alpha-d \in (0,1)$. Applying \labelcref{campanato_estimate_2} to radii $r 2^{-n}, r2^{-n-1}$ for $n\in \N$, we deduce that $(z_{r 2^{-n}})_{n\in \N}$ is a Cauchy sequence, which in turn implies (using \labelcref{campanato_estimate_2} again) that the following limit exists for every $x \in \spt \nu$:
    \[
        z_{\tplan P}(x) \coloneqq  \lim_{r\to 0} z_r(x) = \lim_{r\to 0} \fint_{B_r(x)} z \dd\nu.
    \]
    By triangle inequality \labelcref{campanato_estimate_2} yields
    \begin{gather*}
        \abs{z_r(x)-z_{\tplan P}(x)} \leq \sum_{n=0}^{+\infty} \abs{z_{r 2^{-n}}(x)-z_{r 2^{-(n+1)}}(x)} \leq \frac{C r^\beta}{1-2^{-\beta}},\\
        \shortintertext{and combining with  \labelcref{campanato_estimate_1} we get}
        \fint_{B_r(x)} \abs{z-z_{\tplan P}(x)}\dd\nu \leq \frac{2}{1-2^{-\beta}} C r^\beta.
    \end{gather*}
    Finally, take $x,y \in \spt\nu$ such that $r \coloneqq \abs{y-x}$ and use the last two inequalities to get:
    \begin{align*}
        \abs{z_{\tplan P}(y)-z_{\tplan P}(x)} &\leq \abs{z_{\tplan P}(y)-z_r(y)}+\abs{z_r(y)-z_{\tplan P}(x)}\\
        &\leq \frac{C r^\beta}{1-2^{-\beta}} + \fint_{B_r(y)}  \abs{z-z_{\tplan P}(x)}\dd \nu\\
        &\leq \frac{C r^\beta}{1-2^{-\beta}} + \frac{\nu(B_{2r}(x))}{\nu(B_r(y))} \fint_{B_{2r}(x)} \abs{z-z_{\tplan P}(x)}\dd\nu \leq \left(\frac{1+2^{d+1} (C_A/c_A)}{1-2^{-\beta}}\right)Cr^\beta.
    \end{align*}
    As a consequence, we get \labelcref{item_landscape_holder} with
    \[C_H \coloneqq \frac{2^{2(d+2)}C_\BOT C_A^{1+\alpha}}{(1-2^{-(1+d\alpha -d)})\alpha c_A^2}.\]
    
    Let us now prove \labelcref{item_landscape_well-defined}. Since $\nu$-a.e. point of $\spt\nu$ is a Lebesgue point of $z$ (with respect to $\nu$), we know that $z_{\tplan P} = z$ $\nu$-a.e. thus $z_{\tplan P} = z_{\tplan P^{x_i}}$ $\nu_i$-a.e., but since $z_{\tplan P^{x_i}}$ is lower semicontinuous and $z_{\tplan P}$ is continuous on $\spt \nu_i$, we have $z_{\tplan P^{x_i}} \leq z_{\tplan P}$ everywhere on $\bas(\tplan P, x_i) = \spt \nu_i$. Let us show that we actually have equality. Given $x \in \spt \nu_i$ such that $z_{\tplan P^{x_i}}(x) < \infty$ (otherwise there is nothing to prove), consider a $\tplan P^{x_i}$-good curve $\gamma_i$ from $x_i$ to $x$. Fix $r\leq \diam(\spt \nu)$, take an optimal traffic plan $\tplan Q \in \tplanset(\nu(B_r(x)) \delta_x, \nu\mres B_r(x))$ and by \Cref{concatenation} \labelcref{concat_exist} take a concatenation $\tplan P'\in \tplanset(\nu(B_r(x)) \delta_{x_i},\nu \mres B_r(x))$ given by
    \[
        \tplan P' \in \left(\nu(B_r(x)) \delta_{\gamma_i}\right) : \tplan Q.
    \]
    We build the competitor
    \[
        \tilde{\tplan P} \coloneqq \tplan P - \tplan P \mres \{\gamma(\infty) \in B_r(x)\} + \tplan P'
    \]
    which belongs to $\tplanset(\tilde\mu,\nu)$ for some measure $\tilde \mu$ which is still supported on $\{x_j, 1\leq j\leq N\}$. Using this time the first variation inequality \labelcref{first_variation_mass_landscape} of \Cref{properties_landscape} (applied to each $\tplan P^{x_j}$ with variation $\tilde{\tplan P}^{x_j}$, $j\in \{1,\ldots, N\}$), the subadditivity of the $\alpha$-mass, the mass-optimality of $\mu$, and the disjointness of the $\tplan P^{x_i}$'s together with \labelcref{sub_additivity}, we must have:
    \begin{align*}
        \mass^\alpha(\tplan P) \leq \mass^\alpha(\tilde{\tplan P}) &\leq \mass^\alpha(\tilde{\tplan P}^{x_i}) + \sum_{j\neq i} \mass^\alpha(\tilde{\tplan P}^{x_j})\\
        &\leq \begin{multlined}[t]
        \mass^\alpha(\tplan P^{x_i}-\tplan P^{x_i} \mres \{\gamma(\infty)\in B_r(x)\} + \nu(B_r(x))\delta_{\gamma_i}) + \mass^\alpha(\tplan Q)\\ + \sum_{j\neq i} \mass^\alpha(\tplan P^{x_j}-\tplan P^{x_j} \mres \{\gamma(\infty)\in B_r(x)\})
        \end{multlined}\\
        &\leq \sum_{j=1}^N \left(\mass^\alpha(\tplan P^{x_j}) - \alpha \int_{B_r(x)} z_{\tplan P^{x_j}} \dd\nu_j\right) + \alpha\nu(B_r(x)) z_{\tplan P^{x_i}}(x) + \mass^\alpha(\tplan Q)\\
        &= \mass^\alpha(\tplan P) -\alpha \int_{B_r(x)} z_{\tplan P} \dd\nu + \alpha \nu(B_r(x)) z_{\tplan P^{x_i}}(x) + \mass^\alpha(\tplan Q).
    \end{align*}
    Since $\mass^\alpha(\tplan Q) \leq 2 C_\BOT r^{1+d\alpha}$ it implies
    \[
        \forall r \in (0,\diam(\spt \nu)),\quad \fint_{B_r(x)} z_{\tplan P} \dd\nu \leq z_{\tplan P^{x_i}}(x) + \frac{2C_\BOT}{\alpha c_A} r^\beta \implies z_{\tplan P}(x) \leq z_{\tplan P^{x_i}}(x),
    \]
    hence $z_{\tplan P^{x_i}} = z_{\tplan P}$ on $\bas(\tplan P,x_i)$ for every $i\in \{1,\ldots, N\}$, i.e.\ \labelcref{item_landscape_well-defined} holds true. \qedhere
\end{proof}
\subsection{Applications of the landscape function}

We now generalize the properties of the single-source landscape function of \Cref{properties_landscape} to our setting.
First, we extend the notion of $\tplan P$-good curve when $\tplan P$ is a traffic plan with $N$ sources $\{x_1, \ldots, x_N\}$ such that the traffic plans $\tplan P^{x_i}$'s are disjoint: we say that a curve $\gamma$ is \emph{$\tplan P$-good} if it starts at some source $x_i$ and it is $\tplan P^{x_i}$-good.

\begin{proposition}\label{properties_landscape_extended}
    Assume $\alpha \in (1-1/d,1)$. Let $\nu$ be a compactly supported $d$-Ahlfors regular measure, $\tplan \mu$ be a $N$-point mass-optimal quantizer with respect to $\mathcal X \coloneqq \{x_i\}_{1\leq i \leq N}$ and $\tplan P \in \tplanset(\mu,\nu)$ be an optimal traffic plan with $\mass^\alpha(\tplan P) <\infty$, $\alpha \in [0,1)$. We consider a nonempty subset $\mathcal X' \subseteq \mathcal X$ and we set:
    \begin{align*}
        \mu' &\coloneqq \mu\mres \mathcal X'&
        \nu' &\coloneqq \nu \mres \bigcup_{s\in \mathcal X'} \bas(\tplan P,s)&
        \tplan P' &\coloneqq \sum_{s\in \mathcal X'} \tplan P^s.
    \end{align*}
    The landscape function $z_\tplan P : \spt \nu \to \R_+$ given by \Cref{thm:holderz} satisfies:
    \begin{enumerate}[(A)]
    \item\label{landcape_greater_distance} $z_{\tplan P}(x)\geq \norm{\nu}^{\alpha-1} d(x,\mathcal X')$ for every $x \in \spt \nu'$;
    \item\label{formula_optimal_cost_landscape_extended} the $\alpha$-distance writes as:
    \begin{equation*}
        \mathbf d^\alpha(\mathcal X', \nu') = \mathbf d^\alpha(\mu', \nu') = \mass^\alpha (\tplan P') = \int_{\R^d} z_{\tplan P}(x) \dd\nu'(x);
    \end{equation*}
    \item\label{first_variation_mass_landscape_extended} if $\tplan{\tilde P} \in \tplanset(\tilde \mu_N, \tilde \nu)$ is a traffic plan concentrated on $\tplan P'$-good curves, then
    \begin{equation*}
        \mass^\alpha(\tplan{\tilde P}) \leq \mass^\alpha(\tplan P') + \alpha \int_{\R^d} z_{\tplan P} \dd(\tilde\nu-\nu),
    \end{equation*}
    and the inequality is strict if for some $x_i \in \mathcal X'$, $\Theta_{\tplan{\tilde P}^{x_i}}-\Theta_{\tplan P^{x_i}}$ does not vanish $\hdm^1$-a.e. on $\Sigma_{\tplan P^{x_i}}$;
    \item\label{first_variation_distance_extended} in particular, $z_{\tplan P}$ is an \emph{upper first variation of the irrigation distance}, in the sense that for every $\tilde \nu \in \measspace^+(\R^d)$ which is supported on $\spt{\nu}$,
    \begin{equation*}
        \mathbf d^\alpha(\mathcal X',\tilde\nu) \leq \mathbf d^\alpha(\mathcal X',\nu') + \alpha \int_{\R^d} z_{\tplan P} \dd(\tilde\nu - \nu').
    \end{equation*}
\end{enumerate}
\end{proposition}
\begin{proof}[Sketch of proof]
The results follow rather directly from \Cref{properties_landscape}, \Cref{thm:holderz} and the optimality of the $\tplan P^{x_i}$'s stated in \Cref{quantizers_disjointness}. For \labelcref{landcape_greater_distance} it suffices to note that for every $x\in \spt \nu$ there exists $x_i\in \mathcal X$ such that $z_{\tplan P}(x)=z_{\tplan P^{x_i}}(x)$, and thus using \Cref{properties_landscape} \labelcref{formula_greater distance} we find $z_{\tplan P^{x_i}}(x)\geq d(x,x_i)\geq d(x,\mathcal X)$. For the remaining points, we merely apply the corresponding points from \Cref{properties_landscape} to the traffic plans $\tplan P^{x_i}, x_i\in \mathcal X'$, and combine it with the disjointness properties from \Cref{quantizers_disjointness}.
\end{proof}

\begin{corollary}\label{corr:basindisjoint}
    Under the assumptions of \Cref{thm:holderz} and using the same notations, the basins $\bas(\tplan P, x_i)$ are closed subsets of $\spt \nu$ which form a partition of $\nu$, in the sense that they are $\nu$-essentially disjoint and that their reunion is equal to $\spt \nu$.
\end{corollary}
Before tackling the proof, we stress that this result is not a consequence of \Cref{quantizers_disjointness}, since the fact that the $\nu_i$ ($1\leq i \leq N$) are mutually singular does not imply that the basins $\bas(\tplan P, x_i) = \spt \nu_i$ are $\nu$-essentially disjoint. For example on $[0,1]$ it is possible to build two mutually singular positive measures $\lambda_1,\lambda_2$ such that $\lambda_1 + \lambda_2 = \lambda \coloneqq \lbm^1 \mres [0,1]$ and $\spt{\lambda_1}= \spt{\lambda_2} = [0,1]$, thus their supports are not $\lambda$-essentially disjoint.
\begin{proof}[Proof of \Cref{corr:basindisjoint}]
    The basins are closed since by definition $\bas(\tplan P,x_i) = \spt \nu_i$ where $\nu_i = (e_\infty)_\sharp \tplan P^{x_i}$ for every $i\in \{1,\ldots, N\}$. Furthermore, $\nu = \sum_{i=1}^N \nu_i$ implies that $\spt \nu = \bigcup_{i=1}^N \spt \nu_i$.

    Let us show that for every $i\neq j$, $\nu_j(\bas(\tplan P, x_i)) = 0$, which implies the result, since for $k\neq l$ we can then write
    \begin{align*}
        \nu(\bas(\tplan P, x_k) \cap \bas(\tplan P, x_l)) &= \sum_{j=1}^N \nu_j(\bas(\tplan P, x_k) \cap \bas(\tplan P, x_l))\\
        &\leq \nu_k(\bas(\tplan P,x_l)) + \sum_{j\neq k} \nu_j(\bas(\tplan P,x_k)) = 0.
    \end{align*}
    Suppose by contradiction that  $\nu_j(\bas(\tplan P, x_i)) > 0$ for some $i\neq j$. 
    Take as competitor
    \[
        \tplan P' = \tplan P - \underbrace{\tplan P^{x_j} \mres \{\gamma : \gamma(\infty) \in \bas(\tplan P,x_i)\}}_{\tplan P_{ij}} + \tplan Q,
    \]
    where $\tplan Q \in \tplanset(\nu_j(\bas(\tplan P, x_i))\delta_{x_i},\nu_j \mres \bas(\tplan P,x_i))$ is chosen so that $\tplan Q$-a.e.\ curve $\gamma$ is a $\tplan P^{x_i}$-good curve, which is possible because $z_{\tplan P^{x_i}}$ is finite everywhere on $\bas(\tplan P,x_i)$ by \Cref{thm:holderz}. Since $\tplan P' \in \tplanset(\mu',\nu)$ where $\mu'$ is concentrated on the $x_i$'s, by mass-optimality of the quantizer $\mu$, we have:
\[\mass^\alpha(\tplan P) = \mathbf d^\alpha(\mu,\nu) \leq \mathbf d^\alpha(\mu',\nu) \leq \mass^\alpha(\tplan P').\]
Notice that  $\tplan P'$ is rectifiable, $\Sigma_{\tplan P_{ij}} \subseteq \Sigma_{\tplan P^{x_j}}$ and $\Sigma_{\tplan Q} \subseteq \Sigma_{\tplan P^{x_i}}$, since every $\tplan P^{x_i}$-good curve is $\hdm^1$-a.e. included in $\Sigma_{\tplan P^{x_i}}$. Besides, the $\tplan P^{x_k}$'s are disjoint by \Cref{quantizers_disjointness}, which implies thanks to \labelcref{disjointness_rectifiable} that the networks $\Sigma_{\tplan P^{x_k}}$'s are $\hdm^1$-essentially disjoint. Thus the traffic plans $(\tplan P')^{x_k}$'s are disjoint. We apply the upper first variation inequality \Cref{properties_landscape_extended} \labelcref{first_variation_mass_landscape_extended} to the variation given by replacing $\tplan P^{x_i}\mapsto (\tplan P')^{x_i} = \tplan P^{x_i} - \tplan P_{ij}$ and $\tplan P^{x_j}\mapsto (\tplan P')^{x_j} = \tplan P^{x_j} + \tplan Q$, and we get:
    \begin{align*}
        \mass^\alpha(\tplan P) &= \mathbf d^\alpha(\mu,\nu) \leq \mathbf d^\alpha(\mu',\nu) \leq \mass^\alpha(\tplan P')\\
        &= \mass^\alpha(\tplan P) + (\mass^\alpha(\tplan P^{x_i} - \tplan P_{ij})-\mass^\alpha(\tplan P^{x_i})) + (\mass^\alpha(\tplan P^{x_j} + \tplan Q)-\mass^\alpha(\tplan P^{x_j}))\\
        &< \mass^\alpha(\tplan P) - \alpha \int_{\bas (\tplan P,x_i) \cap \bas(\tplan P,x_j)} z_{\tplan P^{x_i}}\dd\nu_j + \alpha\int_{\bas (\tplan P,x_i) \cap \bas(\tplan P,x_j)} z_{\tplan P^{x_j}} \dd \nu_j\\
        &= \mass^\alpha(\tplan P) - \alpha \int_{\bas (\tplan P,x_i) \cap \bas(\tplan P,x_j)} z_{\tplan P} \dd\nu_j + \alpha \int_{\bas (\tplan P,x_i) \cap \bas(\tplan P,x_j)} z_{\tplan P} \dd \nu_j = \mass^\alpha(\tplan P),
    \end{align*}
    where we used \Cref{thm:holderz} in the last equality, and \Cref{properties_landscape_extended} and noticed that the inequality is strict since $\Theta_{\tplan P_{ij}}$ does not vanish $\hdm^1$-a.e. on $\Sigma_{\tplan P^{x_i}}$ (or similarly that $\Theta_{\tplan Q}$ does not vanish $\hdm^1$-a.e. on $\Sigma_{\tplan P^{x_j}}$). This is a contradiction.
\end{proof}
The measure of basins can be controlled from above and below for optimal plans associated with mass-optimal quantizers. \todopap[inline]{Ce serait mieux de pouvoir dire que $z_{\tplan P}$ est Hölder le long des branches, pour pouvoir dire directement que $\mu(\bas(\tplan P,x_i))^{\alpha-1} \abs{x-x_i}\leq z(x)-z(x_i) \leq C_H \abs{x-x_i}^\beta$, donc $\mu(\bas(\tplan P,x_i)) \geq C_H^{\alpha -1} \abs{x-x_i}^d$.}

\begin{lemma}\label{lem:volbas}
Under the assumptions of \Cref{thm:holderz} and with the same notations, for every source $x_i$ of $\tplan P$, we set
\[\delta(\tplan P,x_i) \coloneqq \max_{y\in \bas(\tplan P,x_i)}\abs{y-x_i}.\]
Then we have for every $x_i$:
\begin{gather}
 c_\bas \diam(\bas(\tplan P,x_i))^d \leq \nu(\bas(\tplan P,x_i))\leq C_A \diam(\bas(\tplan P,x_i))^d,\label{eq:volbas}\\
 \frac 12 \diam(\bas(\tplan P,x_i)) \leq \delta(\tplan P,x_i) \leq \frac{1}{2}\left(\frac{C_A}{c_\bas}\right)^{\frac 1d} \diam(\bas(\tplan P,x_i)),\label{eq:delta_basin}\\
  c_H \diam(\bas(\tplan P,x_i))^\beta \leq \sup_{\bas(\tplan P,x_i)}  z_{\tplan P} \leq C_H' \diam(\bas(\tplan P,x_i))^\beta\label{eq:oscillation_landscape},
\end{gather}
where $c_{\bas} \coloneqq 2^{-d}\left(C_H + \frac{2C_\BOT}{\alpha c_A^{1-\alpha}}\right)^{\frac 1{\alpha-1}}$, $c_H \coloneqq 2^{-\beta} C_A^{\alpha -1}$ and $C_H' \coloneqq 2 c_\bas^{-\frac 1d} C_A^{\frac \beta d}$.
\end{lemma}
\begin{proof}
First of all, since $\bas(\tplan P,x_i)$ is compact, we may consider a point $y \in \bas(\tplan P,x_i)$ such that $\abs{y-x_i} =\max_{y\in \bas(\tplan P,x_i)}\abs{y-x_i}$, and set $r\coloneqq \abs{y-x_i}$.

The upper bound in \labelcref{eq:volbas} follows from the upper $d$-Ahlfors regularity of $\mu$.

We now consider the lower bounds in \labelcref{eq:volbas} and \labelcref{eq:delta_basin}. Take a $\tplan P^{x_i}$-good curve $\gamma_i$ from $x_i$ to $y$. We build a competing traffic plan $\tilde{\tplan P}$ by removing $\tplan P \mres \{\gamma \in \curvspace^d : \gamma(\infty) \in B_{r}(y)\}$ then adding an optimal traffic plan $\tplan Q \in \tplanset(\nu(B_{r}(x_i)) \delta_{x_i}, \nu \mres B_{r}(y))$. 
By optimality of $\tplan P$ and mass-optimality of the source measure $\mu$, using the first variation inequality \Cref{properties_landscape_extended} \labelcref{first_variation_mass_landscape_extended}, and the branched transport upper estimate \labelcref{upper_estimate_alpha_mass}, we get
\begin{equation*}
    -\alpha \int_{B_r(y)} z_{\tplan P} \dd\nu  + C_\BOT (2r) \nu(B_r(y))^\alpha \geq 0 \implies \fint_{B_r(y)} z_{\tplan P} \dd\nu \leq \frac{2C_\BOT }{\alpha c_A^{1-\alpha}} r^\beta,
\end{equation*}
which yields by \Cref{thm:holderz}
\begin{equation}\label{volbas_lower_1}
    z_{\tplan P}(y)= z_{\tplan P}(y) - \fint_{B_r(y)} z_{\tplan P}\dd\nu + \fint_{B_r(y)} z_{\tplan P}\dd\nu \leq C_H r^\beta + \frac{2C_\BOT }{\alpha c_A^{1-\alpha}} r^\beta \eqqcolon C' r^\beta.
\end{equation}
Now recall the definition of landscape function in the single-source case:
\begin{equation}\label{volbas_lower_2}
\begin{aligned}
    C' r^\beta \geq z_{\tplan P}(y) = z_{\tplan P^{x_i}}(y) &= \int_{\gamma_i} \Theta_{\tplan P^{x_i}}(x)^{\alpha -1} \dd x\\
    &\geq \hdm^1(\gamma_i(\R_+)) \nu(\bas(\tplan P,x_i))^{\alpha-1} \geq r\nu(\bas(\tplan P,x_i))^{\alpha-1}.
    \end{aligned}
\end{equation}
As a consequence
\begin{equation}\label{volbas_lower_3}
    \nu(\bas(\tplan P,x_i))\geq (C' r^{\beta-1})^{\frac 1{\alpha-1}} = 2^{-d} {C'}^{\frac 1{\alpha-1}} (2r)^d = c_{\bas} (2r)^d
\end{equation}
where $c_{\bas} \coloneqq 2^{-d}\left(C_H + \frac{2C_\BOT}{\alpha c_A^{1-\alpha}}\right)^{\frac 1{\alpha-1}}$. Notice that $\diam(\bas(\tplan P,x_i)) \leq 2r$ by the triangle inequality, which is exactly the lower bound of \labelcref{eq:delta_basin}, and together with \labelcref{volbas_lower_3} yields the lower bound of \labelcref{eq:volbas}. The upper bound of \labelcref{eq:delta_basin} follows from \labelcref{volbas_lower_3} and the upper bound of \labelcref{eq:volbas}. As for \labelcref{eq:oscillation_landscape}, the lower bound comes from \labelcref{volbas_lower_2} and the upper Ahlfors regularity, while the upper bound comes from \labelcref{volbas_lower_1}, the upper bound of \labelcref{eq:delta_basin} and the triangle inequality, since for every $y'\in \bas(\tplan P,x_i)$
\begin{align*}
    z_{\tplan P}(y') \leq \abs{z_{\tplan P}(y')-z_{\tplan P}(y)} + z_{\tplan P}(y) &\leq C_H\abs{y-y'}^\beta + C' r^\beta\\
    &\leq 2 C' r^\beta  \\
    &= 2 2^{d(\alpha-1)} c_\bas^{\alpha-1} r^\beta\\
    &\leq 2^{\beta} c_\bas^{\alpha-1} \left(\frac{C_A}{2c_\bas}\right)^{\frac \beta d} \diam(\bas(\tplan P,x_i))^\beta\\
    &\leq C_H'\diam(\bas(\tplan P,x_i))^\beta,
\end{align*}
where $C_H' \coloneqq 2 c_\bas^{-\frac 1d} C_A^{\frac \beta d}$.
\end{proof}

\begin{remark}[Voronoi basins]\label{voronoi_basins}
In the case $\alpha =1$, if $\nu$ has a compact convex support $\Omega$ and $\tplan P$ is an optimal traffic plan between $\nu$ and a mass-optimal quantizer $\mu = \sum_{i=1}^N m_i \delta_{x_i}$ associated with the points $\{x_i\}_{1\leq i \leq N}$, then the basins will be exactly the \emph{Voronoi cells} $(\Omega \cap V_i)_{1\leq i \leq N}$ given by
\begin{equation}\label{voronoi_cells}
\forall i\in \{1,\ldots, N\}, \quad V_i \coloneqq \left\{ x : \abs{x-x_i} = \min_{1\leq j \leq N} \abs{x-x_j}\right\}.
\end{equation}
When $\alpha \in (0,1)$, the basins $(\bas(\tplan P,x_i))_{1\leq i \leq n}$ thus extend the notion of Voronoi cells to the case of branched optimal transport, which we may call \emph{(branched) Voronoi basins}. These Voronoi basins are also closed sets which form a partition of the given measure $\nu$, as stated in \Cref{corr:basindisjoint}, but they are much more complicated in several regards:
\begin{itemize}
    \item They need not be convex polyhedra, but are rather thought to exhibit fractal pairwise boundaries. The same conjecture was made concerning the boundaries of the so-called \emph{irrigation balls} of branched transport, i.e.\ the shapes of volume $1$ which are closest to a single point in the sense of branched transport, introduced in \cite{pegonFractalShapeOptimization2019}. In this work, each irrigation ball $\Omega$ is described as a sublevel set of the landscape function and its boundary $\partial \Omega$ as its highest level set, which allowed to show that the dimension of $\partial \Omega$ is at most $d-\beta$, which is the conjectured dimension. This conjecture is supported by \cite{devillanovaRemarksFractalStructure2019}, where it is shown that generic level sets of the landscape function must have (Minkowski) dimension $d-\beta$, if such a dimension exists. However, our Voronoi basins are \emph{not} generalizations of irrigation balls, since for us the target measure $\nu$ is fixed, thus the outer boundary (i.e.\ the boundary of the union of all basins) is prescribed to be the boundary of $\spt \nu$. The generalization of irrigation balls would be irrigation clusters arising from the problem of finding the set of volume $1$ which is closest to $N\geq 2$ arbitrary points. In our case, the interface between adjacent Voronoi basins $\bas(\tplan P,x_i)$ and $\bas(\tplan P,x_j)$ will not be described as a level set of the landscape function, but rather as $\bas(\tplan P,x_i) \cap \bas(\tplan P,x_j) = \{z_{\tplan P^{x_i}} = z_{\tplan P^{x_j}}\}$. However, it seems difficult to exploit this equality to get regularity results, since $z_{\tplan P^{x_i}}$ and $z_{\tplan P^{x_j}}$ both make sense only on the interface $\bas(\tplan P,x_i) \cap \bas(\tplan P,x_j)$ and not on any neighborhood of it.
    \item Classical Voronoi cells do not actually depend on the measure $\nu$ or its support but may be computed directly from the points $\{x_i\}_{1\leq i \leq N}$ through \labelcref{voronoi_cells}, taking the intersection with the support afterwards. On the contrary, there is a priori no reason for Voronoi basins to behave in the same way, and it is well possible that the Voronoi basins for $\nu$ and $\nu' \geq \nu$ are not nested. 
    \item Computing Voronoi basins is much more difficult, as the problem of optimizing the masses given the points does not admit an explicit solution in the form of \labelcref{voronoi_cells}.
\end{itemize}
\end{remark}

\section{Uniform properties of optimal quantizers and partitions}\label{sec:uniform}

In this section we investigate the uniform properties of optimal quantizers at the microscopic scale, i.e.\ at the scale of $N^{-1/d}$, when the measure $\nu$ that is quantized is $d$-Ahlfors regular. Roughly speaking, we are going to show that the atoms of a $N$-point optimal quantizer are distributed somewhat uniformly at this scale, being well-separated and leaving no big hole in the support of $\nu$, and that the basins are somewhat round, having inner and outer balls of comparable size. We also show uniformity bounds on the masses and energies associated with each atom.

\subsection{Delone constants for optimal quantizers}

In this section we prove \Cref{thm:delone}. Given a set $X \subseteq \R^d$ and $\mathcal X \subseteq \R^d$ a finite set of points, we define the \emph{covering radius} (also called \emph{mesh norm} or \emph{fill radius}) of $X$ by $\mathcal X$, as
\[\omega(X,\mathcal X) \coloneqq \sup_{x\in X} \min_{x' \in \mathcal X} d(x,x').\]
It is the smallest $r \geq 0$ such that the closed balls of radius $r$ with centers in $\mathcal X$ cover $X$. The \emph{separation distance} (corresponding to $1/2$ of the \emph{packing radius}) of $\mathcal X$ is defined by
\[\delta(\mathcal X) \coloneqq \min_{x,x'\in \mathcal X : x\neq x'} d(x,x').\]
A set $\mathcal X$ with finite covering radius and nonzero separation distance is called a \emph{Delone set} with respect to $X$, and $(\omega,\delta)$ its \emph{Delone constants}. Given a $d$-Ahlfors regular measure $\nu$, the following theorem shows that the atoms of optimal $N$-point quantizers are Delone sets with respect to $\spt \nu$, providing bounds comparable to $N^{-1/d}$ on its Delone constants.

Our proof of \Cref{thm:delone} is inspired from ideas of \cite{gruberOptimumQuantizationIts2004}, dealing with classical optimal transport costs. We stress that the situation in the branched optimal transport case is much more involved, since the ground cost is not explicit (it depends on all the trajectories and is part of the optimization defining $\mathbf d^\alpha$), and the shapes of basins are not known at all (they are thought to have fractal boundaries). Thus, we shall need to estimate
\begin{itemize}
    \item the cost for merging a \enquote{small} basin to a \enquote{neighbouring} basin ;
    \item the gain to remove part of a \enquote{large} basin.
\end{itemize}
The landscape function, its uniform Hölder regularity and its consequences established in \Cref{sec:landscape} will play a crucial role.

\begin{proof}[Proof of \Cref{thm:delone}]
Let $\tplan P_N \in \tplanset(\mu,\nu_N)$ be an optimal traffic plan. We proceed by proving successively the following.
\begin{enumerate}[(a)]
    \item\label{one_basin_small} \emph{At least one basin is not too large:} there is a constant $\tilde c_1$ (not depending on $N$) such that
    \begin{equation*}
        \forall N\in \N^\ast, \exists j\leq N, \quad \diam(\bas(\tplan P_N,x_j)) \leq \tilde c_1 N^{-\frac 1d}.
    \end{equation*}
    \item\label{one_basin_large} \emph{At least one basin is not too small:} there is a constant $\tilde c_2$ (not depending on $N$) such that
    \begin{equation*}
        \forall N\in \N^\ast, \exists j\leq N, \quad \diam(\bas(\tplan P_N,x_j)) \geq \tilde c_2 N^{-\frac 1d}.
    \end{equation*}
    \item\label{all_basins_small} \emph{All basins are small:} there is a constant $c_1$ (not depending on $N$) such that
    \begin{equation*}
        \forall N\in \N^\ast,\forall i \leq N, \quad \diam(\bas(\tplan P_N,x_i)) \leq c_1 N^{-\frac 1d}.
    \end{equation*}
    \item\label{all_basins_far} \emph{All atoms are far from each other:} there exists a constant $c_2$ such that
    \[\forall N\in \N^\ast, \forall (1\leq j\neq k\leq N), \quad d(x_j,x_k) \geq c_2 N^{-\frac 1d}.\]
\end{enumerate}
Notice that \labelcref{covering} follows from \labelcref{all_basins_small}, since the basins form a covering of $\spt \nu$ by \Cref{corr:basindisjoint}, while \labelcref{separation} is merely a rephrasing of \labelcref{all_basins_far}.

\paragraph{Proof of \labelcref{one_basin_small} and \labelcref{one_basin_large}.}
First note that by \Cref{corr:basindisjoint}, the basins form a partition of $\nu$, thus
\begin{equation*}\label{basins_total_volume}
\sum_{j=1}^N \nu(\bas(\tplan P_N,x_j)) = \norm{\nu},
\end{equation*}
thus there exists an index $j \in \{1,\ldots, N\}$ for which
\[
    \nu(\bas(\tplan P_N,x_j)) \leq \frac{\norm{\nu}}{N} \leq \frac{C_A}N \diam(\spt \nu)^d,
\]
and by \Cref{lem:volbas}, this implies that
\[\diam(\bas(\tplan P_N,x_j)) \leq (\norm{\nu} (c_\bas N)^{-1})^{\frac 1d}  = \tilde c_1 N^{-\frac 1d} \quad\text{where}\quad \tilde c_1 \coloneqq (C_A/c_\bas)^{\frac 1d} \diam(\spt\nu).\]
Similarly, there exists an index $i \in \{1,\ldots, N\}$ such that
\[
    C_A\diam(\bas(\tplan P_N,x_i))^d \geq \nu(\bas(\tplan P_N,x_i)) \geq \frac{\norm{\nu}}N \geq \frac{c_A}{N} \diam(\spt \nu)^d,
\]
which implies that
\[\diam(\bas(\tplan P_N,x_i)) \geq \tilde c_2 N^{-\frac 1d} \quad\text{where}\quad \tilde c_2 \coloneqq ( c_A/C_A)^{\frac 1d} \diam(\spt \nu).\]

\paragraph{Proof of \labelcref{all_basins_small}.} Applying \labelcref{one_basin_small} we take $j \leq N$ such that
\begin{equation}\label{basin_i_small}
    \diam(\bas(\tplan P_N,x_j)) \leq \tilde c_1 N^{-\frac 1d}.
\end{equation}
Suppose that for some $t>1$ there exists $i \leq N$ such that
\begin{equation*}
    \diam(\bas(\tplan P_N,x_i)) \geq t \tilde c_1 N^{-\frac 1d}.
\end{equation*}
We are going to show a contradiction when $t$ is too large (not depending on $N$). For this, let us build a better competitor than $\mu_N$. We shall add an extra point $q$ of $\bas(\tplan P_N,x_i)$ to irrigate a \enquote{costly} ball around $q$, then remove the point $x_j$ (in order to leave the number of points equal to $N$) and irrigate the former basin $\bas(\tplan P_N,x_j)$ from a neighbouring basin $\bas(\tplan P_N,x_k)$.

Let us find a \enquote{costly} ball. Consider
\[
    q \in \argmax_{\bas(\tplan P_N,x_i)} z_{\tplan P_N},
\]
which exists because $z_{\tplan P_N}$ is Hölder-continuous thanks to \Cref{thm:holderz} and basins are compact sets thanks to \Cref{corr:basindisjoint}. We consider the ball $B_{\eps t \tilde c_1 N^{-\frac 1d}}(q)$ for some small $\eps \in (0,1)$ to be fixed later.

Now, we want to remove the point $x_j$ from the quantizer $\mu_N$ and to irrigate the basin $\bas(\tplan P_N,x_j)$ from another basin that is not too far, in order to control the extra cost. By \labelcref{basin_i_small} and Ahlfors-regularity of $\nu$, for $s > 1$ we have
\[
    \nu(B_{s\tilde c_1 N^{-\frac 1d}}(x_j) \setminus \bas(\tplan P_N,x_j)) \geq c_A (s\tilde c_1)^d N^{-1} - C_A \tilde c_1^d N^{-1}= (c_A s^d-C_A) \tilde c_1^d N^{-1}.
\]
This is strictly positive if we take for example $s \coloneqq (2 C_A/c_A)^{\frac 1d}$, in which case there exists a point $p$ such that
\begin{equation}\label{exists_close_point}
    p\in\bas(\tplan P_N,x_k) \cap \left(B_{s \tilde c_1 N^{-1/d}}(x_j) \setminus \bas(\tplan P_N,x_j)\right) \quad\text{for some}\quad k\neq j,
\end{equation}
because the basins form of covering of $\spt \nu.$

We are now ready to build our competitor $\tplan P_N^*$, modifying $\tplan P_N$ according to the following sketch; the addition of curves (which increases the $\alpha$-mass) are labeled by (A), while the removal of curves (which decreases the $\alpha$-mass) are labeled by (R).
\begin{enumerate}
    \item[(R$_1$)] Remove all curves starting at $x_j$. 
    \item[(R$_2$)] Remove all curves ending in $B_{\eps t \tilde c_1 N^{-1/d}}(q)$.
    \item[(A$_1$)] Re-irrigate $\bas(\tplan P_N,x_j)$ by
    \begin{itemize}
    \item bringing a mass $m_j$ from $x_k$ to $p$ following a $\tplan P_N^{x_k}$-good curve $\gamma$,
    \item then concatenating a traffic plan $\tplan Q^1 \in \tplanset(m_j \delta_p,\nu \mres \bas(\tplan P_N,x_j))$ which is optimal.
    \end{itemize}
    \item[(A$_2$)] Add an optimal traffic plan $\tplan Q^2$ from $m \delta_q$ to $\nu \mres (B_{\eps t \tilde c_1 N^{-\frac 1d}}(q) \setminus \bas(\tplan P_N,x_j))$, where $m$ is the mass of the latter and $\eps > 0$ is a small number to be chosen later (independently from $N$ and $t$).
\end{enumerate}
We shall first estimate the energy gain and loss associated with modifications of the existing network, by applying the first variation formulas of \Cref{properties_landscape_extended} starting from a mass-optimal quantizer, then estimate the energy loss associated with the concatenation of new networks, by subadditivity of the $\alpha$-mass along concatenations, as stated in \labelcref{concat_alpha_mass}. Since A$_1$ involves the two types of modifications, its treatment will be split in two.

We start by doing the modifications along the existing network, corresponding to (R$_1$), (R$_2$) and the first part of (A$_1$). Setting $\Gamma_{x_j} \coloneqq \{\gamma : \gamma(0) = x_j\}$ and $\Gamma_q \coloneqq \{\gamma : \gamma(+\infty) \in B_{\eps t \tilde c_1 N^{-\frac 1d}}(q)\}$, we define
\[\tplan P_N'\coloneqq \tplan P_N - \tplan P_N \mres \Gamma_q - \tplan P_N \mres (\Gamma_{x_j} \setminus \Gamma_q) + m_j \delta_\gamma.\]
Secondly, we add the new curves and pieces of curves corresponding to the second part of (A$_1$) and (A$_2$). We set $\nu' \coloneqq (e_\infty)_\sharp \tplan P_N'$, and $\tilde{\tplan Q}^1 \coloneqq \tplan Q^1 + \iota_\sharp (\nu'-m_j \delta_p)$ where $\iota : \R^d \to \curvspace^d$ denotes the canonical injection which sends a point $x$ to the constant curve $\gamma_x \equiv x$. We define our competitor $\tplan P_N^*$ by
\[\tplan P_N^* \coloneqq \tplan P_N'' + \tplan Q^2 \quad\text{where}\quad \tplan P_N'' \in \tplan P_N' : \tilde{\tplan Q}^1.\]
We estimate the gain and cost of these operations. First of all, using \labelcref{upper_estimate_alpha_mass} we have
\begin{equation}\label{uniformity_covering_upper_bound_0}
\begin{aligned}
    \mass^\alpha(\tplan P_N^*) &\leq \mass^\alpha(\tplan P_N'') +\mass^\alpha(\tplan Q^2) \leq \mass^\alpha(\tplan P_N') +\mass^\alpha(\tplan Q^2)+\mass^\alpha(\tplan Q^1)\\
    &\leq \mass^\alpha(\tplan P_N') + C_\BOT C_A^\alpha (\eps t \tilde c_1 N^{-1/d})^{1+d\alpha} + C_\BOT C_A^\alpha (\tilde c_1 (1+s)N^{-1/d})^{1+d\alpha}\\
    &= \mass^\alpha(\tplan P_N') + C_1 N^{-\left(\alpha +\frac 1d\right)}(1 + (\eps t)^{1+d\alpha}),
\end{aligned}
\end{equation}
for some constant $C_1$ which does not depend on $N$.
The $\alpha$-mass of $\tplan P_N'$ may then be estimated through the first variation formula \Cref{properties_landscape_extended} \labelcref{first_variation_mass_landscape_extended}:
\begin{equation}\label{uniformity_covering_upper_bound_1}
\begin{aligned}
    \qquad\ & \mass^\alpha(\tplan P_N')\\
    \leq\  & \mass^\alpha(\tplan P_N) - \alpha\int_{B_{\eps t \tilde c_1 N^{-1/d}}(q)} z_{\tplan P_N} \dd \nu - \alpha\int_{\R^d} z_{\tplan P_N} \dd (e_\infty)_\sharp (\tplan P_N \mres (\Gamma_{x_j}\setminus \Gamma_q)) + m_j z_{\tplan P_N}(p)\\
    \leq\ & \mass^\alpha(\tplan P_N) - \alpha\int_{B_{\eps t \tilde c_1 N^{-1/d}}(q)} z_{\tplan P_N} \dd \nu  + m_j z_{\tplan P_N}(p).
\end{aligned}
\end{equation}
Let us estimate $z_{\tplan P_N}(p)$ from above and $z_{\tplan P_N}$ from below on $B_{\eps t \tilde c_1 N^{-1/d}}(q)$. By \Cref{lem:volbas}, we know that
\[
z_{\tplan P_N}(q) \geq c_H \diam(\bas(\tplan P_N,x_i))^\beta \geq c_H (t\tilde c_1)^\beta N^{-\beta/d},
\]
thus for every $y\in B_{\eps t \tilde c_1 N^{-1/d}}(q)$
\begin{equation}\label{uniformity_covering_lower_bound_z}
\begin{aligned}
    z_{\tplan P_N}(y) &\geq c_H(t\tilde c_1)^\beta N^{-\beta/d} -C_H \abs{y-q}^\beta \\
    &\geq (c_H-\eps^\beta C_H) (t\tilde c_1)^\beta N^{-\beta/d} \geq (c_H/2) (t\tilde c_1)^\beta N^{-\beta/d},
\end{aligned}
\end{equation}
provided we have chosen $\eps^\beta \leq (c_H/2C_H)$. Besides, by \Cref{lem:volbas} again and \labelcref{exists_close_point},
\begin{equation}\label{uniformity_covering_upper_bound_z}\begin{aligned}
    z_{\tplan P_N}(p) &\leq  \sup_{y\in\bas(\tplan P_N,x_j)} \abs{z_{\tplan P_N}(y) - z_{\tplan P_N}(p)} + \sup_{y\in\bas(\tplan P_N,x_j)} z_{\tplan P_N}(y)\\
    &\leq C_H((s+1)\tilde c_1 N^{-1/d})^\beta + C_H' (\tilde c_1 N^{-1/d})^\beta\\
    &= C N^{-\beta/d}
    \end{aligned}
\end{equation}
    for some $C$ which does not depend on $N$. Reporting  \labelcref{uniformity_covering_lower_bound_z} and \labelcref{uniformity_covering_upper_bound_z} in \labelcref{uniformity_covering_upper_bound_1} yields
    \begin{align*}
        \mass^\alpha(\tplan P_N') &\leq \mass^\alpha(\tplan P_N) - \alpha(c_H/2) (t\tilde c_1)^\beta N^{-\beta/d} c_A (\eps t \tilde c_1 N^{-1/d})^d + \alpha C N^{-\beta/d} C_A(\tilde c_1 N^{-1/d})^d\\
        &\leq \mass^\alpha(\tplan P_N) - N^{-\left(\alpha + \frac 1d\right)}(C_2 \eps^d t^{1+d\alpha} -C_3),
    \end{align*}
    for some constant $C_2,C_3 > 0$ which do not depend on $N$.

Injecting this into \labelcref{uniformity_covering_upper_bound_0}, we get
\begin{align*}
    \mass^\alpha(\tplan P_N^*) \leq \mass^\alpha(\tplan P_N) +  N^{-\left(\alpha +\frac 1d\right)}(C_1+C_3 + C_1(\eps t)^{1+d\alpha}-C_2 \eps^d t^{1+d\alpha}).
\end{align*}
Notice that because $1+d\alpha - d = \beta > 0$, it is possible to choose $\eps$ small enough, independently from $N$ and $t$ (e.g.\footnote{Recall that we had the condition $\eps^\beta \leq (c_H/2C_H)$ for \labelcref{uniformity_covering_lower_bound_z}.} $\eps^\beta = (C_2/2C_1) \wedge (c_H/2C_H)$) so that
\[\mass^\alpha(\tplan P_N^*) \leq \mass^\alpha(\tplan P_N) +  N^{-\left(\alpha +\frac 1d\right)}(C_1+C_3 - (C_2/2) \eps^d t^{1+d\alpha}).\]
Now, if $t$ is too large, depending on the constants $C_1,C_2,C_3,\eps$ which we stress do not depend on $N$, it leads to $\mass^\alpha(\tplan P_N^*) < \mass^\alpha(\tplan P_N)$, which contradicts the optimality of $\mu_N$, because by construction the target measure of $\tplan P_N^*$ is $\nu$, and its source measure is an $N$-point quantizer. As a conclusion, \labelcref{all_basins_small} holds true for $c_1=t_1\tilde c_1$ where $t_1 = (2(C_1+C_3)/\varepsilon^d C_2)^{1/(1+d\alpha)}$.

\paragraph{Proof of \labelcref{all_basins_far}.} Take $t > 0$ and suppose that there are two atoms $x_j,x_k$ of $\mu_N$, with $j\neq k$, such that $d(x_j,x_k) \leq t N^{-1/d}$. We are going to show a lower bound on $t> 0$ (not depending on $N$). Applying \labelcref{one_basin_large}, choose an index $i \leq N$ such that
\[\diam(\bas(\tplan P_N,x_i)) \geq \tilde c_2 N^{-1/d}.\]
Up to interchanging $j$ with $k$, we may assume that $i\neq j$ ($k$ may be equal to $i$, but it will not matter). The strategy is very similar as what we did above: we shall add an extra point $q \in \bas(\tplan P_N,x_i)$ to irrigate a \enquote{costly} ball around $q$, while removing the point $x_j$ from the quantizer and irrigating the basin $\bas(\tplan P_N,x_j)$ from the close point $x_k$.

Consider a point
\[
    q \in \argmax_{\bas(\tplan P_N,x_i)} z_{\tplan P_N},
\]
which exists because $z_{\tplan P_N}$ is Hölder-continuous thanks to \Cref{thm:holderz} and basins are compact sets thanks to \Cref{corr:basindisjoint}. Take some small $\eps \in (0,1)$ that we shall fix later. We build a better competitor thanks to the following construction.
\begin{enumerate}
\item[(R$_1$)] Remove all curves ending in $B_{\eps \tilde c_2N^{-1/d}}(q)$, of total mass $m$.
\item[(A$_1$)] Add an optimal traffic plan $\tplan Q^1 \in \tplanset(m\delta_q, \nu \mres B_{\eps \tilde c_2 N^{-1/d}}(q))$ to irrigate $\nu \mres B_{\eps \tilde c_2N^{-1/d}}(q)$ again.
\item[(A$_2$)] Irrigate $\bas(\tplan P_N,x_j)$ from the point $x_k$ instead of $x_j$ by concatenating a (unit-speed parameterization of) the segment $[x_k,x_j]$ to all curves starting at $x_j$.
\end{enumerate}
The removal (R$_1$) produces the new traffic plan
\[\tplan P'_N \coloneqq \tplan P_N - \tplan P_N \mres \Gamma_q\quad\text{where}\quad\Gamma_q \coloneqq\{\gamma : \gamma(\infty) \in B_{\eps \tilde c_2N^{-1/d}}(q)\}.\]
We show by \Cref{lem:volbas}, as in \labelcref{uniformity_covering_lower_bound_z} above, that for every $y\in B_{\eps \tilde c_2 N^{-1/d}}(q)$
    \begin{equation*}\label{uniformity_separation_lower_bound_z}
        z_{\tplan P_N}(y) \geq (c_H/2) (\tilde c_2)^\beta N^{-\beta/d},
    \end{equation*}
    provided we have chosen $\eps^\beta \leq (c_H/2C_H)$. Thus the cost gain can be estimated by using the first variation formula \Cref{properties_landscape_extended} \labelcref{first_variation_mass_landscape_extended}:
\begin{equation}\label{uniformity_separation_upper_bound_0}
\begin{aligned}
    \mass^\alpha(\tplan P_N') &\leq \mass^\alpha(\tplan P_N) - \alpha \int_{B_{\eps \tilde c_2N^{-1/d}}(q)} z_{\tplan P_N} \dd\nu\\
    &\leq \mass^\alpha(\tplan P_N) - \alpha C_A \tilde c_2^\beta (\eps \tilde c_2N^{-1/d})^d (c_H/2) N^{-\beta/d} \\
    &\leq \mass^\alpha(\tplan P_N) - C_4 \eps^d N^{-\left(\alpha + \frac 1d\right)},
\end{aligned}
\end{equation}
for some constant $C_4 >0$ which does not depend on $N$. For the addition of the (pieces of) curves (A$_1$) and (A$_2$), we denote by $\gamma_{k,j}$ the unit-speed parameterized segment from $x_k$ to $x_j$, $\tplan P_{N,j}' \coloneqq \tplan P_N' \mres\{\gamma : \gamma(0) = x_j\}$, $m_j' \coloneqq \norm{\tplan P_{N,j}'}$ and we set $\tplan Q^2 \coloneqq m'_j \delta_{\gamma_{k,j}} + (\iota \circ e_0)_\sharp(\tplan P_N'-\tplan P_{N,j}')$. We define our competitor $\tplan P_N^*$ by
\[\tplan P_N^* \coloneqq \tplan P_N'' + \tplan Q^1 \quad\text{where}\quad \tplan P_N'' \in \tplan Q^2 : \tplan P_N'.\]
From \labelcref{all_basins_small} we know that
\[m_j' \leq m_j \leq C_A(\diam (\bas(\tplan P_N,x_j)))^d \leq C_A c_1^d N^{-1},\]
and we compute, using \labelcref{uniformity_separation_upper_bound_0}:
\begin{equation*}\label{uniformity_separation_upper_bound_1}
\begin{aligned}
    \mass^\alpha(\tplan P_N^*) &\leq \mass^\alpha(\tplan P_N'') + \mass^\alpha(\tplan Q^1) \leq \mass^\alpha(\tplan P_N') + \mass^\alpha(\tplan Q^1) + \mass^\alpha(\tplan Q^2) \\
    &\leq \mass^\alpha(\tplan P_N') + C_\BOT C_A^\alpha (\eps \tilde c_2 N^{-1/d})^{1+d\alpha} + (m_j')^\alpha (tN^{-1/d})\\
    &\leq \mass^\alpha(\tplan P_N) + N^{-\left(\alpha +\frac 1d\right)}(-C_4 \eps^d + C_5 \eps^{1+d\alpha} + C_6 t),
\end{aligned}
\end{equation*}
for some $C_5,C_6 > 0$ not depending on $N$. Taking $\eps > 0$ such that $\eps^\beta \leq c_H/2C_H$ and $\eps^\beta \leq C_4/2C_5$ (e.g. take $\eps^\beta$ to be the minimum of the two), we get $C_5 \eps^{1+d\alpha} \leq (C_4/2)\eps^d$ and thus
\[\mass^\alpha(\tplan P_N^*)\leq \mass^\alpha(\tplan P_N) + N^{-\left(\alpha +\frac 1d\right)}(C_6t- (C_4/2)\eps^d),\]
which leads to a contradiction if $t$ is too small (independently from $N$). Hence $t$ is lower bounded by some constant $c_2 =C_4/(2 C_6)\varepsilon^d> 0$ and \labelcref{all_basins_far} holds true with this value of $c_2$.

Finally, we explain why the constants $c_1,c_2$ are of the form $\diam(\spt\nu)$ multiplied by a constant depending only on $(\alpha,d,c_A,C_A)$. For this part only, we abbreviate $\Delta\coloneqq\diam(\spt(\nu))$ and we write $a\approx b$ if there exists a constant $c > 0$ which depends only on $\alpha,d,c_A,C_A$ but not on $\Delta$ such that $a = c b$. First, note that $\tilde c_1 \approx \Delta$ and $\tilde c_2 \approx \Delta$. Using this, we inspect the proofs of \labelcref{all_basins_small} and \labelcref{all_basins_far} to deduce the dependence of $c_1, c_2$ on $\Delta$. For \labelcref{all_basins_small}, we find that following the proof directly, one can choose $C \approx \Delta^\beta, C_1 \approx \Delta^{1+d\alpha}, C_2 \approx \Delta^{1+d\alpha}$ and $C_3 \approx\Delta^{1+d\alpha}$. Then $\varepsilon\approx 1$ and $t_1 \approx (C_1+C_3)/C_2 \approx 1$ thus $c_1=t_1\tilde c_1\approx\Delta$, as desired. For the case of $c_2$, consider the proof of \labelcref{all_basins_far}. Here we can choose $C_4\approx\Delta^{1+\alpha d}, C_5\approx\Delta^{1+\alpha d}$ and $C_6\approx\Delta^{\alpha d}$. As a consequence $\varepsilon\approx 1$ and $c_2\approx C_4/C_6\approx\Delta$.
\end{proof}

\subsection{Inner and outer ball property of the basins}

\begin{proposition}\label{prp:quasi-uniformity_basins}
    Assume $\alpha \in(1-1/d,1)$. Let $\nu \in \measspace_c^+(\R^d)$ be a $d$-Ahlfors regular measure with constants $0 < c_A\leq C_A$. Let $\mu_N = \sum_{i=1}^N m_i \delta_{x_i}$ be a $N$-point optimal quantizer of $\nu$ and $\tplan P_N$ an optimal traffic plan in $\tplanset(\mu_N,\nu)$. There are constants $c,C >0$ depending only on $(\alpha,d,c_A,C_A,\diam(\spt \nu))$ such that for all $i \in \{1,\ldots, N\}$,
    \begin{equation}\label{outer_ball_property}
        \bas(\tplan P_N,x_i) \subseteq B_{C N^{-1/d}}(x_i),
    \end{equation}
    and
    \begin{equation}\label{inner_ball_property}
        B_{c N^{-1/d}}(x_i) \subseteq \R^d \setminus \bigcup_{j\neq i} \bas(\tplan P_N,x_j).
    \end{equation}
\end{proposition}
\begin{remark}
    In particular, if $\nu = \lbm^d \mres \Omega$ for some open bounded set $\Omega$ with Lipschitz boundary, then for every source $x_i$ such that $d(x_i,\partial \Omega) > c N^{-1/d}$, \labelcref{outer_ball_property} and \labelcref{inner_ball_property} rewrite as
    \[B_{c N^{-1/d}}(x_i) \subseteq \bas(\tplan P_N,x_i) \subseteq B_{C N^{-1/d}}(x_i).\]
    Besides, this translates as uniform inner and outer ball properties of optimal partitions (solutions to \labelcref{pb:opt_partition_sets}). Indeed, by the equivalence between the optimal quantization and optimal partition problems stated in \Cref{thm:opoleq}, solutions  $(\Omega_i)_{1\leq i\leq N}$ to \labelcref{pb:opt_partition_sets} are actually $\lbm^d$-equivalent to the basins $(\bas(\tplan P_N,x_i))_{1\leq i\leq N}$ associated with some traffic plan $\tplan P_N$ with sources $\{x_i\}_{1\leq i\leq N}$.

    Finally, we remark that the number of points such that $d(x_i,\partial \Omega) > c N^{-1/d}$ is $N + O(N^{1-1/d})$ because the other points each convey a mass $\approx N^{-1}$ by \Cref{uniformity_masses_energies}, and their basins are included in $\{d(\cdot,\partial \Omega) < C' N^{-1/d}\}$ for some $C'$, a set of volume $\approx N^{-1/d}$. However, without extra assumptions on $\Omega$, some points $x_i$ may very well belong to $\R^d \setminus \bar\Omega$. This is ruled out for example when $\Omega$ is convex, but then it is not clear whether $x_i \in \Omega$ for all $x_i$'s.
\end{remark}

\begin{proof}[Proof of \Cref{prp:quasi-uniformity_basins}]
    The outer ball property \labelcref{outer_ball_property} holds if we take $C$ to be equal to $C \coloneqq c_2(C_A/(2c_\bas))^{1/d}$, by \labelcref{eq:delta_basin} in \Cref{lem:volbas} and \labelcref{all_basins_small} in the proof of \Cref{thm:delone}. For the inner ball property \labelcref{inner_ball_property}, assume that for some $i\neq j$, $d(x_i,\bas(\tplan P_N,x_j)) \leq \eps N^{-1/d}$. We shall find a lower bound on $\eps$ that depends only on the constants $(\alpha,d,c_A,C_A,\diam(\spt\nu))$. If $\eps \leq c_1/2$ where $c_1$ is the separation constant in \labelcref{separation}, then taking a point $x \in B(x_i,\eps N^{-1/d}) \cap \bas(\tplan P_N,x_j)$, we have
    \[d(x,x_j) \geq d(x_i,x_j) -d(x_i,x) \geq (c_1/2) N^{-1/d},\]
    thus taking $\gamma_j$ a $\tplan P_N$-good curve from $x_j$ to $x$, its length is greater than $(c_1/2) N^{-1/d}$, hence
    \[z_{\tplan P_N}(x) = \int_{\gamma_j} \Theta_{\tplan P_N^{x_j}}^{\alpha -1}\dd\hdm^1 \geq (c_1/2) N^{-1/d} \nu(\bas(\tplan P_N,x_j))^{\alpha -1}\geq \frac12 {c_1 (C_A c_2^d)^{\alpha -1}} N^{-\beta/d}.\]
    As a consequence, setting $c' \coloneqq \frac12 {c_1 (C_A c_2^d)^{\alpha -1}}$ and $c'' \coloneqq (c'/2C_H)^{1/\beta}$, by \Cref{thm:holderz} for every $y\in B(x, c'' N^{-1/d}) \cap \spt \nu$,
    \[z_{\tplan P_N}(y) \geq z_{\tplan P_N}(x) - C_H\abs{y-x}^\beta \geq c'/2 N^{-\beta/d}.\]
    If $\eps \leq c''$ we build a competitor $\tplan P_N'$ by removing from $\tplan P_N$ all curves going to $B(x, \eps N^{-1/d})$ and adding an optimal traffic plan $Q \in \tplanset(m_\eps\delta_{x_i},\nu \mres B(x, \eps N^{-1/d})$ where $m_\eps \coloneqq \nu(B(x, \eps N^{-1/d}))$. Using the first variation formula \Cref{properties_landscape_extended} \labelcref{first_variation_mass_landscape_extended} and the subadditivity of the $\alpha$-mass, we get by optimality of $\mu_N$
    \begin{align*}
        \mass^\alpha(\tplan P_N) \leq \mass^\alpha(\tplan P_N') &\leq \mass^\alpha(\tplan P_N) - \alpha \int_{B_{\eps N^{-1/d}}(x)} z_{\tplan P_N} \dd\nu + C_\BOT (2\eps N^{-1/d}) (C_A \eps^d N^{-1})^\alpha\\
        &\leq \mass^\alpha(\tplan P_N) - (\alpha c_A c'/2)  \eps^d N^{-\left(\alpha + \frac 1d\right)} + 2C_\BOT C_A^\alpha \eps^{1+d\alpha} N^{-\left(\alpha + \frac 1d\right)}.
    \end{align*}
    We get a contradiction when $\eps$ is smaller than some constant $c'''>0$ (depending only on $\alpha$, $d$, $c_A$, $c'$, $C_\BOT$, $C_A$) because $d < 1 + d\alpha$. During the reasoning we made, recall that we assumed $\eps \leq c_2/2$ then $\eps \leq c''$, thus we must have:
    \[\eps \geq c \coloneqq \min\{c_2/2,c'',c'''\},\]
    and \labelcref{inner_ball_property} holds with this constant $c > 0$.
\end{proof}

\subsection{Uniformity bounds for masses and energies}

\begin{proposition}\label{uniformity_masses_energies}
    Assume $\alpha \in(1-1/d,1)$. Let $\nu \in \measspace_c^+(\R^d)$ be a $d$-Ahlfors regular measure with constants $c_A,C_A >0$. Let $\mu_N = \sum_{i=1}^N m_i \delta_{x_i}$ be an $N$-point optimal quantizer and $\tplan P_N \in \tplanset(\mu_N,\nu)$ be an optimal traffic plan. There are constants $c,C >0$ depending only on $(\alpha,d,c_A,C_A,\diam(\spt \nu))$ such that for all $i \in \{1,\ldots, N\}$,
    \begin{equation*}
        c N^{-1} \leq m_i \leq C N^{-1}
    \end{equation*}
    and
    \begin{equation*}
        c N^{-(\alpha + 1/d)} \leq \mathbf d^\alpha(m_i\delta_{x_i}, \nu \mres \bas(\tplan P_N,x_i)) \leq C N^{-(\alpha + 1/d)}.
    \end{equation*}
\end{proposition}

\begin{proof}
The upper bounds come from the fact that, by \Cref{thm:delone}, for every $i\in\{1,\ldots,N\}$, $\bas(\tplan P_N,x_i)$ has diameter less than $c_1 N^{-1/d}$, thus 
\[m_i = \nu(\bas(\tplan P_N,x_i)) \leq C_A c_1^d N^{-1},\]
and thus from the usual estimate of the branched transport cost we get \[\mathbf d^\alpha(m_i\delta_{x_i}, \nu \mres \bas(\tplan P_N,x_i)) \leq C_\BOT(c_1 N^{-1/d})(C_A c_1^d N^{-1})^\alpha.\]

The lower bounds result from \Cref{prp:quasi-uniformity_basins}, which implies that ($c,C$ being the inner and outer ball constants)
\[\nu \mres \bas(\tplan P_N, x_i) \geq \nu \mres B_{cN^{-1/d}}(x_i),\]
and thus
\[\nu(\bas(\tplan P_N, x_i)) \geq c_A c^d N^{-1}.\]
Since $\nu$ is $d$-Ahlfors regular on $\R^d$, by the Lebesgue--Besicovitch differentiation theorem \cite[Theorem~2.22]{ambrosioFunctionsBoundedVariation2000} it may be written as $\nu = f \lbm^d \mres X$ on some Borel set $X \subseteq \spt \nu$, where the Radon--Nikodým derivative $f = \dd\nu/\dd\lbm^d$ satisfies $c_A/\omega_d \leq f \leq C_A/\omega_d$. Therefore, $\omega_d C_A^{-1}\nu \mres \bas(\tplan P_N, x_i)$ is a measure which is absolutely continuous with respect to Lebesgue, with density in $[0,1]$ and total mass greater than  $m \coloneqq \omega_dc_A/C_A c^d N^{-1}$, thus
\[\mathbf d^\alpha(m_i\delta_{x_i}, \nu \mres \bas(\tplan P_N,x_i)) \geq e_{\alpha,d} m^{\alpha +  \frac 1d},\]
where $e_{\alpha,d} > 0$ is the constant from the optimal shape problem studied in \cite{pegonFractalShapeOptimization2019}, whose definition is given in \labelcref{optimal_shape_constant}.
\end{proof}

\smallskip

\noindent{\textbf{Acknowledgments.}} P.P. acknowledges the academic leave provided by Inria Paris for the year 2022-2023, and the \enquote{Young Researcher's Project} funding provided by the Scientific Council of Université Paris--Dauphine. M.P. acknowledges support from the Chilean Fondecyt Regular grant 1210426 entitled "Rigidity, stability and uniformity for large point configurations". We thank the anonymous referees for their careful reading of our manuscript, which helped us improve it.

\sloppy
\printbibliography

\end{document}